\documentclass[11pt]{amsart}
\usepackage{amsmath}
\usepackage{cases}
\usepackage{amssymb}
\usepackage{amsfonts}
\usepackage{amssymb,amsfonts}
\allowdisplaybreaks

\begin{document}
\theoremstyle{plain}
\newtheorem{thm}{Theorem}[section]
\newtheorem{theorem}[thm]{Theorem}
\newtheorem{lemma}[thm]{Lemma}
\newtheorem{corollary}[thm]{Corollary}
\newtheorem{corollary*}[thm]{Corollary*}
\newtheorem{proposition}[thm]{Proposition}
\newtheorem{proposition*}[thm]{Proposition*}
\newtheorem{conjecture}[thm]{Conjecture}
\theoremstyle{definition}
\newtheorem{construction}{Construction}
\newtheorem{notations}[thm]{Notations}
\newtheorem{question}[thm]{Question}
\newtheorem{problem}[thm]{Problem}
\newtheorem{remark}[thm]{Remark}
\newtheorem{remarks}[thm]{Remarks}
\newtheorem{definition}[thm]{Definition}
\newtheorem{claim}[thm]{Claim}
\newtheorem{assumption}[thm]{Assumption}
\newtheorem{assumptions}[thm]{Assumptions}
\newtheorem{properties}[thm]{Properties}
\newtheorem{example}[thm]{Example}
\newtheorem{comments}[thm]{Comments}
\newtheorem{blank}[thm]{}
\newtheorem{observation}[thm]{Observation}
\newtheorem{defn-thm}[thm]{Definition-Theorem}

\newcommand{\sM}{{\mathcal M}}


\title[The Laplace Transform of the Cut-and-Join Equation of
Mari\~no-Vafa Formula]{The Laplace Transform of the Cut-and-Join
Equation of Mari\~no-Vafa Formula and Its Applications}
       \author{Shengmao Zhu}
        \address{Center of Mathematical Sciences, Zhejiang University, Hangzhou, Zhejiang 310027, China}
        \email{zhushengmao@gmail.com}

        \begin{abstract}
        By the same method introduced in \cite{EMS}, we calculate the Laplace transform of the
        celebrated cut-and-join equation of Mari\~no-Vafa formula
        discovered by C. Liu, K. Liu and J. Zhou \cite{LLZ}. Then, we study the
        applications of the polynomial identity $(1)$ obtained in theorem 1.1 of this paper.
        We show the proof Bouchard-Mari\~no conjecture for $\mathbb{C}^3$
        which was given by L. Chen \cite{Ch2} firstly. Subsequently,
        we will present how to obtain series Hodge integral identities from
        this polynomial identity $(1)$. In particular, the main result in \cite{EMS} is one of special case
        in such series of Hodge integral identities.
        At last, we give a explicit formula for the computation
        of Hodge integral
        $\langle\tau_{b_{L}}\lambda_{g}\lambda_{1}\rangle_{g}$ where
        $b_{L}=(b_{1},...,b_{l})$.
        \end{abstract}
    \maketitle
\tableofcontents

\section{Introduction}
In a recent paper \cite{EMS}, Eynard, Mulase and Safnuk stated the
Laplace transform of the cut-and-join equation satisfied by the
partition function of Hurwitz numbers. They obtained a polynomial
identity of linear Hodge integrals. As an application, they proved
the Bouchard-Mari\~no conjecture on Hurwitz numbers. Then, Mulase
and Zhang \cite{MZ} stated that this polynomial identity can also be
used to derive the $DVV$ equation and $\lambda_{g}$-integral with
the same method introduced by \cite{GJV}.

In 2003, C. Liu, K. Liu and J. Zhou \cite{LLZ} proved the celebrated
Mari\~no-Vafa conjecture \cite{MV}. The main step in their proof is
to show the generating function of Hodge integral with triple
$\lambda$-classes $\mathcal{C}(\lambda,p;\tau)$ satisfies the
cut-and-join equation. Combining the cut-and-join restriction from
combinatorial side formula \cite{Zhou}, they finished the proof of
Mari\~no-Vafa formula. Moreover, the famous ELSV \cite{ELSV} formula
for Hurwitz numbers is a large framing limit of Mari\~no-Vafa
formula \cite{LLZ2}.

With the above motivations, we calculate the Laplace transform of
the cut-and-join equation for Mari\~no-Vafa's case in the first part
of this paper. The main result is
\begin{theorem}
For $g\geq 1$ and $l\geq 1$, we have the following equation:
\begin{align}
&-(\tau^2+\tau)^{l-2}\sum_{b_{L}\geq
0}\left((l-1)(2\tau+1)\langle\tau_{b_L}\Gamma_{g}(\tau)\rangle_{g}+
(\tau^{2}+\tau)\langle\tau_{b_{L}}\frac{d}{d\tau}\Gamma_{g}(\tau)\rangle_{g}\right)
\hat{\Psi}_{b_{L}}(t_{L};\tau)\\\nonumber
&-(\tau^{2}+\tau)^{l-1}\sum_{b_{L}\geq
0}\langle\tau_{b_{L}}\Gamma_{g}(\tau)\rangle_{g}\sum_{i=1}^{l}\left(\frac{\partial}{\partial
\tau}\hat{\Psi}_{b_{i}}(t_{i};\tau)+\frac{1}{t_{i}\tau+1}\hat{\Psi}_{b_{i}+1}(t_{i};\tau)\right)\hat{\Psi}_{b_{L\setminus
\{i\}}}(t_{L\setminus \{i\}};\tau)\\\nonumber
&=-\frac{(\tau^2+\tau)^{l-2}}{\tau+1}\sum_{1\leq i<j\leq
l}\sum_{\substack{a\geq 0\\b_{L\setminus\{i,j\} }\geq
0}}\langle\tau_{a}\tau_{b_{L\setminus\{i,j\}}}\Gamma_{g}(\tau)\rangle_{g}
\hat{\Psi}_{b_{L\setminus\{i,j\}
}}(t_{L\setminus\{i,j\}};\tau)\\\nonumber
&\cdot\frac{(t_{j}-1)(t_{i}^2\tau+t_{i})\hat{\Psi}_{a+1}(t_{i};\tau)-(t_{i}-1)(t_{j}^2\tau+t_{j})\hat{\Psi}_{a+1}(t_{j};\tau)}
{t_{i}-t_{j}}\\\nonumber
&+\frac{(\tau^2+\tau)^{l-1}}{2}\sum_{i=1}^{l}\sum_{\substack{a_{1}\geq
0\\a_{2}\geq 0\\b_{L\setminus\{i\}}\geq
0}}\left((\tau^2+\tau)\langle\tau_{a_{1}}\tau_{a_{2}}\tau_{b_{L\setminus\{i\}}}\Gamma_{g-1}(\tau)\rangle_{g-1}\right.\\\nonumber
&\left.-\sum_{\substack{g_{1}+g_{2}=g\\
\mathcal{I}\coprod\mathcal{J}=L\setminus\{i\}}}^{stable}\langle\tau_{a_{1}}\tau_{b_{\mathcal{I}}}\Gamma_{g_{1}}(\tau)
\rangle_{g_{1}}\langle\tau_{a_{2}}
\tau_{b_{\mathcal{J}}}\Gamma_{g_{2}}(\tau)\rangle_{g_{2}}\right)
\prod_{n=1}^{2}\hat{\Psi}_{a_{n}+1}(t_{i};\tau)\hat{\Psi}_{b_{L\setminus\{i\}}}(t_{L\setminus\{i\}};\tau)
\end{align}
Where $L=\{1,2,..,l\}$ is an index set, and for any subset $I\subset
L$, we denote
\begin{align*}
t_{I}=(t_{i})_{i\in I}, b_{I}=\{b_{i}|i\in I\},
\tau_{b_{I}}=\prod_{i\in I}\tau_{b_{i}},
\hat{\Psi}_{b_{I}}(t_{I},\tau)=\prod_{i\in
I}\hat{\Psi}_{b_{i}}(t_{i},\tau)
\end{align*}
and
$\Gamma_{g}(\tau)=\Lambda_{g}^{\vee}(1)\Lambda_{g}^{\vee}(-\tau-1)\Lambda_{g}^{\vee}(\tau)$,
$\hat{\Psi}_{n}(t;\tau)=\left(\frac{(t^2-t)(t\tau+1)}{\tau+1}\frac{d}{dt}\right)^n\left(\frac{t-1}{\tau+1}\right)$
for $n\geq 0$. The last summation in the formula is taken over all
partitions of g and disjoint subsets
$\mathcal{I}\coprod\mathcal{J}=L$ subject to the stability condition
$2g_{1}-1+|\mathcal{I}|>0$ and $2g_{2}-1+|\mathcal{J}|>0$.
\end{theorem}

We remark that theorem 1.1 is equivalent to the symmetrized
cut-and-joint equation of Mari\~no-Vafa formula obtained by L. Chen
\cite{Ch}. We will present this equivalence in appendix B.

In \cite{BM}, V. Bouchard and M. Mari\~no proposed a conjecture for
the calculation of topological string amplitudes of the toric three
fold $\mathbb{C}^{3}$ based on the recent work \cite{BKMP}.(We will
call this conjecture as Bouchard-Mari\~no conjecture for
$\mathbb{C}^{3}$. The Bourchard-Mari\~no conjecture for Hurwitz
number proved in \cite{EMS} is the large framing limit of it). We
calculate in section 4 that, the topological relation in
Bouchard-Mari\~no conjecture for $\mathbb{C}^{3}$ is equivalent to
the following identity of Hodge integral
\begin{align}
&(\tau^2+\tau)^{l-1}\sum_{b_L\geq
0}\langle\tau_{b_L}\Gamma_{g}(\tau)\rangle_{g}d\hat{\Psi}_{b_{L}}(t_L;\tau)\\\nonumber
&=-(\tau^2+\tau)^{l-2}\sum_{i=2}^{l}\sum_{a,b_{L\setminus
\{1,i\}}}\langle\tau_{a}\tau_{b_{L\setminus\{1,i\}}}\Gamma_{g}(\tau)
\rangle_{g}P_{a}(t_{1},t_{i};\tau)dt_{1}dt_{i}d\hat{\Psi}_{b_{L\setminus\{1,i\}}}(t_{L\setminus
\{1,i\}};\tau)\\\nonumber &+\sum_{\substack{a_1,a_2\geq 0
\\b_{L\setminus \{1\}}}}(\tau^2+\tau)^{l-1}\left((\tau^2+\tau)\langle\tau_{a_1}\tau_{a_2}\tau_{b_{L\setminus\{1\}}}\Gamma_{g-1}(\tau)\rangle_{g-1}
\right.\\\nonumber
&\left.-\sum_{\substack{g_1+g_2=g\\
\mathcal{I}\bigsqcup\mathcal{J}=L\setminus\{i\}}}^{stable}\langle
\tau_{a_1}\tau_{b_{I}}\Gamma_{g_1}(\tau)\rangle_{g_1,|I|+1}\langle
\tau_{a_2}\tau_{b_{J}}\Gamma_{g_2}(\tau)\rangle_{g_2,|J|+1}\right)
P_{a_1,a_2}(t_1;\tau)dt_1d\hat{\Psi}_{b_L\setminus\{1\}}(t_{L\setminus
\{1\}};\tau)
\end{align}

On the other hand, thanks to the Mari\~no-Vafa formula, the
topological amplitudes of $\mathbb{C}^3$ is completely determined by
the relation $(1)$ in theorem 1.1. In section 4, we will illustrate
that the identity $(2)$ is implied in $(1)$. Therefore, the
Bouchard-Mari\~no conjecture for $\mathbb{C}^3$ holds.

Furthermore, we use the result of theorem 1.1 to obtain some Hodge
integral identities. Indeed, $\hat{\Psi}_{b}(t,\tau)$ can be written
as
\begin{align*}
\hat{\Psi}_{b}(t;\tau)=\sum_{k=0}^{b}\frac{\tau^k}{(\tau+1)^{b+1}}\Psi_{b}^{k}(t)
\end{align*}
and $\Psi_{b}^{k}(t)$, $0\leq k\leq b$ could be computed from the
recursion relation they defined on.

At first, we consider the expansion of $\tau$ at $\infty$. Taking
the highest level of formula $(1)$, we get
\begin{corollary}
\begin{align*}
&\sum_{b_{L}\geq
0}\langle\tau_{b_{L}}\Lambda_{g}^{\vee}(1)\rangle\left((2g-2+l)\Psi_{b_{L}}^{b_{L}}(t_{L})+\sum_{i=1}^{l}(t_{i}^{2}-t_{i})
\frac{\partial}{\partial
t_{i}}\Psi_{b_{i}}^{b_{i}}(t_{i})\Psi_{b_{L\setminus\{i\}}}^{b_{L\setminus\{i\}}}(t_{L\setminus\{i\}})\right)\\
&=\sum_{1\leq i<j\leq l}\sum_{\substack{a\geq 0\\b_{L\setminus
\{i,j\}}\geq 0}}\langle\tau_{a}\tau_{b_{L\setminus
\{i,j\}}}\Lambda_{g}^{\vee}(1) \rangle_{g}\Psi_{b_{L\setminus
\{i,j\}}}^{b_{L\setminus
\{i,j\}}}\frac{(t_{j}-1)t_{i}^{2}\Psi_{a+1}^{a+1}(t_{i})-(t_{i}-1)t_{j}^{2}\Psi_{a+1}^{a+1}(t_{j})}{t_{i}-t_{j}}\\
&+\frac{1}{2}\sum_{i=1}^{l}\sum_{\substack{a_{1}\geq 0\\a_{2}\geq
0\\b_{L \setminus \{i\}}\geq
0}}\left(\langle\tau_{a_{1}}\tau_{a_{2}}\tau_{b_{L\setminus
\{i\}}}\Lambda_{g-1}^{\vee}(1)
\rangle_{g-1}+\sum_{\substack{g_{1}+g_{2}=g\\
\mathcal{I}\cup\mathcal{J}=L\setminus\{i\}}}^{stable}\langle\tau_{a_{1}}\tau_{b_{\mathcal{I}}}\Lambda_{g_{1}}^{\vee}(1)
\rangle_{g_{1}}\langle\tau_{a_{2}}\tau_{b_{\mathcal{J}}}\Lambda_{g_{2}}^{\vee}(1)
\rangle_{g_{2}} \right)\\
&\times
\Psi_{a_{1}+1}^{a_{1}+1}(t_{i})\Psi_{a_{2}+1}^{a_{2}+1}(t_{i})\Psi_{b_{L\setminus\{i\}
}}^{b_{L\setminus\{i\} }}(t_{L\setminus\{i\}})
\end{align*}
\end{corollary}
We show in section 5 that $\Psi_{b}^{b}(t)$ is equal to
$\hat{\xi}_{b}(t)$ which is defined by
$\hat{\xi}_{b}(t)=((t^3-t^2)\frac{d}{dt})^{b}(t-1)$ in formula (1.2)
in \cite{EMS}. Hence, the main theorem 1.1 in \cite{EMS} is a
special case of formula $(1)$ in this paper. Taking the sub-highest
level of theorem 1.1, we also get another Hodge integral identity
corollary 5.2 in section 5. Unfortunately, this identity does not
contain any new information.

Next, we consider the expansion of $\tau$ at $\tau=0$. In this case,
the lowest level of formula $(1)$ is
\begin{corollary}
\begin{align*}
&\sum_{b_{L}\geq
0}\langle\tau_{b_{L}}\lambda_{g}\rangle_{g}\Psi_{b_{L}}^{0}(t_{L})=\frac{1}{l-1}\sum_{1\leq i<j \leq l}\sum_{\substack{a\geq 0\\
b_{L\setminus\{i,j\}}\geq 0}}\langle\tau_{a}\tau_{b_{L\setminus
\{i,j\}}}\lambda_{g}\rangle_{g}\Psi_{b_{L\setminus
\{i,j\}}}^{0}(t_{L\setminus\{i,j\}})\\\nonumber
&\cdot\frac{(t_{j}-1)t_{i}^{2}\Psi_{a+1}^{0}(t_{i})-(t_{i}-1)t_{j}^{2}\Psi_{a+1}^{0}(t_{j})}
{t_{i}-t_{j}}
\end{align*}
\end{corollary}
We can rederive the $\lambda_{g}$-integral \cite{FabP} from
Corollary 1.3.

We also pick up the sub-lowest level of formula $(1)$ and have
\begin{corollary}
\begin{align*}
&\sum_{b_{L}\geq 0}\langle\tau_{b_{L}}\lambda_{g}\rangle_{g}l\left(
-(|b_{L}|+1)\Psi_{b_{L}}^{0}(t_{L})+\sum_{j=1}\Psi_{b_{j}}^{1}(t_{j})\Psi_{b_{L\setminus\{j\}}}^{0}(t_{L\setminus\{j\}})\right)\\\nonumber
&+\sum_{b_{L}\geq
0}\langle\tau_{b_{L}}\lambda_{g}\rangle_{g}\sum_{i=1}^{l}(t_{i}^{2}-t_{i})\frac{\partial}{\partial
t_{i}}\Psi_{b_{i}}^{0}(t_{i})\Psi_{b_{L\setminus\{i\}}}^{0}(t_{L\setminus\{i\}})+\sum_{b_{L}\geq
0}\langle\tau_{b_{L}}\sum_{d=g-1}^{3g-3}P_{d}(\lambda)\rangle_{g}l\Psi_{b_{L}}^{0}(t_{L})\\\nonumber
&=\sum_{1\leq i<j \leq l}\sum_{\substack{a\geq 0\\
b_{L\setminus\{i,j\}}\geq 0}}\langle\tau_{a}\tau_{b_{L\setminus
\{i,j\}}}\lambda_{g}\rangle_{g}\Psi_{b_{L\setminus
\{i,j\}}}^{0}(t_{L\setminus\{i,j\}})[\frac{(t_{j}-1)t_{i}\left(\Psi_{a+1}^{1}(t_{i})+(t_{i}-|b_{L\setminus\{i,j\}}|-a-3)
\Psi_{a+1}^{0}(t_{i})\right)} {t_{i}-t_{j}}\\\nonumber & -\frac{
(t_{i}-1)t_{j}\left(\Psi_{a+1}^{1}(t_{j})+(t_{j}-|b_{L\setminus\{i,j\}}|-a-3)\Psi_{a+1}^{0}(t_{j})\right)}
{t_{i}-t_{j}}]\\\nonumber &+\sum_{1\leq i<j \leq l}\sum_{\substack{a\geq 0\\
b_{L\setminus\{i,j\}}\geq 0}}\langle\tau_{a}\tau_{b_{L\setminus
\{i,j\}}}\lambda_{g}\rangle_{g}\sum_{r\neq
i,j}\Psi_{b_{r}}^{1}(t_{r})\Psi_{b_{L\setminus
\{i,j\}}}^{0}(t_{L\setminus\{i,j\}})\frac{(t_{j}-1)t_{i}\Psi_{a+1}^{0}(t_{i})-(t_{i}-1)t_{j}\Psi_{a+1}^{0}(t_{j})}
{t_{i}-t_{j}}\\\nonumber &+
\sum_{1\leq i<j \leq l}\sum_{\substack{a\geq 0\\
b_{L\setminus\{i,j\}}\geq 0}}\langle\tau_{a}\tau_{b_{L\setminus
\{i,j\}}}\sum_{d=g-1}^{3g-3}P_{d}(\lambda)\rangle_{g}\Psi_{b_{L\setminus\{i,j\}}}^{0}(t_{L\setminus\{i,j\}})
\frac{(t_{j}-1)t_{i}\Psi_{a+1}^{0}(t_{i})-(t_{i}-1)t_{j}\Psi_{a+1}^{0}(t_{j})}{t_{i}-t_{j}}\\\nonumber
&+\frac{1}{2}\sum_{i=1}^{l}\sum_{\substack{a_{1}\geq 0\\a_{2}\geq
0\\b_{L\setminus\{i\}}\geq
0}}\sum_{\substack{g_{1}+g_{2}=g\\|\mathcal{I}|\cup|\mathcal{J}|=L\setminus\{i\}}}^{stable}\langle\tau_{a_{1}}\tau_{b_{\mathcal{I}}}
\lambda_{g_{1}}\rangle_{g_{1}}\langle\tau_{a_{2}}\tau_{b_{\mathcal{J}}}
\lambda_{g_{2}}\rangle_{g_{2}}\Psi_{a_{1}+1}^{0}(t_{i})\Psi_{a_{2}+1}^{0}(t_{i})\Psi_{b_{L\setminus\{i\}}}^{0}(t_{L\setminus\{i\}})
\end{align*}
where
$\sum_{d=g-1}^{3g-3}P_{d}(\lambda)=\Lambda_{g}^{\vee}(1)a_{1}(\lambda)$
and
$a_{1}(\lambda)=\sum_{m=1}^{g}m\lambda_{g-m}\lambda_{g}-(-1)^{g}\Lambda_{g}^{\vee}(-1)\lambda_{g-1}$.
\end{corollary}
The above identity contains Hodge integral of type
$\langle\tau_{b_{L}}\sum_{d=g-1}^{3g-3}P_{d}(\lambda)\rangle_{g}$.
By some direct calculations,
\begin{align*}
P_{g-1}(\lambda)=\lambda_{g-1}, P_{g}(\lambda)=g\lambda_{g},
P_{g+1}(\lambda)=-\lambda_{g}\lambda_{1},\cdots,
P_{3g-3}(\lambda)=(-1)^{g+1}\lambda_{g}\lambda_{g-1}\lambda_{g-2}
\end{align*}
As an application of corollary 1.4, we get the following Hodge
integral recursion.
\begin{theorem}
If $\sum_{i=1}^{l}b_{i}=2g-4+l$, there exists a constant
$C(g,l,b_{1},..,b_{l})$ related to $g,l,b_{1},..,b_{l}$, such that
\begin{align*}
\langle\tau_{b_{L}}\lambda_{g}\lambda_{1}\rangle_{g}=\frac{1}{l}\sum_{1\leq
i<j \leq l}\langle\tau_{b_{i}+b_{j}-1}\tau_{b_{L\setminus
\{i,j\}}}\lambda_{g}\lambda_{1}\rangle_{g}\frac{(b_{i}+b_{j})!}{b_{i}!b_{j}!}+C(g,l,b_{1},..,b_{l})
\end{align*}
where $C(g,l,b_{1},..,b_{l})$ is a very verbose combinatoric
constant which is given at Appendix B.
\end{theorem}

The initial value
$\langle\tau_{2g-3}\lambda_{g}\lambda_{1}\rangle_{g}=\frac{1}{12}[g(2g-3)b_{g}+b_{1}b_{g-1}]$
has been computed by Y. Li \cite{Li}. Thus, by the recursion formula
in theorem 1.5, we can compute out all the Hodge integral of type $
\langle\tau_{b_{L}}\lambda_{g}\lambda_{1}\rangle_{g} $. In fact, the
Hodge integrals $ \langle\tau_{b_{L}}P_{d}(\lambda)\rangle_{g}$
appeared in Corollary 1.4 could also be calculated via the same
method. But the computation will be more complicated.
\\

{\bf Acknowledgements.} The author would like to thank his advisor
Kefeng Liu's encouragement, valuable discussions, and bringing the
paper \cite{EMS} to his attention.

\vskip 30pt
\section{The Laplace Transform of  Mari\~no-Vafa Formula}
At first, we introduce some notations followed by \cite{LLZ}. Let
\begin{align*}
\Lambda_{g}^{\vee}(t)=t^{g}-\lambda_{1}t^{g-1}+\cdots+(-1)^{g}\lambda_{g}
\end{align*}
be the Chern polynomial of $\mathbb{E}^{\vee}$, the dual of the
Hodge bundle. For a partition $\mu$ given by
\begin{align*}
\mu_{1}\geq \mu_{2}\geq \cdots \geq \mu_{l(\mu)}>0,
\end{align*}
let $|\mu|=\sum_{i=1}^{l(\mu)}\mu_{i}$ and $\Gamma_{g}(\tau)=
\Lambda_{g}^{\vee}(1)\Lambda_{g}^{\vee}(-\tau-1)\Lambda_{g}^{\vee}(\tau)$,
define
\begin{align*}
\mathcal{C}_{g,\mu}(\tau)=-\frac{\sqrt{-1}^{|\mu|+l(\mu)}}{|Aut(\mu)|}[\tau(\tau+1)]^{l(\mu)-1}\prod_{i=1}^{l(\mu)}
\frac{\prod_{a=1}^{\mu_{i}-1}(\mu_{i}\tau+a)}{(\mu_{i}-1)!}\int_{\overline{\mathcal{M}}_{g,l(\mu)}}
\frac{\Gamma_{g}(\tau)}{\prod_{i=1}^{l(\mu)}(1-\mu_{i}\psi_{i})}
\end{align*}

The Deligne-Mumford stack $\overline{\mathcal{M}}_{g,l}$ is defined
as the moduli space of stable curves satisfying the stability
condition $2g-2+l>0$. For the unstable cases $(g,l)=(0,1)$ and
$(0,2)$, we define
\begin{align*}
\int_{\overline{\mathcal{M}}_{0,1}}\frac{1}{1-\mu\psi}=\frac{1}{\mu^2}
\end{align*}
\begin{align*}
\int_{\overline{\mathcal{M}}_{0,2}}\frac{1}{(1-\mu_{1}\psi_{1})(1-\mu_{2}\psi_{2})}=\frac{1}{\mu_{1}+\mu_{2}}
\end{align*}

Thus
\begin{align*}
\mathcal{C}_{0,d}(\tau)=-\sqrt{-1}^{d+1}\frac{1}{\tau}\frac{\prod_{a=0}^{d-1}(d\tau+a)}{d!}d^{-2}
\end{align*}
\begin{align*}
\mathcal{C}_{0,(\mu_{1},\mu_{2})}(\tau)=\frac{\sqrt{-1}^{\mu_{1}+\mu_{2}}}{|Aut(\mu_{1},\mu_{2})|}
\frac{\tau+1}{\tau}\prod_{i=1}^{2}\frac{\prod_{a=0}^{\mu_{i}-1}(\mu_{i}\tau+a)}{\mu_{i}!}\frac{1}{\mu_{1}+\mu_{2}}
\end{align*}

The Mari\~no-Vafa formula proved in \cite{LLZ} gives a direct
combinatorial formula associated to representation of symmetric
groups for the generation function of $\mathcal{C}_{g,\mu}(\tau)$.
But we don't go further to this formula here.

Through a direct calculation, we have
\begin{align}
\mathcal{C}_{g,\mu}(\tau)=-\frac{\sqrt{-1}^{|\mu|+l(\mu)}}{|Aut(\mu)|}[\tau(\tau+1)]^{l(\mu)-1}\sum_{\substack{b_{i}\geq
0\\i=1,..,l(\mu)}}\langle\prod_{i=1}^{l(\mu)}\tau_{b_{i}}\Gamma_{g}(\tau)
\rangle_{g}\prod_{i=1}^{l(\mu)}\frac{1}{\tau}\frac{\prod_{a=0}^{\mu_{i}-1}(\mu_{i}\tau+a)}{\mu_{i}!}\mu_{i}^{b_{i}}
\end{align}

Let $\mathcal{C}_{g}(\mu;\tau)=|Aut(\mu)|\mathcal{C}_{g,\mu}(\tau)$,
the Laplace transform of $\mathcal{C}_{g,\mu}(\tau)$ is defined by,
\begin{align*}
\mathcal{C}_{g,l}(w_{1},..,w_{l})=\sum_{\mu\in
 \mathbb{N}^{l}}\frac{1}{\sqrt{-1}^{l+|\mu|}}\mathcal{C}_{g}(\mu;\tau)e^{-(\mu_{1}w_{1}+\cdots+\mu_{l}w_{l})}
\end{align*}
where we have let $l=l(\mu)$.

In order to simplify the above expression of laplace transform, we
introduce some new variables. Consider the framed Lambert curve
$x=y(1-y)^\tau$, and the coordinate change $x=e^{-w}$. By Lagrange
inversion theorem, L. Chen \cite{Ch} got,
\begin{align*}
y=\sum_{k\geq1}\frac{\prod_{a=0}^{k-2}(k\tau+a)}{k!}x^{k}
\end{align*}
Now we introduce the variable $t$ by,
\begin{align*}
t=1+\frac{1+\tau}{\tau}\sum_{k\geq
1}\frac{\prod_{a=0}^{k-1}(k\tau+a)}{k!}e^{-kw}
\end{align*}
and
\begin{align*}
\hat{\Psi}_{n}(t;\tau)=\frac{1}{\tau}\sum_{k\geq1}\frac{\prod_{a=0}^{k-1}(k\tau+a)}{k!}k^{n}e^{-kw}
\end{align*}
which is a polynomial with top degree $2n+1$ in variable $t$ via the
recursion relation
\begin{align*}
\hat{\Psi}_{n+1}(t;\tau)=(t^2-t)\frac{t\tau+1}{\tau+1}\frac{d}{dt}\hat{\Psi}_{n}(t;\tau)
\end{align*}
for $n\geq 0$ and $\hat{\Psi}_{0}(t;\tau)=\frac{t-1}{\tau+1}$.

Now, the laplace transform of $\mathcal{C}_{g,\mu}(\tau)$ can be
written in $t$ variable,
\begin{align}
\hat{\mathcal{C}}_{g,l}(t_{1},..,t_{l})&=\sum_{\mu\in
\mathbb{N}^{l}}\frac{1}{\sqrt{-1}^{l+|\mu|}}\mathcal{C}_{g}(\mu;\tau)e^{-(\mu_{1}w_{1}+\cdots+\mu_{l}w_{l})}\\\nonumber
&=-(\tau(\tau+1))^{l-1}\sum_{\substack{b_{i}\geq
0\\i=1,..,l}}\langle\prod_{i=1}^{l}\tau_{b_{i}}\Gamma_{g}(\tau)\rangle_{g}
\prod_{i=1}^{l}\hat{\Psi}_{b_{i}}(t_{i};\tau)
\end{align}

\section{The Laplace Transform of The Cut-and-join Equation of Mari\~no-Vafa Formula}

Let us introduce formal variables $p=(p_{1},..,p_{n},..)$, and
define
$$p_{\mu}=p_{\mu_{1}}\cdots p_{l(\mu)}$$
for a partition $\mu$. Define generating functions
\begin{align*}
\mathcal{C}_{g,l}(\lambda,p;\tau)=\sum_{\mu,l(\mu)=l}\frac{\mathcal{C}_{g}(\mu;\tau)}{|Aut(\mu)|}p_{\mu}\lambda^{2g-2+l}
\end{align*}
\begin{align*}
\mathcal{C}(\lambda,p;\tau)=\sum_{\substack{g\geq 0\\l\geq
1}}\mathcal{C}_{g,l}(\lambda,p;\tau)
\end{align*}

In 2003, C. Liu, K. Liu and J. Zhou \cite{LLZ} proved that
$\mathcal{C}(\lambda,p;\tau)$ satisfies the cut-and-join equation
\begin{align*}
\frac{\partial\mathcal{C}}{\partial
\tau}=\frac{\sqrt{-1}}{2}\lambda\sum_{i,j\geq
1}((i+j)p_{i}p_{j}\frac{\partial\mathcal{C}}{\partial
p_{i+j}}+ijp_{i+j}\frac{\partial^2\mathcal{C}}{\partial
p_{i}\partial p_{j}}+ijp_{i+j}\frac{\partial\mathcal{C}}{\partial
p_{i}}\frac{\partial\mathcal{C}}{\partial p_{j}})
\end{align*}

For every choice of $g\geq 1$ and a partition $\mu$, the coefficient
of $p_{\mu}\lambda^{2g-2+l(\mu)}$ is

\begin{align}
\frac{\partial}{\partial
\tau}\left(\frac{\mathcal{C}_{g}(\mu;\tau)}{|Aut(\mu)|}\right)&=\sqrt{-1}\sum_{i<j}(\mu_{i}+\mu_{j})\frac{\mathcal{C}_{g}(\mu;\tau)}
{|Aut(\mu(\hat{i},\hat{j}))|}\\\nonumber
&+\frac{\sqrt{-1}}{2}\sum_{i=1}^{l}\sum_{\alpha+\beta=\mu_{i}}\alpha\beta
\left(\frac{\mathcal{C}_{g-1}(\mu(\alpha,\hat{i};\tau))}{|Aut(\mu(\alpha,\hat{i}))|}+
\sum_{\substack{g_{1}+g_{2}=g\\\nu_{1}\coprod\nu_{2}=\mu(\alpha,\hat{i})}}\frac{\mathcal{C}_{g_{1}}(\nu_{1};\tau)}{|Aut(\nu_{1}|)}
\frac{\mathcal{C}_{g_{2}}(\nu_{2};\tau)}{|Aut(\nu_{2})|}\right)
\end{align}

Let us first calculate the Laplace transform of the cut-and-join
equation for the $l=1$ case which includes Proposition 3.3 in
\cite{EMS} as its special case.
\begin{proposition}
For $g\geq 1$, the Laplace transform of the cut-and-join equation
for the $l=1$ case is:
\begin{align*}
&-\sum_{b\geq 0}\langle\tau_{b}
\frac{d}{d\tau}\Gamma_{g}(\tau)\rangle_{g}\hat{\Psi}_{b}(t;\tau)-\sum_{b\geq
0}\langle\tau_{b}
\Gamma_{g}(\tau)\rangle_{g}\left(\frac{d}{d\tau}\hat{\Psi}_{b}(t;\tau)+
\frac{1}{t\tau+1}\hat{\Psi}_{b+1}(t;\tau)\right)\\
&=\frac{1}{2}\sum_{a_{1},a_{2}\geq
0}\left((\tau(\tau+1))\langle\tau_{a_{1}}\tau_{a_{2}}
\Gamma_{g-1}(\tau)\rangle_{g-1}-\sum_{\substack{g_{1}+g_{2}=g\\g_{1}>0\\g_{2}>0}}\langle\tau_{a_{1}}
\Gamma_{g_{1}}(\tau)\rangle_{g_{1}}\langle\tau_{a_{2}}
\Gamma_{g_{2}}(\tau)\rangle_{g_{2}}\right)\hat{\Psi}_{a_{1}+1}(t;\tau)
\hat{\Psi}_{a_{2}+1}(t;\tau)
\end{align*}
\end{proposition}
\begin{proof}
The cut-and-join equation for $l=1$ case is
\begin{align}
\frac{\partial\mathcal{C}_{g}(\mu;\tau)}{\partial
\tau}=\frac{\sqrt{-1}}{2}\sum_{\alpha+\beta=\mu}\alpha\beta\left(\frac{\mathcal{C}_{g-1}((\alpha,\beta);\tau)}{|Aut((\alpha,\beta))|}
+\sum_{g_{1}+g_{2}=g}\mathcal{C}_{g_{1}}(\alpha;\tau)\mathcal{C}_{g_{2}}(\beta;\tau)\right)
\end{align}
Then, the Laplace transform of the LHS of $(6)$ is
\begin{align}
\frac{\partial}{\partial
\tau}\mathcal{\hat{C}}_{g,1}(t;\tau)=\frac{\partial t}{\partial
\tau}\frac{\partial}{\partial
t}\mathcal{\hat{C}}_{g,1}(t;\tau)+\frac{\partial}{\partial\tau}\mathcal{\hat{C}}_{g,1}(t;\tau)
\end{align}

The Laplace transform of the stable part of RHS of $(6)$ is
\begin{align*}
&\frac{1}{2}(\tau(\tau+1))\sum_{a_{1},a_{2}\geq
0}\langle\tau_{a_{1}}\tau_{a_{2}}
\Gamma_{g-1}(\tau)\rangle_{g-1}\hat{\Psi}_{a_{1}+1}(t;\tau)
\hat{\Psi}_{a_{2}+1}(t;\tau)\\
&-\frac{1}{2}\sum_{\substack{g_{1}+g_{2}=g\\g_{1}>0\\g_{2}>0}}\langle\tau_{a_{1}}
\Gamma_{g_{1}}(\tau)\rangle_{g_{1}}\langle\tau_{a_{2}}
\Gamma_{g_{2}}(\tau)\rangle_{g_{2}}\hat{\Psi}_{a_{1}+1}(t;\tau)
\hat{\Psi}_{a_{2}+1}(t;\tau)
\end{align*}

The unstable term is
$\mathcal{C}_{0}(\alpha;\tau)\mathcal{C}_{g}(\beta;\tau)+\mathcal{C}_{g}(\alpha;\tau)\mathcal{C}_{0}(\beta;\tau)$.
Because
\begin{align*}
\mathcal{C}_{0,d}(\tau)=-\sqrt{-1}^{d+1}\frac{1}{\tau}\frac{\prod_{a=0}^{d-1}(d\tau+a)}{d!}\frac{1}{d^2}
\end{align*}
It's Laplace transform $\hat{\mathcal{C}}_{0}(w;\tau)=\sum_{d\geq
1}-\frac{1}{\sqrt{-1}^{d+1}}\mathcal{C}_{0,d}(\tau)e^{-w}$
satisfies:
\begin{align*}
\frac{d}{dw}\hat{\mathcal{C}}_{0}(w;\tau)=\frac{t-1}{\tau+1}=\frac{y}{1-(\tau+1)y}
\end{align*}
It is easy to calculate
$\frac{d}{dw}=\frac{(1-y)y}{1-(1+\tau)y}\frac{d}{dy}$, hence,
$\hat{\mathcal{C}}_{0}(w;\tau)=-ln(1-y)$. Moreover, remember that
the framed Lambert curve is $y(1-y)^{\tau}=x$, where y depends on
$\tau$. Taking derivation of $\tau$, we have the identity,
\begin{align*}
-ln(1-y)=\frac{\partial y}{\partial \tau}\frac{1-(\tau+1)y}{y(1-y)}
\end{align*}

Therefore, the unstable part of RHS of $(6)$ is
\begin{align}
\frac{\partial y}{\partial
\tau}\frac{1-(\tau+1)y}{y(1-y)}\frac{(t^2-t)(t\tau+1)}{\tau+1}\frac{\partial}{\partial
t}\hat{\mathcal{C}}_{g,1}(t;\tau)=\frac{\partial y}{\partial
\tau}t^2(\tau+1)\frac{\partial}{\partial
t}\hat{\mathcal{C}}_{g,1}(t;\tau)
\end{align}
where we have used $t=\frac{1}{1-(\tau+1)y}$. we also have
\begin{align}
\frac{\partial t}{\partial \tau}=t^{2}y+\frac{\partial y}{\partial
\tau}t^2(\tau+1)=\frac{t^2-t}{\tau+1}+\frac{\partial y}{\partial
\tau}t^2(\tau+1)
\end{align}
Hence, move the unstable part to left hand side, by $(7)$,$(8)$ and
$(9)$, we get
\begin{align*}
\left(\frac{\partial}{\partial
\tau}+\frac{t^2-t}{(\tau+1)}\frac{\partial }{\partial
t}\right)\hat{C}_{g,1}(t;\tau)=stable\quad part
\end{align*}
which is just the Proposition 3.1.
\end{proof}

For the general case of $l$, we need to introduce two lemmas first.

\begin{lemma}
When $(g,l)=(0,2)$, we have the following Laplace transformation
formula:
\begin{align*}
\hat{\mathcal{C}}_{0,2}(w_{1},w_{2};\tau)&=-\sum_{\alpha,\beta \geq
1}\frac{\tau+1}{\tau}\frac{1}{\alpha+\beta}\frac{\prod_{a=0}^{\alpha-1}(\alpha\tau+a)}{\alpha!}
\frac{\prod_{a=0}^{\beta-1}(\beta\tau+a)}{\beta!}e^{-\alpha w_{1}}e^{-\beta w_2}\\
&=-ln\left(\frac{y_{1}-y_{2}}{x_{1}-x_{2}}\right)-\tau\left(ln(1-y_{1})+ln(1-y_{2})\right)
\end{align*}
\end{lemma}
\begin{proof}
By definition,
\begin{align*}
\hat{\mathcal{C}}_{0,2}(w_{1},w_{2};\tau)&=\sum_{\alpha,\beta\geq
1}\frac{1}{\sqrt{-1}^{2+\alpha+\beta}}|Aut(\alpha,\beta)|\hat{\mathcal{C}}_{0,(\alpha,\beta)}\\
&=-\sum_{\alpha,\beta \geq
1}\frac{\tau+1}{\tau}\frac{1}{\alpha+\beta}\frac{\prod_{a=0}^{\alpha-1}(\alpha\tau+a)}{\alpha!}
\frac{\prod_{a=0}^{\beta-1}(\beta\tau+a)}{\beta!}e^{-\alpha
w_{1}}e^{-\beta w_2}
\end{align*}
Thus
\begin{align*}
&\left(\frac{d}{dw_{1}}+\frac{d}{dw_{2}}\right)\hat{\mathcal{C}}_{0,2}(w_{1},w_{2};\tau)=-\tau(\tau+1)\hat{\Psi}_{0}(t_{1};\tau)
\hat{\Psi}_{0}(t_{2};\tau)\\
&=-\tau(\tau+1)\frac{y_{1}y_{2}}{(1-(\tau+1)y_{1})(1-(\tau+1)y_{2})}
\end{align*}
Because,
\begin{align*}
\frac{d}{dw}=x\frac{d}{dx}=\frac{(1-y)y}{1-(\tau+1)y}\frac{d}{dy}
\end{align*}
Then, it is easy to get
\begin{align*}
\hat{\mathcal{C}}_{0,2}(w_{1},w_{2};\tau)=-ln\left(\frac{y_{1}-y_{2}}{x_{1}-x_{2}}\right)-\tau\left(ln(1-y_{1})+ln(1-y_{2})\right)
\end{align*}
\end{proof}

\begin{lemma}
\begin{align*}
&\sum_{\alpha,\beta\geq
1}\frac{1}{\tau}\frac{\prod_{a=0}^{(\alpha+\beta)-1}((\alpha+\beta)\tau+a)}{(\alpha+\beta)!}(\alpha+\beta)^{a+1}
e^{-\alpha w_{i}}e^{-\beta w_{j}}\\
&=\frac{x_{i}}{x_{i}-x_{j}}\hat{\Psi}_{a+1}(t_{i};\tau)-\frac{x_{j}}{x_{i}-x_{j}}\hat{\Psi}_{a+1}(t_{j};\tau)
-\hat{\Psi}_{a+1}(t_{i};\tau)
-\hat{\Psi}_{a+1}(t_{j};\tau)
\end{align*}
\end{lemma}
\begin{proof}
This calculation is the same as formula $(3.15)$ showed in paper
\cite{EMS}.

Let $\mu=\alpha+\beta$ and $\nu=\beta$, then
\begin{align*}
&\sum_{\alpha,\beta\geq
1}\frac{1}{\tau}\frac{\prod_{a=0}^{(\alpha+\beta)-1}((\alpha+\beta)\tau+a)}{(\alpha+\beta)!}
(\alpha+\beta)^{a+1}e^{-\alpha w_{i}}e^{-\beta w_{j}}\\
&=\sum_{\alpha,\beta\geq
0}\frac{1}{\tau}\frac{\prod_{a=0}^{(\alpha+\beta)-1}((\alpha+\beta)\tau+a)}{(\alpha+\beta)!}
(\alpha+\beta)^{a+1}e^{-\alpha w_{i} }e^{-\beta w_{j} }\\
&-\sum_{\alpha\geq
1}\frac{1}{\tau}\frac{\prod_{a=0}^{\alpha-1}(\alpha\tau+a)}{\alpha!}\alpha^{a+1}e^{-\alpha
w_{i}}-\sum_{\beta\geq
1}\frac{1}{\tau}\frac{\prod_{a=0}^{\beta-1}(\beta\tau+a)}{\beta!}\beta^{a+1}e^{-\beta w_{j}}\\
&=\sum_{\mu\geq
0}\sum_{\nu=0}^{\mu}\frac{1}{\tau}\frac{\prod_{a=0}^{\mu-1}(\mu\tau+a)}{\mu!}\mu^{a+1}e^{-(\mu-\nu)
w_{i}}e^{-\nu w_{j} } -\hat{\Psi}_{a+1}(t_{i};\tau)
-\hat{\Psi}_{a+1}(t_{j};\tau)\\
&=\sum_{\mu\geq
0}\frac{1}{\tau}\frac{\prod_{a=0}^{\mu-1}(\mu\tau+a)}{\mu!}\mu^{a+1}\frac{x_{i}^{\mu+1}-x_{j}^{\mu+1}}{x_{i}-x_{j}}
-\hat{\Psi}_{a+1}(t_{i};\tau)
-\hat{\Psi}_{a+1}(t_{j};\tau)\\
&=\frac{x_{i}}{x_{i}-x_{j}}\hat{\Psi}_{a+1}(t_{i};\tau)-\frac{x_{j}}{x_{i}-x_{j}}\hat{\Psi}_{a+1}(t_{j};\tau)
-\hat{\Psi}_{a+1}(t_{i};\tau)
-\hat{\Psi}_{a+1}(t_{j};\tau)
\end{align*}
\end{proof}

Now we give our main result in this paper.
\begin{theorem}
For $g\geq 1$ and $l\geq 1$, we have the following equation:
\begin{align}
&-(\tau^2+\tau)^{l-2}\sum_{b_{L}\geq
0}\left((l-1)(2\tau+1)\langle\tau_{b_L}\Gamma_{g}(\tau)\rangle_{g}+
(\tau^{2}+\tau)\langle\tau_{b_{L}}\frac{d}{d\tau}\Gamma_{g}(\tau)\rangle_{g}\right)
\hat{\Psi}_{b_{L}}(t_{L};\tau)\\\nonumber
&-(\tau^{2}+\tau)^{l-1}\sum_{b_{L}\geq
0}\langle\tau_{b_{L}}\Gamma_{g}(\tau)\rangle_{g}\sum_{i=1}^{l}\left(\frac{\partial}{\partial
\tau}\hat{\Psi}_{b_{i}}(t_{i};\tau)+\frac{1}{t_{i}\tau+1}\hat{\Psi}_{b_{i}+1}(t_{i};\tau)\right)\hat{\Psi}_{b_{L\setminus
\{i\}}}(t_{L\setminus \{i\}};\tau)\\\nonumber
&=-\frac{(\tau^2+\tau)^{l-2}}{\tau+1}\sum_{1\leq i<j\leq
l}\sum_{\substack{a\geq 0\\b_{L\setminus\{i,j\} }\geq
0}}\langle\tau_{a}\tau_{b_{L\setminus\{i,j\}}}\Gamma_{g}(\tau)\rangle_{g}
\hat{\Psi}_{b_{L\setminus\{i,j\}
}}(t_{L\setminus\{i,j\}};\tau)\\\nonumber
&\cdot\frac{(t_{j}-1)(t_{i}^2\tau+t_{i})\hat{\Psi}_{a+1}(t_{i};\tau)-(t_{i}-1)(t_{j}^2\tau+t_{j})\hat{\Psi}_{a+1}(t_{j};\tau)}
{t_{i}-t_{j}}\\\nonumber
&+\frac{(\tau^2+\tau)^{l-1}}{2}\sum_{i=1}^{l}\sum_{\substack{a_{1}\geq
0\\a_{2}\geq 0\\b_{L\setminus\{i\}}\geq
0}}\left((\tau^2+\tau)\langle\tau_{a_{1}}\tau_{a_{2}}\tau_{b_{L\setminus\{i\}}}\Gamma_{g-1}(\tau)\rangle_{g-1}\right.\\\nonumber
&\left.-\sum_{\substack{g_{1}+g_{2}=g\\
\mathcal{I}\coprod\mathcal{J}=L\setminus\{i\}}}^{stable}\langle\tau_{a_{1}}\tau_{b_{\mathcal{I}}}\Gamma_{g_{1}}(\tau)
\rangle_{g_{1}}\langle\tau_{a_{2}}
\tau_{b_{\mathcal{J}}}\Gamma_{g_{2}}(\tau)\rangle_{g_{2}}\right)
\prod_{n=1}^{2}\hat{\Psi}_{a_{n}+1}(t_{i};\tau)\hat{\Psi}_{b_{L\setminus\{i\}}}(t_{L\setminus\{i\}};\tau)
\end{align}
\end{theorem}
\begin{proof}
The Laplace transform of LHS of equation $(5)$ is
\begin{align}
&\sum_{\mu\in
\mathbb{N}^{l}}\frac{1}{\sqrt{-1}^{l(\mu)+|\mu|}}\frac{\partial}{\partial
\tau}\mathcal{C}_{g}(\mu;\tau)e^{-(\mu_{1}w_{1}+\cdots+\mu_{l}w_{l})}\\\nonumber
&=\frac{\partial}{\partial \tau}\left(\sum_{\mu\in
\mathbb{N}}\frac{1}{\sqrt{-1}^{l(\mu)+|\mu|}}\mathcal{C}_{g}(\mu;\tau)e^{-(\mu_{1}w_{1}+\cdots+\mu_{l}w_{l})}
\right)\\\nonumber &=\frac{\partial}{\partial
\tau}\hat{\mathcal{C}}_{g,l}(t_{1},..,t_{l};\tau)\\\nonumber
&=\sum_{i=1}^{l}\frac{\partial t_{i}}{\partial
\tau}\frac{\partial}{\partial
t_{i}}\hat{\mathcal{C}}_{g,l}(t_{1},..,t_{l};\tau)+\frac{\partial}{\partial
\tau}\hat{\mathcal{C}}_{g,l}(t_{1},..,t_{l};\tau)
\end{align}

The Laplace transform of stable geometry in the cut-term of RHS of
$(5)$ is
\begin{align}
&\frac{\sqrt{-1}}{2}\sum_{\mu\in
\mathbb{N}^{l}}\frac{1}{\sqrt{-1}^{l(\mu)+|\mu|}}\sum_{i=1}^{l}\sum_{\alpha+\beta=\mu_{i}}\alpha\beta
(\mathcal{C}_{g-1}(\mu(\alpha,\hat{i});\tau)\\\nonumber
&+\sum_{\substack{g_{1}+g_{2}=g\\\nu_{1}\coprod\nu_{2}=\mu(\alpha,\hat{i})\\2g_{1}-2+|\nu_{1}|>0\\
2g_{2}-2+|\nu_{2}|>0}}\mathcal{C}_{g_{1}}(\nu_{1};\tau)\mathcal{C}_{g_2}(\nu;\tau))e^{-(\mu_{1}w_{1}+\cdots+\mu_{l}w_{l})}\\\nonumber
&=\frac{(\tau^2+\tau)^{l-1}}{2}\sum_{i=1}^{l}\sum_{\substack{a_{1}\geq
0\\a_{2}\geq 0\\b_{L\setminus\{ k\}} \geq
0}}((\tau^2+\tau)\langle\tau_{a_{1}}\tau_{a_{2}}\tau_{b_{L\setminus
\{k\}}}\Gamma_{g-1}(\tau)\rangle_{g-1}\\\nonumber
&-\sum_{\substack{g_{1}+g_{2}=g\\ \mathcal{I}\coprod\mathcal{J}=\{1,.,\hat{i},.l\}\\2g_{1}-1+|\mathcal{I}|>0\\
2g_{2}-1+|\mathcal{J}|>0}}\langle\tau_{a_{1}}\tau_{b_{\mathcal{I}}}\Gamma_{g_{1}}(\tau)\rangle_{g_{1}}
\langle\tau_{a_{2}}\tau_{b_{\mathcal{J}}}\Gamma_{g_{2}}(\tau)\rangle_{g_{2}})
\prod_{n=1}^{2}\hat{\Psi}_{a_{n}+1}(t_{i};\tau)\hat{\Psi}_{b_{L\setminus\{i\}}}(t_{L\setminus\{i\}};\tau)
\end{align}

The unstable geometry in the cut-term of RHS has two terms for
$l\geq 2$.
\begin{align*}
U_{1}=\frac{\sqrt{-1}}{2}\sum_{i=1}^{l}\sum_{\alpha+\beta=\mu_{i}}\alpha\beta\left(\mathcal{C}_{0}(\alpha;\tau)\cdot
\frac{\mathcal{C}_{g}(\mu_{\hat{i}},\beta;\tau)}{|Aut(\mu_{\hat{i}})|}+
\frac{\mathcal{C}_{g}(\mu_{\hat{i}},\alpha;\tau)}{|Aut(\mu_{\hat{i}},\alpha)|}\cdot\mathcal{C}_{0}(\beta;\tau)
\right)
\end{align*}
\begin{align*}
U_{2}=\frac{\sqrt{-1}}{2}\sum_{i=1}^{l}\sum_{\alpha+\beta=\mu_{i}}\alpha\beta\sum_{j\neq
i}\left(\frac{\mathcal{C}_{0}(\mu_{j},\alpha;\tau)}{|Aut(\mu_{j},\alpha)|}\cdot
\frac{\mathcal{C}_{g}(\mu_{\hat{i},\hat{j}},\beta;\tau)}{|Aut(\mu_{\hat{i},\hat{j}})|}+
\frac{\mathcal{C}_{g}(\mu_{\hat{i},\hat{j}},\alpha;\tau)}{|Aut(\mu_{\hat{i},\hat{j}},\alpha)|}\cdot\mathcal{C}_{0}(\beta;\tau)
\right)
\end{align*}

As we have calculated in the proof of proposition 3.1, the Laplace
transform of $U_{1}$ is
\begin{align}
\sum_{i=1}^{l}\frac{\partial y_{l}}{\partial
\tau}t_{i}^2(\tau+1)\frac{\partial}{\partial
t_{i}}\hat{\mathcal{C}}_{g,l}(t_{1},..,t_{l};\tau)
\end{align}
Moving the formula $(13)$ to left hand side, it will cancel the
first term of formula $(11)$, i.e.
\begin{align}
&(11)-(13)=\left(\frac{\partial}{\partial
\tau}+\sum_{i=1}^{l}\frac{t_{i}^2-t_{i}}{\tau+1}\frac{\partial}{\partial
t_{i}}\right)\hat{\mathcal{C}}_{g,l}(t_{1},..,t_{l};\tau)\\\nonumber
&=-(\tau^2+\tau)^{l-2}\sum_{b_{L}\geq
0}\left((l-1)(2\tau+1)\langle\tau_{b_{L}}\Gamma_{g}(\tau)\rangle_{g}+
(\tau^{2}+\tau)\langle\tau_{b_{L}}\frac{d}{d\tau}\Gamma_{g}(\tau)\rangle_{g}\right)
\hat{\Psi}_{b_{L}}(t_{L};\tau)\\\nonumber
&-(\tau^{2}+\tau)^{l-1}\sum_{b_{L}\geq
0}\langle\tau_{b_{L}}\Gamma_{g}(\tau)\rangle_{g}\sum_{i=1}^{l}\left(\frac{\partial}{\partial
\tau}\hat{\Psi}_{b_{i}}(t_{i};\tau)+\frac{1}{t_{i}\tau+1}\hat{\Psi}_{b_{i}+1}(t_{i};\tau)\right)\hat{\Psi}_{b_{L\setminus
\{i\}}}(t_{L\setminus \{i\}};\tau)
\end{align}

The Laplace transform of $U_{2}$ equals to
\begin{align}
&\sum_{\mu\in\mathcal{P}_{l}}\frac{1}{\sqrt{-1}^{l+|\mu|}}U_{2}e^{-(\mu_{1}w_{1}+\cdots+\mu_{l}w_{l})}\\\nonumber
&=-\sum_{i=1}^{l}\sum_{j\neq i}\sum_{\substack{\mu_{j}\geq 1\\
\alpha\geq
1}}\frac{1}{\sqrt{-1}^{2+\mu_{j}+\alpha}}\alpha\mathcal{C}_{0}(\mu_{j},\alpha;\tau)e^{-\mu_{j}w_{j}}e^{-\alpha
w_{i}}\\\nonumber &\times \sum_{\substack{\beta\geq 0\\ \mu_{k}\geq
0\\k\neq
i,j}}\frac{1}{\sqrt{-1}^{l+|\mu|-\alpha-1}}\beta\mathcal{C}_{g}(\mu_{\hat{i},\hat{j}},\beta;\tau)e^{-\beta
w_{i}}e^{-\sum_{k\neq i,j}\mu_{k}w_{k}}\\\nonumber
&=\sum_{i=1}^{l}\sum_{j\neq i}\left(\frac{\partial}{\partial
w_{i}}\hat{\mathcal{C}}_{0,2}(w_{i},w_{j};\tau)(\tau^2+\tau)^{l-2}\sum_{\substack{a\geq
0\\b_{k}\geq 0, k\neq i,j}}\langle\tau_{a}\prod_{k\neq
i,j}\tau_{b_{k}}\Gamma_{g}(\tau)\rangle_{g}\hat{\Psi}_{a+1}(t_{i};\tau)\prod_{k\neq
i,j}\hat{\Psi}_{b_{k}}(t_{k};\tau)\right)\\\nonumber
&=\sum_{i=1}^{l}\sum_{j\neq
i}\left(-\frac{y_{i}}{y_{i}-y_{j}}(1+\frac{\tau
y_{j}}{1-(\tau+1)y_{i}})+\frac{x_{i}}{x_{i}-x_{j}}\right)\\\nonumber
&\cdot(\tau^2+\tau)^{l-2}\sum_{\substack{a\geq 0\\b_{k}\geq 0, k\neq
i,j}}\langle\tau_{a}\prod_{k\neq
i,j}\tau_{b_{k}}\Gamma_{g}(\tau)\rangle_{g}\hat{\Psi}_{a+1}(t_{i};\tau)\prod_{k\neq
i,j}\hat{\Psi}_{b_{k}}(t_{k};\tau)\\\nonumber
&=(\tau^2+\tau)^{l-2}\sum_{1\leq i<j\leq l}\sum_{\substack{a\geq
0\\b_{k}\geq 0, k\neq i,j}}\langle\tau_{a}\prod_{k\neq
i,j}\tau_{b_{k}}\Gamma_{g}(\tau)\rangle_{g}\prod_{k\neq
i,j}\hat{\Psi}_{b_{k}}(t_{k};\tau)(\frac{x_{i}}{x_{i}-x_{j}}\hat{\Psi}_{a+1}(t_{i};\tau)\\\nonumber
&-\frac{x_{j}}{x_{i}-x_{j}}
\hat{\Psi}_{a+1}(t_{j};\tau)-\frac{y_{i}}{y_{i}-y_{j}}(1+\frac{\tau
y_{j}}{1-(\tau+1)y_{i}})\hat{\Psi}_{a+1}(t_{i};\tau)+\frac{y_{j}}{y_{i}-y_{j}}(1+\frac{\tau
y_{i}}{1-(\tau+1)y_{j}})\hat{\Psi}_{a+1}(t_{j};\tau))
\end{align}
where we have used the lemma 3.3.

At last, we need to calculate the Laplace transform of join term in
RHS. It's Laplace transform is
\begin{align}
&\sqrt{-1}\sum_{\mu\in
\mathbb{N}^{l}}\sum_{i<j}\frac{1}{\sqrt{-1}^{l(\mu)+|\mu|}}(\mu_{i}+\mu_{j})\mathcal{C}_{g}(\mu(\hat{i},\hat{j});\tau)
e^{-\mu_{i}w_{i}}e^{-\mu_{j}w_{j}}e^{-\sum_{k\neq
i,j}\mu_{k}w_{k}}\\\nonumber &=-(\tau^2+\tau)^{l-2}\sum_{\mu\in
\mathbb{N}^{l}}\sum_{i<j}\sum_{\substack{a\geq
0\\b_{L\setminus\{i,j\}}\geq
0}}\langle\tau_{a}\tau_{b_{L\setminus\{i,j\}}}\Gamma_{g}(\tau)\rangle_{g}\prod_{k\neq
i,j}\frac{1}{\tau}\frac{\prod_{a=0}^{\mu_{k}-1}(\mu_{k}\tau+a)}{\mu_{k}!}\mu_{k}^{b_{k}}e^{-\sum_{k\neq
i,j}\mu_{k}w_{k}}\\\nonumber
&\cdot\frac{1}{\tau}\frac{\prod_{a=0}^{\mu_{i}+\mu_{j}-1}((\mu_{i}+\mu_{j})\tau+a)}{(\mu_{i}+\mu_{j})!}
(\mu_{i}+\mu_{j})^{a+1}e^{-\mu_{i}w_{i}}
e^{-\mu_{j}w_{j}}\\\nonumber &=-(\tau^2+\tau)^{l-2}\sum_{1\leq
i<j\leq l}\sum_{\substack{a\geq 0\\b_{L\setminus\{i,j\}\geq
0}}}\langle\tau_{a}\tau_{b_{L\setminus\{i,j\}}}\Gamma_{g}(\tau)\rangle_{g}\hat{\Psi}_{b_{L\setminus\{i,j\}}}(t_{L\setminus\{i,j\}};\tau)
\left(\frac{x_{i}}{x_{i}-x_{j}}\hat{\Psi}_{a+1}(t_{i};\tau)\right.\\\nonumber
&\left.-\frac{x_{j}}{x_{i}-x_{j}}\hat{\Psi}_{a+1}(t_{j};\tau)-\hat{\Psi}_{a+1}(t_{i};\tau)-\hat{\Psi}_{a+1}(t_{j};\tau)\right)
\end{align}

Combining (15) and (16), we get
\begin{align}
&(10)+(11)=-\frac{(\tau^2+\tau)^{l-2}}{\tau+1}\sum_{1\leq i<j\leq
l}\sum_{\substack{a\geq 0\\b_{L\setminus\{i,j\}}\geq
0}}\langle\tau_{a}\tau_{b_{L\setminus\{i,j\}}}\Gamma_{g}(\tau)\rangle_{g}
\hat{\Psi}_{b_{L\setminus\{i,j\}}}(t_{L\setminus\{i,j\}};\tau)\\\nonumber
&\cdot\frac{(t_{j}-1)(t_{i}^2\tau+t_{i})\hat{\Psi}_{a+1}(t_{i};\tau)-(t_{i}-1)(t_{j}^2\tau+t_{j})\hat{\Psi}_{a+1}(t_{j};\tau)}
{t_{i}-t_{j}}
\end{align}
Collecting the remainder equations  $(12)$,$(14)$ and $(17)$, we
obtain the equation $(10)$ of theorem 3.4.
\end{proof}
\section{Application I: Proof of the Bouchard-Mari\~no conjecture for $\mathbb{C}^{3}$}
Motivated by Eynard and his collaborators' series works on Matrix
model \cite{Ey1,EO1,EMO,EO2} , Bourchard, Klemm, Mari\~no and
Pasquetti propose a new approach to compute the topological string
amplitudes for local Calabi-Yau manifolds \cite{BKMP}. Based on such
proposal, Bouchard and Mari\~no calculated the string amplitude for
toric three fold $\mathbb{C}^3$ \cite{BM}. In this section, we will
follow the work of \cite{EMS} and \cite{Ch2} to illustrate that
Bouchard-Mar\~no's approach is implied by theorem 1.1.

Following the work of \cite{EMS}, in subsection 4.1, we illustrate
Bourchard-Mari\~no's proposal of calculation the string amplitude of
$\mathbb{C}^3$( which is also called Bouchard-Mari\~no conjecture
for $\mathbb{C}^{3}$ ). By some residue computations, we obtain a
equivalent description of Bourchard-Mar\~no's approach in subsection
4.2. The calculation in subsections 4.3 and 4.4 to prove the
Bouchard-Mari\~no conjecture for $\mathbb{C}^{3}$ is firstly done by
L. Chen \cite{Ch2}

\subsection{Bouchard-Mari\~no conjecture for $\mathbb{C}^{3}$}
Here we consider the spectral curve
\begin{align}
C: y(1-y)^{\tau}=x
\end{align}
and the variable $t=\frac{1}{1-(\tau+1)y}$.

It is easy to see that, variables $x,y,t$ have the relations that
\begin{align*}
x\frac{d}{dx}=\frac{(1-y)y}{1-(\tau+1)y}\frac{d}{dy}=\left(\frac{t(t-1)(t\tau+1)}{\tau+1}\right)\frac{d}{dt}
\end{align*}
We defined in section 2 that,
\begin{align*}
\hat{\Psi}_{n}(t;\tau)=\left(\left(\frac{t(t-1)(t\tau+1)}{\tau+1}\right)\frac{d}{dt}
\right)^{n}\frac{t-1}{\tau+1}
\end{align*}
for $n\geq 0$.

In the following state, we need to introduce the differential,
\begin{align*}
\Psi_{n}(t;\tau)=d\hat{\Psi}(t;\tau)
\end{align*}
which can be formulated in variables $x,y$ as
\begin{align}
\Psi_{n}(y;\tau)=dy\frac{1-(\tau+1)y}{y(1-y)}\left(\frac{y(1-y)}{1-(\tau+1)y}\frac{d}{dy}\right)^{n+1}\frac{1}{(\tau+1)(1-(\tau+1)y)}
\end{align}
\begin{align}
\Psi_{n}(x;\tau)=\frac{1}{\tau}\sum_{k\geq
1}\frac{\prod_{a=0}^{k-1}(k\tau+a)}{k!}k^{n+1}x^{k-1}dx
\end{align}
respectively.

In \cite{BM}, in order to formulate their conjecture, Bourchard and
Mari\~no introduce
\begin{align}
W_{g}(x_{1},..,x_{l};\tau)=\sum_{\mu,l(\mu)=l}z_{\mu}W_{g,\mu}(\tau)\frac{1}{|Aut(\mu)|}\sum_{\sigma\in
S_{l}}\prod_{i=1}^{l}x_{\sigma(i)}^{\mu_{i}-1}
\end{align}
where
$W_{g,\mu}(\tau)=-\frac{1}{\sqrt{-1}^{g+l}}\mathcal{C}_{g,\mu}(\tau)$.

By $(3)$,$(19)$, $(20)$ and $(21)$, we have
\begin{align*}
W_{g}(x_{1},..,x_{l};\tau)dx_{1}\cdots
dx_{l}=(-1)^{g+l}(\tau^2+\tau)^{l-1}\sum_{b_{L}\geq
0}\langle\tau_{b_{L}}\Gamma_{g}(\tau)\rangle_{g}\Psi_{b_{L}}(y_{L};\tau)
\end{align*}
\begin{definition}
Let us call the symmetric polynomial differential form
\begin{align*}
d^{\otimes
l}\hat{\mathcal{C}}_{g,l}(t_{1},..,t_{l};\tau)=-(\tau(\tau+1))^{l-1}\sum_{b_{L}\geq
0}\langle\tau_{b_{L}}\Gamma_{g}(\tau)\rangle_{g}\Psi_{b_L}(t_L;\tau)
\end{align*}
on $C^{l}$ the Mari\~no-Vafa differential of type $(g,l)$.
\end{definition}
Thus
\begin{align*}
W_{g}(x_{1},..,x_{l};\tau)dx_{1}\cdots dx_{l}=(-1)^{g+l-1}d^{\otimes
l}\hat{\mathcal{C}}_{g,l}(t_{1},..,t_{l};\tau)
\end{align*}

$t=\frac{1}{1-(\tau+1)y}$, it then follows that
$y=\frac{t-1}{(\tau+1)t}$. Substituting to the spectral curve
$(18)$, we have $
\frac{t-1}{t}\left(\frac{t\tau+1}{t}\right)^{\tau}\frac{1}{(\tau+1)^{\tau+1}}=x
$. Let $s(t)$ be the solution of function equation
\begin{align}
\frac{t-1}{t}\left(\frac{t\tau+1}{t}\right)^{\tau}=\frac{s(t)-1}{s(t)}\left(\frac{s(t)\tau+1}{s(t)}\right)^{\tau}
\end{align}
From $(22)$, we have
\begin{align}
ln\left(\frac{t\tau+1}{t}\right)-ln\left(\frac{s(t)\tau+1}{s(t)}\right)=\frac{1}{\tau}\left(ln\left(\frac{s(t)-1}{s(t)}\right)
-ln\left(\frac{t-1}{t}\right)\right)
\end{align}

It is clear that the spectral curve $C:y(1-y)^{\tau}=x$ has only one
ramification point of the $x$-projection, which we denote by $\nu$.
It is given by $y(\nu)=\frac{1}{\tau+1}$. Then, we can find two
points $q$ and $\overline{q}$ on the curve such that
$x(q)=x(\overline{q})$ near the ramification point $\nu$. Let us
write
\begin{align*}
y(q)=\frac{1}{\tau+1}-z,\quad y(\overline{q})=\frac{1}{\tau+1}-P(z)
\end{align*}
with
\begin{align*}
P(z)=-z+\mathcal {O}(z^2)
\end{align*}
From functional equation
$y(q)(1-y(q))^{\tau}=y(\overline{q})(1-y(\overline{q}))^{\tau}$, we
can solve exact $P(z)$ as a power series in $z$,
\begin{align*}
P(z)=-z-\frac{2(-1+\tau^{2})}{3\tau}z^{2}-\frac{4(-1+\tau^{2})^{2}}{9\tau^{2}}z^{3}-\frac{2(1+\tau)^3(-22+57\tau-57\tau^{2}+22\tau^3)}
{135\tau^{3}}z^4+\cdots
\end{align*}
Note that $P(z)$ is an involution $P(P(z))=z$.

In terms of the t-coordinate, the involution $P(z)$ corresponds to
$s(t)$ in $(22)$. We have,

\begin{equation} \label{eq:1}
\left\{ \begin{aligned}
         t(q) &=\frac{1}{1-(\tau+1)y(q)}=\frac{1}{\tau+1}\frac{1}{z}=t \\
                  t(\overline{q})&=\frac{1}{1-(\tau+1)y(\overline{q})}=\frac{1}{\tau+1}\frac{1}{P(z)}=s(t)
                          \end{aligned} \right.
                          \end{equation}
It follows that
\begin{align*}
\frac{(\tau+1)dt}{(t^2-t)(t\tau+1)}=\frac{(\tau+1)ds(t)}{s(t)^2(s(t)-1)}=\frac{dx}{x}
\end{align*}
Using the coordinate $t$ of spectral curve $C$, the Bergman kernel
is defined by
\begin{align*}
B(t_{1},t_{2})=\frac{dt_{1}dt_{2}}{(t_{1}-t_{2})}
\end{align*}
Define a 1-form on $C$ by,
\begin{align*}
dE(q,\overline{q};t_{2})=\frac{1}{2}\int_{q}^{\overline{q}}B(\cdot,t_{2})=\frac{1}{2}\left(
\frac{1}{t_{1}-t_{2}}-\frac{1}{s(t_{1})-t_{2}}\right)dt_{2}
\end{align*}
where the integral is taken with respect to the first variable of
$B(t_{1},t_{2})$ along any path from $q$ to $\overline{q}$. The
natural holomorphic symplectic form on $\mathbb{C}\times\mathbb{C}$
is given by $ \Omega=dln(1-y)\wedge ln(x) $.

Again, let us introduce another 1-form on the curve $C$ by,
\begin{align*}
\omega(q,\overline{q})=\int_{q}^{\overline{q}}\Omega(\cdot,x)&=\left(ln(1-y(\overline{q}))-ln(1-y(q))\right)\frac{dx}{x}\\
&=\frac{1}{\tau}(lny(q)-lny(\overline{q}))\frac{dx}{x}\\
&=\frac{1}{\tau}\left(ln\left(1-\frac{1}{t}
\right)-ln\left(1-\frac{1}{s(t)} \right)
\right)\frac{(\tau+1)}{t^2-t}\frac{dt}{(t\tau+1)}
\end{align*}
The kernel operator is defined as the quotient
\begin{align*}
K(t_{1},t_{2};\tau)=\frac{dE(q,\overline{q},t_{2})}{\omega(q,\overline{q})}=\frac{\tau}{2(\tau+1)}
\frac{(s(t_{1})-t_{1})(t_{1}^{2}-t_{1})(t_{1}\tau+1)}
{(t_{1}-t_{2})(s(t_{1})-t_{2})(ln(1-\frac{1}{t_{1}})-ln(1-\frac{1}{s(t_{1})}))}\frac{1}{dt_{1}}dt_{2}
\end{align*}
It is clear that $K(s(t_{1}),t_{2};\tau)=K(t_{1},t_{2};\tau)$.
\begin{definition}
The topological recursion formula is an inductive mechanism of
defining a symmetric $l$-form $W_{g,l}(t_{1},..,t_{l};\tau)$ on
$C^{l}$ for $(g,l)$ subject to $2g-2+l>0$ by

\begin{multline*}
W_{g,l+1}(t_{0},t_{L};\tau)=-\frac{1}{2\pi i}\oint
_{\gamma_{\infty}}\left[
K(t,t_{0};\tau)\left(W_{g-1,l+2}(t,s(t),t_{L};\tau)
+\right.\right.\\
\left.\left.\sum_{i=1}^{l}\left(W_{g,l}(t,t_{L \setminus \{i\}})
B(s(t),t_{i})+W_{g,l}(s(t),t_{L\setminus \{i\}};\tau)
B(t,t_{i})\right)\right.\right.\\
\left.\left.+\sum_{\substack{g_1+g_2=g,\mathcal{I}\cup\mathcal{J}=L}}^{stable}W_{g_1,|\mathcal{I}|+1}(t,t_{\mathcal{I}};\tau)
W_{g_2,|\mathcal{J}|+1}(s(t),t_{\mathcal{J}};\tau)\right) \right]
\end{multline*}
where $t_{I}=(t_{i})_{i\in I}$ for a subset $I\subset L$, and the
last sum is taken over all partitions of $g$ and disjoint
decompositions $I\coprod J=L$ subject to the stability condition
$2g_{i}-1+|I|>0$ and $2g_{2}-1+|J|>0$. The integration is taken with
respect to $dt$ on the contour $\gamma_{\infty}$, which is a
positively oriented loop in the complex $t$-plane of large radius
such that $|t|>max(|t_{0}|,s(t_{0}))$ for $t\in \gamma_{\infty}$.
\end{definition}
\begin{remark}
The original topological recursion formula for
$W_{g,l}(t_1,..,t_{l};\tau)$ was formulated by taken the residue at
$z=0$ \cite{BM}. But here, for simplicity, we have changed the
coordinate from $z$ to $t$ by $t=\frac{1}{(\tau+1)z}$. Hence, we get
the contour integral in definition 4.2.
\end{remark}

Now, we can formulate the Bouchard-Mari\~no conjecture for
$\mathbb{C}^3$.
\begin{conjecture}
For every $g$ and $l$ subject to the stability condition $2g-2+l>0$,
the topological recursion formula with the initial condition

\begin{equation} \label{eq:1}
\left\{ \begin{aligned}
         W_{0,3}(t_{1},t_{2},t_{3};\tau) &=d^{\otimes 3}\hat{\mathcal{C}}_{0,3}(t_{1},t_{2},t_{3};\tau)
         =-\frac{\tau^{2}}{(\tau+1)}dt_{1}dt_{2}dt_{3} \\
                  W_{1,1}(t_{1};\tau)&=-d\hat{\mathcal{C}}_{1,1}(t_{1};\tau)=\frac{1}{24}((1+\tau+\tau^2)\Psi_{0}(t_1;\tau)
                  -\tau(\tau+1)\Psi_{1}(t_1;\tau))
                          \end{aligned} \right.
                          \end{equation}
gives the Mari\~no-Vafa differential with a signature.
\begin{align}
W_{g,l}(t_{1},..,t_{l};\tau)=(-1)^{g+l-1}d^{\otimes
l}\hat{\mathcal{C}}_{g,l}(t_{1},..,t_{l};\tau)
\end{align}

\end{conjecture}
\subsection{Residue calculation}
In this subsection, we calculate the residues(contour integrals)
appearing definition 4.2. In fact, we have to calculate the
following two types of residues:

i)$R_{a,b}(t;\tau)=-\frac{1}{2\pi
i}\oint_{\gamma_{\infty}}K(t',t;\tau)\Psi_{a}(t';\tau)\Psi_{b}(s(t');\tau)
$

ii)$R_{n}(t,t_{i};\tau)=-\frac{1}{2\pi
i}\oint_{\gamma_{\infty}}K(t',t;\tau)\left(\Psi_{n}(t';\tau)B(s(t'),t_{i})+\Psi_{n}(s(t');\tau)B(t',t_{i})
\right)$

We need one lemma first,
\begin{lemma}
In the $z$-coordinate, the kernel $K(t'(z),t;\tau)$ has the
expansion,
\begin{align*}
&K=\left(\frac{\tau^2}{2(\tau+1)^3}z^{-1}-\frac{(\tau-1)}{2(\tau+1)^2}+\frac{3\tau^2t^2+2\tau(1-\tau)t-4\tau}{6(\tau+1)}z\right.\\
&\left.+\left(
\frac{\tau-\tau^2}{6}t^2+\frac{\tau^2-2\tau+1}{9}t+\frac{(\tau+1)}{18}
\right)z^2+\cdots \right)dt\frac{1}{dz}
\end{align*}
In particular, the coefficients of $z$ in above expansion are
polynomials of $t$.
\end{lemma}
\begin{proof}
Substituting $t'=\frac{1}{(\tau+1)z}$ and
$s(t')=\frac{1}{(\tau+1)s(z)}$ to $K(t',t;\tau)$ and calculate
directly, we obtain lemma 4.5.
\end{proof}
\begin{definition}
For a laurent series $\sum_{n\in \mathbb{Z}}a_{n}t^{n}$, we denote
\begin{align*}
\left[\sum_{n\in \mathbb{Z}}a_{n}t^n\right]_+=\sum_{n>0}a_{n}t^{n}
\end{align*}

\end{definition}
Now, we have two theorems for the calculation of $R_{a,b}(t;\tau)$
and $R_{n}(t,t_{i};\tau)$.

\begin{theorem}
\begin{align*}
&R_{a,b}(t;\tau)=-\frac{1}{2\pi
i}\oint_{\gamma_{\infty}}K(t',t;\tau)\Psi_{a}(t';\tau)\Psi_{b}(s(t');\tau)\\
&=-\frac{\tau}{2}\left(\frac{1}{ln\left(1-\frac{1}{t}\right)-ln\left(1-\frac{1}{s(t)}\right)}
\left(\Psi_{a}(t;\tau)\hat{\Psi}_{b+1}(s(t);\tau)+\hat{\Psi}_{a+1}(s(t);\tau)\Psi_{b}(t;\tau)\right)
\right)_{+}
\end{align*}
\end{theorem}
\begin{proof}
We note that, in term of the original $z$-coordinate of \cite{BM},
the residue $R_{a,b}(t;\tau)$ is the coefficient of $z^{-1}$ in
$K(t',t;\tau)\Psi_{a}(t';\tau)\Psi_{b}(s(t');\tau)$, after expanding
it in the Laurent series in z. $\Psi_{n}(t';\tau)$ is a polynomial
in $t'=\frac{1}{(\tau+1)z}$, thus the contribution to the $z^{-1}$
term in the expression is a polynomial in $t$ by lemma 4.5. Hence
$R_{a,b}(t;\tau)$ is a polynomial in $t$.

Let us write $\Psi_{n}(t';\tau)=f_{n}(t';\tau)dt'$. Let
$\gamma_{\epsilon}$ be a positively oriented loop, such that
function
\begin{align}
\frac{(t'^{2}-t')(t'\tau)+1}{ln\left(1-\frac{1}{t'}\right)-ln\left(1-\frac{1}{s(t')}\right)}s'(t')
f_{a}(t';\tau)f_{b}(s(t');\tau)
\end{align}
has no singularity on and outside loop $\gamma_{\epsilon}$. On the
compact set $\gamma_{\epsilon}$ we have a bound $M_{\epsilon}$,
\begin{align*}
\mid\frac{(t'^{2}-t')(t'\tau)+1}{ln\left(1-\frac{1}{t'}\right)-ln\left(1-\frac{1}{s(t')}\right)}s'(t')
f_{a}(t';\tau)f_{b}(s(t');\tau)\mid<M_{\epsilon}
\end{align*}
Choose $|t|>>1$. Then,
\begin{align*}
&\mid\frac{1}{2\pi
i}\oint_{\gamma_{\epsilon}}K(t',t;\tau)\Psi(t';\tau)\Psi(s(t');\tau)\mid\\
&=\mid\frac{1}{2\pi
i}\oint_{\gamma_{\epsilon}}\frac{\tau}{2(\tau+1)}\left(\frac{1}{t'-t}-\frac{1}{s(t')-t}\right)
\frac{(t'^{2}-t')(t'\tau+1)}{ln\left(1-\frac{1}{t'}\right)-ln\left(1-\frac{1}{s(t')}\right)}s'(t')
f_{a}(t';\tau)f_{b}(s(t');\tau) dt'\mid dt\\
&<\frac{\tau M_{\epsilon}}{4\pi
(\tau+1)}\oint_{\gamma_{\epsilon}}\mid\frac{1}{t'-t}-\frac{1}{s(t')-t}
\mid dt'dt \sim \frac{\tau M_{\epsilon}}{4\pi (\tau+1)|t|}dt
\end{align*}
Thus, we have
\begin{align*}
\frac{1}{2\pi
i}\oint_{\gamma_{\infty}}K(t',t;\tau)\Psi_{a}(t';\tau)\Psi_{b}(s(t');\tau)&=\frac{1}{2\pi
i}\oint_{\gamma_{\infty}-\gamma_{\epsilon}}K(t',t;\tau)\Psi_{a}(t';\tau)\Psi_{b}(s(t');\tau)+O(t^{-1})\\
&=\left[\frac{1}{2\pi
i}\oint_{\gamma_{\infty}-\gamma_{\epsilon}}K(t',t;\tau)\Psi_{a}(t';\tau)\Psi_{b}(s(t');\tau)\right]_{+}
\end{align*}
\begin{align*}
&\frac{1}{2\pi
i}\oint_{\gamma_{\infty}-\gamma_{\epsilon}}K(t',t;\tau)\Psi_{a}(t';\tau)\Psi_{b}(s(t');\tau)\\
&=\frac{\tau}{4(\tau+1)\pi
i}\oint_{\gamma_{\infty}-\gamma_{\epsilon}}\left(\frac{1}{t'-t}-\frac{1}{s(t')-t}\right)
\frac{(t'^2-t')(t'\tau+1)}{ln\left(1-\frac{1}{t'}\right)-ln\left(1-\frac{1}{s(t')}\right)}s'(t')f_{a}(t';\tau)f_{b}(s(t');\tau)dt'dt\\
&=\frac{\tau}{2(\tau+1)}\frac{(t^2-t)(t\tau+1)}{ln\left(1-\frac{1}{t}\right)-ln\left(1-\frac{1}{s(t')}\right)}s'(t')
f_{a}(t;\tau)f_{b}(s(t);\tau)dt\\
&+\frac{\tau}{2(\tau+1)}\frac{(s(t)^{2}-s(t))(s(t)\tau+1)}{ln\left(1-\frac{1}{t}
\right)-ln\left(1-\frac{1}{s(t)}\right)}f_{a}(s(t);\tau)f_{b}(t;\tau)dt\\
&=\frac{\tau}{2(\tau+1)}\frac{(t^2-t)(t\tau+1)}{ln\left(1-\frac{1}{t}
\right)-ln\left(1-\frac{1}{s(t)}
\right)}s'(t)\left(f_{a}(t;\tau)f_{b}(s(t);\tau)+f_{a}(s(t);\tau)f_{b}(t;\tau)\right)dt\\
&=\frac{\tau}{2}\frac{1}{ln\left(1-\frac{1}{t}\right)-ln\left(1-\frac{1}{s(t)}\right)}\left(\Psi_{a}(t;\tau)\hat{\Psi}_{b+1}(s(t);\tau)
+\Psi_{a+1}(s(t);\tau)\hat{\Psi}_{b}(t;\tau)\right)
\end{align*}
\end{proof}

Similarly, with the same proof as showed above, we have
\begin{theorem}
\begin{align*}
R_{n}(t,t_{i};\tau)&=-\frac{1}{2\pi
i}\oint_{\gamma_{\infty}}K(t',t;\tau)\left(\Psi_{n}(t';\tau)B(s(t'),t_{i})+\Psi_{n}(s(t');\tau)B(t';t_{i})\right)\\
&=-\tau\left[\frac{1}{ln\left(1-\frac{1}{t}\right)-ln\left(1-\frac{1}{s(t)}\right)}\left(
\hat{\Psi}_{n+1}(t;\tau)B(s(t),t_{i})+\hat{\Psi}_{n+1}(s(t);\tau)B(t,t_{i})
\right)\right]_{+}
\end{align*}
\end{theorem}
Let us define polynomials $P_{a,b}(t;\tau)$ and $P_{n}(t,t_i;\tau)$
by
\begin{align}
&P_{a,b}(t;\tau)dt=-\frac{\tau(\tau+1)}{2}\left[\frac{1}{ln\left(1-\frac{1}{t}
\right)-ln\left(1-\frac{1}{s(t)}\right)}\left(\hat{\Psi}_{a+1}(t;\tau)\hat{\Psi}_{b+1}(s(t);\tau)
\right.\right.\\\nonumber
&\left.\left.+\hat{\Psi}_{a+1}(s(t);\tau)\hat{\Psi}_{b+1}(t;\tau)\right)\frac{dt}{(t^2-t)(t\tau+1)}\right]_+
\end{align}
\begin{align}
P_{n}(t,t_{i};\tau)dtdt_{i}=-\tau\left[\frac{1}{ln\left(1-\frac{1}{t}\right)-ln\left(1-\frac{1}{s(t)}\right)}\left(
\hat{\Psi}_{n+1}(t;\tau)B(s(t),t_{i})+\hat{\Psi}_{n+1}(s(t);\tau)B(t,t_{i})
\right)\right]_{+}
\end{align}

With the above notations, the Bourchard-Mari\~no conjecture for
$\mathbb{C}^{3}$ is equivalent to
\begin{align}
&(\tau^2+\tau)^{l-1}\sum_{b_L\geq
0}\langle\tau_{b_L}\Gamma_{g}(\tau)\rangle_{g}d\hat{\Psi}_{b_{L}}(t_L;\tau)\\\nonumber
&=-(\tau^2+\tau)^{l-2}\sum_{i=2}^{l}\sum_{a,b_{L\setminus
\{1,i\}}}\langle\tau_{a}\tau_{b_{L\setminus\{1,i\}}}\Gamma_{g}(\tau)
\rangle_{g}P_{a}(t_{1},t_{i};\tau)dt_{1}dt_{i}d\hat{\Psi}_{b_{L\setminus\{1,i\}}}(t_{L\setminus
\{1,i\}};\tau)\\\nonumber &+\sum_{\substack{a_1,a_2\geq 0
\\b_{L\setminus \{1\}}}}(\tau^2+\tau)^{l-1}\left((\tau^2+\tau)\langle\tau_{a_1}\tau_{a_2}\tau_{b_{L\setminus\{1\}}}\Gamma_{g-1}(\tau)\rangle_{g-1}
\right.\\\nonumber
&\left.-\sum_{\substack{g_1+g_2=g\\
\mathcal{I}\bigsqcup\mathcal{J}=L\setminus\{i\}}}^{stable}\langle
\tau_{a_1}\tau_{b_{I}}\Gamma_{g_1}(\tau)\rangle_{g_1,|I|+1}\langle
\tau_{a_2}\tau_{b_{J}}\Gamma_{g_2}(\tau)\rangle_{g_2,|J|+1}\right)
P_{a_1,a_2}(t_1;\tau)dt_1d\hat{\Psi}_{b_L\setminus\{1\}}(t_{L\setminus
\{1\}};\tau)
\end{align}

\subsection{Calculation in new variable $v$}
Now, let us introduce a new variable $v$ by
$x=e^{-w}=e^{-\frac{1}{2}v^2}$. As $y(1-y)^\tau=x$ and
$t=\frac{1}{1-(\tau+1)y}$, we have
\begin{align*}
\left(
\frac{t-1}{(\tau+1)t}\right)\left(1-\frac{t-1}{(\tau+1)t}\right)^{\tau}=e^{-\frac{1}{2}v^2}
\end{align*}
Thus
\begin{align*}
&v^2=-2\left(ln\left(1-\frac{1}{t}\right)+\tau
ln\left(1+\frac{1}{\tau t}\right)+\tau ln
(\tau)-(\tau+1)ln(\tau+1)\right)\\
&=2\left(\sum_{n\geq 2}\frac{1}{n}\left((-1)^n+\frac{1}{\tau^{n-1}}
\right)\frac{1}{t^n}+(\tau+1)ln(\tau+1)-\tau ln(\tau)\right)
\end{align*}

Solving the above functional equation,
\begin{align}
v=\frac{1}{t}\left(\sqrt{\frac{\tau+1}{\tau}}+\sum_{k\geq
1}a_{k}(\tau)\frac{1}{t^k}\right)
\end{align}
and
\begin{align}
-v=\frac{1}{s(t)}\left(\sqrt{\frac{\tau+1}{\tau}}+\sum_{k\geq
1}a_{k}(\tau)\frac{1}{s(t)^k}\right)
\end{align}
Taking the inverse of $(31)$, we have

\begin{align}
\frac{1}{t}=v\left(\sqrt{\frac{\tau}{\tau+1}}+\sum_{k\geq
1}b_{k}(\tau)v^{k}\right)
\end{align}
and
\begin{align}
\frac{1}{s(t)}=(-v)\left(\sqrt{\frac{\tau}{\tau+1}}+\sum_{k\geq
1}b_{k}(\tau)(-v)^{k}\right)
\end{align}

\begin{lemma}
Let
\begin{align*}
&\eta_{-1}(v;\tau)=\frac{1}{2}\left(\hat{\Psi}_{-1}(t;\tau)-\hat{\Psi}_{-1}(s(t);\tau)
\right)\\
&=-\frac{1}{2}\left(ln\left(\frac{t\tau+1}{t\tau}\right)-ln\left(\frac{s(t)\tau+1}{s(t)\tau}\right)\right)\\
&=\frac{1}{2\tau}\left(ln\left(1-\frac{1}{t}\right)-ln\left(1-\frac{1}{s(t)}\right)
\right)
\end{align*}
and define
 $\eta_{n+1}(v;\tau)=-\frac{1}{v}\frac{d}{dv}\eta_{n}(v;\tau)$ when
$n\geq -1$, then
\begin{align*}
\eta_{n}(v;\tau)=\frac{1}{2}\left(\hat{\Psi}_{n}(t;\tau)-\hat{\Psi}_{n}(s(t);\tau)
\right)
\end{align*}
Moreover, there exists formal series $F_{n}(w;\tau)$ such that
$\eta_{n}(v;\tau)=\hat{\Psi}_{n}(t;\tau)+F_{n}(w;\tau)$.
\end{lemma}
\begin{proof}
Let $\hat{\Psi}_{-1}(t;\tau)=-ln\left(\frac{t\tau+1}{t\tau}\right)$
be the solution of differential equation
\begin{align*}
\frac{(t\tau+1)(t^2-t)}{\tau+1}\frac{d}{dt}\hat{\Psi}_{-1}(t;\tau)=\hat{\Psi}_{0}(t;\tau)=\frac{t-1}{\tau+1}
\end{align*}
then
\begin{align*}
\eta_{-1}(v;\tau)=\frac{1}{2}\left(\hat{\Psi}_{-1}(t;\tau)-\hat{\Psi}_{-1}(s(t);\tau)\right)=-\frac{1}{2}\left(ln\left(
\frac{t\tau+1}{t\tau}\right)-ln\left(\frac{s(t)\tau+1}{s(t)\tau}\right)
\right)
\end{align*}
By definition,
\begin{align}
x\frac{d}{dx}=-\frac{d}{dw}=-\frac{1}{v}\frac{d}{dv}=\frac{(t\tau+1)(t^2-t)}{\tau+1}\frac{d}{dt}
=\frac{(s(t)\tau+1)(s(t)^2-s(t))}{\tau+1}\frac{d}{ds(t)}
\end{align}
Therefore
\begin{align*}
\eta_{n}(v;\tau)=\frac{1}{2}\left(\hat{\Psi}_{n}(t;\tau)-\hat{\Psi}(s(t);\tau)
\right)
\end{align*}
Moreover,
\begin{align}
&\eta_{-1}(v;\tau)-\hat{\Psi}_{-1}(t;\tau)\\\nonumber
&=-\frac{1}{2}\left(\hat{\Psi}_{-1}(t;\tau)+\hat{\Psi}_{-1}(s(t);\tau)
\right)\\\nonumber
&=\frac{1}{2}\left(ln\left(1+\frac{1}{t\tau}\right)+ln\left(1+\frac{1}{s(t)\tau}\right)
\right)\\\nonumber
&=\frac{1}{2}\left(\sum_{n\geq
1}\frac{(-1)^{n-1}}{n}\frac{1}{\tau^{n}}\left(
\frac{1}{t^n}+\frac{1}{s(t)^n}\right) \right)
\end{align}
By $(33)$ and $(34)$, we note that right hand side of $(36)$ is a
series of variable $w=\frac{1}{2}v^2$, we write it as
$F_{-1}(w;\tau)$. $(36)$ is equivalent to
\begin{align*}
\eta_{-1}(v;\tau)=\hat{\Psi}_{-1}(t;\tau)+F_{-1}(w;\tau)
\end{align*}
Let us define
\begin{align}
F_{n}(w;\tau)=\left(-\frac{d}{dw}\right)^{n+1}F_{-1}(w;\tau)
\end{align}
then, we have
\begin{align*}
\eta_{n}(v;\tau)=\hat{\Psi}_{n}(t;\tau)+F_{n}(w;\tau)
\end{align*}
\end{proof}

\begin{corollary}
\begin{align*}
P_{a,b}(t;\tau)dt=-\frac{1}{2}\left(\frac{\eta_{a+1}(v;\tau)\eta_{b+1}(v;\tau)}{\eta_{-1}(v;\tau)}
vdv|_{v=v(t)}\right)_{+}
\end{align*}
\end{corollary}
\begin{proof}
By $(23)$ and $(28)$, we have
\begin{align*}
&\frac{1}{2}\frac{1}{ln\left(1+\frac{1}{\tau
t}\right)-ln\left(1+\frac{1}{\tau
s(t)}\right)}\left(\hat{\Psi}_{a+1}(t;\tau)\hat{\Psi}_{b+1}(s(t);\tau)+\hat{\Psi}_{a+1}(s(t);\tau)\hat{\Psi}_{b+1}(t;\tau)
\right)\frac{(\tau+1)dt}{(t^2-t)(t\tau+1)}\\
&=\frac{1}{2\eta_{-1}(v;\tau)}\frac{(\tau+1)dt}{(t^2-t)(t\tau+1)}\left(\frac{(\hat{\Psi}_{a+1}(t;\tau)-\hat{\Psi}_{a+1}(s(t);\tau))
(\hat{\Psi}_{b+1}(t;\tau)-\hat{\Psi}_{b+1}(s(t);\tau))}{4}\right.\\
&\left.-\frac{(\hat{\Psi}_{a+1}(t;\tau)+\hat{\Psi}_{a+1}(s(t);\tau))
(\hat{\Psi}_{b+1}(t;\tau)+\hat{\Psi}_{b+1}(s(t);\tau))}{4}\right)
\\
&=-\frac{vdv}{2\eta_{-1}(v;\tau)}\left(\eta_{a+1}(v;\tau)\eta_{b+1}(v;\tau)-F_{a+1}(w;\tau)F_{b+1}(w;\tau)
\right)
\end{align*}
Noticing that
\begin{align*}
\left(\frac{F_{a+1}(w;\tau)F_{b+1}(w;\tau)vdv}{\eta_{-1}(v;\tau)}|_{v=v(t)}\right)_{+}=0
\end{align*}
we complete the proof.
\end{proof}

\subsection{Proof of the Bouchard-Mari\~no conjecture for $\mathbb{C}^{3}$}
In the proof of the theorem 3.4, we know that LHS of $(10)$ is equal
to
\begin{align*}
&\frac{d}{d\tau}\hat{\mathcal{C}}_{g,l}(t_{1},t_{2},..,t_{l};\tau)-\sum_{i=1}^{l}\frac{\partial
y_{i}}{\partial \tau}t_{i}^{2}(\tau+1)\frac{\partial}{\partial
t_{i}}\hat{\mathcal{C}}_{g,l}(t_{1},t_{2},..,t_{l};\tau)\\
&=\frac{\partial}{\partial
\tau}\hat{\mathcal{C}}_{g,l}(t_{1},t_{2},..,t_{l};\tau)+\sum_{i=1}^{l}\frac{\partial
v_{i}}{\partial \tau}\frac{\partial}{\partial
v_{i}}\hat{\mathcal{C}}_{g,l}(v_{1},v_{2},..v_{l};\tau)-\sum_{i=1}^{l}\frac{\partial
y_{i}}{\partial \tau}t_{i}^{2}(\tau+1)\frac{\partial}{\partial
t_{i}}\hat{\mathcal{C}}_{g,l}(t_{1},t_{2},..t_{l};\tau)
\end{align*}

Recall the following relationship of $x,y,t,v$.
\begin{align*}
x=e^{-\frac{1}{2}v^2}, y(1-y)^{\tau}=x, t=\frac{1}{1-(\tau+1)y}
\end{align*}

Thus
\begin{align*}
\frac{\partial v_{1}}{\partial \tau}=0, \frac{\partial y}{\partial
\tau}=\frac{y(1-y)ln(1-y)}{y(\tau+1)-1},
\frac{(t^2-t)(t\tau+1)}{\tau+1}\frac{d}{d
t}=-\frac{1}{v}\frac{d}{dv}
\end{align*}

After some simple calculation,
\begin{align*}
\text{LHS of $(10)$}=\frac{\partial}{\partial
\tau}\hat{\mathcal{C}}_{g,l}(t_{1},t_{2},..,t_{l};\tau)+\sum_{i=1}^{l}ln(1-y_{i})\left(
-\frac{1}{v_{i}}\frac{d}{dv_{i}}\right)\hat{\mathcal{C}}_{g,l}(t_{1},t_{2},..,t_{l};\tau)
\end{align*}

The direct image of a function $f(v)$ on $C$ via the projection
$\pi: C\rightarrow \mathbb{C}$ is:
\begin{align*}
\pi_{*} f=f(v)+f(-v)
\end{align*}

Thus,
\begin{align*}
\pi_{*}(\hat{C}_{g,l}(t_{1},t_{2},..,t_{l};\tau))=-(\tau^2+\tau)^{l-1}\sum_{b_{L}\geq
0}\langle\tau_{b_{L}}\Gamma_{g}(\tau)\rangle_{g}\left(\hat{\Psi}_{b_{1}}(t_{1};\tau)+\hat{\Psi}_{b_{1}}(s(t_1);\tau)
\right)\hat{\Psi}_{b\backslash  \{1\}}(t_{L\backslash \{1\}};\tau)
\end{align*}
By the definition of $F_{b}(w;\tau)$,
\begin{align*}
&\pi_{*}\left(\frac{\partial}{\partial
\tau}\hat{C}_{g,l}(t_{1},t_{2},..,t_{l};\tau)
\right)=\frac{\partial}{\partial
\tau}\left(\pi_{*}\hat{\mathcal{C}}_{g,l}(t_{1},t_{2},..,t_{l};\tau)
\right)\\
&=\frac{\partial}{\partial
\tau}\left((\tau^2+\tau)^{l-1}\sum_{b_{L}\geq
0}\langle\tau_{b_{L}}\Gamma_{g}(\tau)\rangle_{g}F_{b_1}(w_1;\tau)\hat{\Psi}_{b_{L\backslash
\{1\} }}(t_{L\backslash \{1\}})\right)
\end{align*}
is regular in $w_{1}$. We will use $\mathcal{O}(w_{1})$ to denote
some function regular in $w_{1}$.

Consider the direct sum
\begin{align*}
\pi_{*}\left(\sum_{i=1}^{l}ln(1-y_{i})\left(-\frac{1}{v_{1}}\frac{d}{dv_{1}}
\right)\hat{\mathcal{C}}_{g,l}(t_{1},t_{2},..,t_{l};\tau)\right)
\end{align*}
when $i=1$,
\begin{align*}
&\pi_{*}\left(ln(1-y_{1})\left(-\frac{1}{v_{1}}\frac{d}{dv_{1}}
\right)\hat{\mathcal{C}}_{g,l}(t_{1},t_{2},..,t_{l};\tau)\right)\\
&=-(\tau^2+\tau)^{l-1}\sum_{b_L\geq
0}\langle\tau_{b_{L}}\Gamma_{g}(\tau)\rangle_{g}\left[ln\left(\frac{t_{1}\tau+1}{(\tau+1)t_{1}}
\right)\left(
-\frac{1}{v_{1}}\frac{d}{dv_1}\right)\left(\eta_{b_1}(v_1;\tau)-F_{b_1}(w_1;\tau)
\right)\right.\\
&\left.+ln\left(\frac{s(t_{1})\tau+1}{(\tau+1)s(t_{1})}
\right)\left(
-\frac{1}{v_{1}}\frac{d}{dv_1}\right)\left(\eta_{b_1}(v_1;\tau)-F_{b_1}(w_1;\tau)
\right) \right]\hat{\Psi}_{b \backslash\{1\}}(t_{b
\backslash\{1\}})\\
&=-(\tau^2+\tau)^{l-1}\sum_{b_L\geq
0}\langle\tau_{b_{L}}\Gamma_{g}(\tau)\rangle_{g}\left[\left(ln\left(\frac{t_{1}\tau+1}{(\tau+1)t_{1}}
\right)-ln\left(\frac{s(t_{1})\tau+1}{(\tau+1)s(t_{1})}
\right)\right)\eta_{b_1}(v_1;\tau)\right.\\
&\left.+\left(ln\left(\frac{t_{1}\tau+1}{(\tau+1)t_{1}}
\right)-ln\left(\frac{s(t_{1})\tau+1}{(\tau+1)s(t_{1})}
\right)\right)F_{b_1}(w_1;\tau) \right]\hat{\Psi}_{b
\backslash\{1\}}(t_{b \backslash\{1\}})\\
&=(\tau^2+\tau)^{l-1}\sum_{b_L\geq
0}\langle\tau_{b_L}\Gamma_{g}(\tau)\rangle_{g}2\eta_{-1}(v_{1};\tau)\eta_{b_{1}+1}(v_1;\tau)\hat{\Psi}_{b_L\backslash
\{1\}}(t_{L\backslash \{1\}};\tau)+\mathcal{O}(w_{1})
\end{align*}

and it is easy to see that
\begin{align*}
\pi_{*}\left(\sum_{i=2}^{l}ln(1-y_{i})\left(-\frac{1}{v_{1}}\frac{d}{dv_{1}}
\right)\hat{\mathcal{C}}_{g,l}(t_{1},t_{2},..,t_{l};\tau)\right)=\mathcal{O}(w_1)
\end{align*}
Thus
\begin{align*}
\pi_{*}(\text{LHS of $(10)$}) =(\tau^2+\tau)^{l-1}\sum_{b_L\geq
0}\langle\tau_{b_L}\Gamma_{g}(\tau)\rangle_{g}2\eta_{-1}(v_{1};\tau)\eta_{b_{1}+1}(v_1;\tau)\hat{\Psi}_{b_L\backslash
\{1\}}(t_{L\backslash \{1\}};\tau)+\mathcal{O}(w_{1})
\end{align*}
\begin{align*}
&(\tau^2+\tau)^{l-1}\sum_{b_L\geq
0}\langle\tau_{b_L}\Gamma_{g}(\tau)\rangle_{g}d\hat{\Psi}_{b_L}(t_{L};\tau)\\
&=(\tau^2+\tau)^{l-1}\sum_{b_L\geq
0}\langle\tau_{b_L}\Gamma_{g}(\tau)\rangle_{g}dv_{1}\frac{d}{dv_1}\left(\hat{\Psi}_{b_1}(t_1;\tau)
\right)d\hat{\Psi}_{b_L\backslash \{1\}}(t_{L\backslash
\{1\}};\tau)\\
&=(-v_{1}dv_{1})d_{2}\cdots d_{l}\left(
(\tau^2+\tau)^{l-1}\sum_{b_L\geq
0}\langle\tau_{b_L}\Gamma_{g}(\tau)\rangle_{g}\left(-\frac{1}{v_1}\frac{d}{dv_1}
\right)\left(\eta_{b_1}(v_1;\tau)-F_{b_1}(w_1;\tau)\right)\right)\hat{\Psi}_{b_L\backslash
\{1\}}(t_{L\backslash \{1\}})\\
&=\frac{-v_1dv_1}{2\eta_{-1}(v_1;\tau)}d_{2}\cdots d_{l}\left(
(\tau^2+\tau)^{l-1}\sum_{b_L\geq
0}\langle\tau_{b_L}\Gamma_{g}(\tau)\rangle_{g}2\eta_{-1}(v_1;\tau)\left(\eta_{b_1+1}(v_1;\tau)
-F_{b_1}(w_1;\tau)\right)\right)\hat{\Psi}_{b_L\backslash
\{1\}}(t_{L\backslash \{1\}})\\
&=\frac{-v_1dv_1}{2\eta_{-1}(v_1;\tau)}d_{2}\cdots
d_{l}\left(\pi_{*}(\text{LHS of
$(10)$})-2\eta_{-1}(v_1;\tau)F_{b_1}(w_1;\tau)(\tau^2+\tau)^{l-1}\sum_{b_L\geq
0}\langle\tau_{b_L}\Gamma_{g}(\tau)\rangle_{g}\hat{\Psi}_{b_L\backslash
\{1\}}(t_{L\backslash \{1\}})\right.\\
&\left.-\mathcal{O}(w_1)\right)
\end{align*}
It is important to note that
\begin{align*}
\left(\frac{-v_1dv_1\mathcal{O}(w_1)}{2\eta_{-1}(v_1;\tau)}|_{v_1=v_1(t_1)}
\right)_{+}=0
\end{align*}
and
\begin{align*}
\left(-v_1dv_1F_{b_1+1}(w_1;\tau)|_{v_1=v_1(t_1)} \right)_{+}=0
\end{align*}
Thus,
\begin{align}
(\tau^2+\tau)^{l-1}\sum_{b_L\geq
0}\langle\tau_{b_{L}}\Gamma_{g}(\tau)\rangle_{g}d\hat{\Psi}_{b_L}(t_L;\tau)=\left(
-\frac{v_1dv_1}{2\eta_{-1}(v_1;\tau)}d_2\cdots
d_l\left(\pi_{*}(\text{LHS of $(10)$}) \right)\right)
\end{align}

Then we need to calculate $\left(
\frac{-v_1dv_1}{2\eta_{-1}(v_1;\tau)}d_2\cdots
d_{l}\pi_{*}(T_1)|_{v_1=v_1(t_1)}\right)_{+}$

Let us write $T_{1}=T_{11}+T_{12}$, where
\begin{align*}
&T_{11}=-\frac{(\tau^2+\tau)^{l-2}}{(\tau+1)}\sum_{2\leq j\leq
l}\sum_{\substack{a\geq 0\\ b_{L\backslash\{1,j\}}\geq 0}}\langle
\tau_a\tau_{b_{L\backslash\{1,j\}}}\Gamma_{g}(\tau)\rangle_{g}\hat{\Psi}_{b_{L\backslash\{1,j\}}}(t_{L\backslash
\{1,j\}};\tau)\\
&\frac{(t_{j}-1)(t_{1}^{2}\tau+t_{1})\hat{\Psi}_{a+1}(t_{1};\tau)-(t_1-1)(t_{j}^2\tau+t_j)\hat{\Psi}_{a+1}(t_j;\tau)}{t_{1}-t_{j}}
\end{align*}
\begin{align*}
&T_{12}=-\frac{(\tau^2+\tau)^{l-2}}{(\tau+1)}\sum_{2\leq i < j\leq
l}\sum_{\substack{a\geq 0\\ b_{L\backslash\{i,j\}}\geq 0}}\langle
\tau_a\tau_{b_{L\backslash\{i,j\}}}\Gamma_{g}(\tau)\rangle_{g}\hat{\Psi}_{b_{L\backslash\{i,j\}}}(t_{L\backslash
\{i,j\}};\tau)\\
&\frac{(t_{j}-1)(t_{i}^{2}\tau+t_{i})\hat{\Psi}_{a+1}(t_{i};\tau)-(t_i-1)(t_{j}^2\tau+t_j)\hat{\Psi}_{a+1}(t_j;\tau)}{t_{i}-t_{j}}
\end{align*}
By the definition,
\begin{align}
\left(\frac{-v_1dv_1}{2\eta_{-1}(v_1;\tau)}\pi_{*}T_{12}|_{v_1=v_1(t_1)}\right)_{+}=0
\end{align}

we have that
\begin{align}
&\frac{-v_1dv_1}{2\eta_{-1}(v_1;\tau)}d_{j}\left(\frac{(t_j-1)(t_1^2\tau+t_1)\hat{\Psi}_{a+1}(t_1;\tau)}{t_1-t_j}\right)\\\nonumber
&=\frac{1}{2\eta_{-1}(v_1;\tau)}d_{j}\left(\frac{(\tau+1)dt_1}{(t_1^2-t_1)(t_1\tau+1)}\frac{(t_j-1)(t_1^2\tau+t_1)}{t_1-t_j}\hat{\Psi}_{a+1}(t_1;\tau)
\right)\\\nonumber
&=\frac{\tau+1}{2\eta_{-1}(v_1;\tau)}\frac{\hat{\Psi}_{a+1}(t_1;\tau)dt_1dt_j}{(t_1-t_j)^2}\\\nonumber
&=\frac{\tau+1}{2\eta_{-1}(v_1;\tau)}\hat{\Psi}_{a+1}(t_1;\tau)B(t_1,t_j;\tau)
\end{align}
Similarly,
\begin{align}
&\frac{-v_1dv_1}{2\eta_{-1}(v_1;\tau)}d_j\left(\frac{(t_j-1)(s(t_1)^2\tau+s(t_1))\hat{\Psi}_{a+1}(s(t_1);\tau)}{s(t_1)-t_j}\right)\\\nonumber
&=\frac{\tau+1}{2\eta_{-1}(v_1;\tau)}\hat{\Psi}_{a+1}(s(t_1);\tau)B(s(t_1),t_j;\tau)
\end{align}
Moreover, by the expansion,
\begin{align*}
\frac{t_1-1}{t_1-t_j}+\frac{s(t_1)-1}{s(t_1)-t_j}=\sum_{k\geq
0}\left(\frac{t_j^k}{t_1^k}+\frac{t_j^k}{t_1^{k+1}}+\frac{t_j^k}{s(t_1)^k}+\frac{t_j^k}{s(t_1)^{k+1}}
\right)
\end{align*}
we have
\begin{align}
\left(\frac{-v_1dv_1}{2\eta_{-1}(v_1;\tau)}\left(\frac{t_1-1}{t_1-t_j}+\frac{s(t_1)-1}{s(t_1)-t_j}\right)|_{v_1=v_1(t_1)}\right)_{+}=0
\end{align}
Therefore
\begin{align}
&\left(\frac{-v_1dv_1}{2\eta_{-1}(v_1;\tau)}d_2\cdots
d_{l}\pi_{*}T_{1}|_{v_1=v_1(t_1)}\right)_{+}\quad \text{by
$(39)$}\\\nonumber
&=\left(\frac{-v_1dv_1}{2\eta_{-1}(v_1;\tau)}d_2\cdots
d_{l}\pi_{*}T_{11}|_{v_1=v_1(t_1)}\right)_{+}\quad \text{by $(40)$,
$(41)$, $(42)$}
\\\nonumber
&=(\tau^2+\tau)^{l-2}\sum_{2\leq j\leq l}\sum_{\substack{a\geq
0\\b_{L\backslash \{1,j\}}
}}\langle\tau_a\tau_{b_{L\backslash\{1,j\}}}\Gamma_{g}(\tau)
\rangle_{g}d\hat{\Psi}_{b_{L\backslash\{1,j\}}}(t_{L\backslash\{1,j\}};\tau)\\\nonumber
&\left(\frac{\hat{\Psi}_{a+1}(t_1;\tau)B(t_1,t_j)+\hat{\Psi}_{a+1}(s(t_1);\tau)B(s(t_1),t_j)}
{2\eta_{-1}(v_1;\tau)}|_{v_1=v_1(t_1)}\right)_{+}\\\nonumber
&=-(\tau^2+\tau)^{l-2}\sum_{2\leq j\leq l}\sum_{\substack{a\geq
0\\b_{L\backslash \{1,j\}}
}}\langle\tau_a\tau_{b_{L\backslash\{1,j\}}}\Gamma_{g}(\tau)
\rangle_{g}d\hat{\Psi}_{b_{L\backslash\{1,j\}}}(t_{L\backslash\{1,j\}};\tau)\\\nonumber
&\left(\frac{\hat{\Psi}_{a+1}(t_1;\tau)B(s(t_1),t_j)+\hat{\Psi}_{a+1}(s(t_1);\tau)B(t_1,t_j)}
{2\eta_{-1}(v_1;\tau)}|_{v_1=v_1(t_1)}\right)_{+}\\\nonumber
&=-(\tau^2+\tau)^{l-2}\sum_{2\leq j\leq l}\sum_{\substack{a\geq
0\\b_{L\backslash\{1,j\}}}}\langle\tau_a\tau_{b_{L\backslash
\{1,j\}}}\Gamma_{g}(\tau)\rangle_{g}P_{a+1}(t_1,t_j;\tau)dt_1dt_jd\hat{\Psi}_{b_{L\backslash\{1,j\}}}(t_{L\backslash
\{1,j\}};\tau)
\end{align}

Finally, we need to calculate $
\left(\frac{-v_1dv_1}{2\eta_{-1}(v_1;\tau)}d_2\cdots d_l
\pi_{*}(T_2+T_3)|_{v_1=v_1(t_1)} \right)_{+}$

Because
\begin{align*}
&\left(\frac{-v_1dv_1}{2\eta_{-1}(v_1;\tau)}\pi_{*}\left(\hat{\Psi}_{a_1+1}(t_1;\tau)\hat{\Psi}_{a_2+1}(t_1;\tau)
\right)|_{v_1=v_1(t_1)} \right)_{+}\\
&=\left(\frac{-v_1dv_1}{2\eta_{-1}(v_1;\tau)}\left(\hat{\Psi}_{a_1+1}(t_1;\tau)\hat{\Psi}_{a_2+1}(t_1;\tau)+
\hat{\Psi}_{a_1+1}(s(t_1);\tau)\hat{\Psi}_{a_2+1}(s(t_1);\tau)
\right)|_{v_1=v_1(t_1)} \right)_{+}\\
&=\left(\frac{-v_1dv_1}{2\eta_{-1}(v_1;\tau)}\left(2\eta_{a_1+1}(v_1;\tau)\eta_{a_2+1}(v_1;\tau)+2F_{a_1+1}(w_1;\tau)
F_{a_2+1}(w_1;\tau) \right)|_{v_1=v_1(t_1)} \right)_{+}\\
&=2P_{a_1,a_2}(t_1;\tau)dt_1
\end{align*}
In the last "=", we have used Corollary 4.10 and
$\left(\frac{F_{a_1+1}(w_1;\tau)F_{a_2+1}(w_1;\tau)v_1dv_1}{\eta_{-1}(v_1;\tau)}
|_{v_{1}=v_1(t_1)}\right)_{+}=0$.

We have
\begin{align}
&\left(\frac{-v_1dv_1}{2\eta_{-1}(v_1;\tau)}d_2\cdots d_l
\pi_{*}(T_2+T_3)|_{v_1=v_1(t_1)}
\right)_{+}=(\tau^2+\tau)^{l-1}\sum_{\substack{a_1\geq 0\\ a_2\geq
0\\ b_{L\backslash\{i\} }\geq
0}}\left((\tau^2+\tau)\langle\tau_{a_1}\tau_{a_2}\tau_{b_{L\backslash
\{i\}}}\Gamma_{g-1}(\tau)\rangle_{g-1}\right.\\\nonumber
&\left.-\sum_{\substack{g_1+g_2=g\\ \mathcal{I}\bigcup
\mathcal{J}=L\backslash\{i\}}}^{stable}\langle\tau_{a_1}\tau_{b_{\mathcal{I}}}\Gamma_{g_1}(\tau)
\rangle_{g_1}\langle\tau_{a_2}\tau_{b_{\mathcal{J}}}\Gamma_{g_2}(\tau)
\rangle_{g_2} \right)P_{a_1,a_2}(t_1;\tau)dt_1d\hat{\Psi}_{b _{L
\backslash \{1\}}}(t_{L\backslash \{1\}};\tau)
\end{align}

Hence, by equations $(38)$, $(43)$, $(44)$ and $(10)$, we have
proved the identity $(30)$, i.e the Bouchard-Marino conjecture for
$\mathbb{C}^3$ case.
\begin{remark}
At the same time as L. Chen published his proof \cite{Ch2}, J. Zhou
also proved the Bouchard-Mari\~no conjecture for $\mathbb{C}^{3}$
\cite{Zhou2}. He formulated this new recursion relation of
conjecture and equation $(1)$ of theorem 1.1 in the coordinate $v$.
Then, he regarded equation $(1)$ as meromorphic functions in $v_1$,
taking the principal parts and only the even powers in $v_1$ would
get the recursion relation for this conjecture. J. Zhou has also
applied his method to formulate and prove this new recursion
relation for topological vertex \cite{Zhou3} based on the work
\cite{LLLZ}.
\end{remark}

\section{Application II: Derivation of Some Hodge integral identities}
In this section, we will use the main theorem 1.1 to obtain some
Hodge integral identities.
\subsection{Preliminary calculations}

 For convenience, let us recall the
formula $(1)$ in theorem 1.1 at first: for $g\geq 1$ and $l\geq 1$,
Then,
\begin{align}
&LHS:=-(\tau^2+\tau)^{l-2}\sum_{b_{L}\geq
0}\left((l-1)(2\tau+1)\langle\tau_{b_L}\Gamma_{g}(\tau)\rangle_{g}+
(\tau^{2}+\tau)\langle\tau_{b_{L}}\frac{d}{d\tau}\Gamma_{g}(\tau)\rangle_{g}\right)
\hat{\Psi}_{b_{L}}(t_{L};\tau)\\\nonumber
&-(\tau^{2}+\tau)^{l-1}\sum_{b_{L}\geq
0}\langle\tau_{b_{L}}\Gamma_{g}(\tau)\rangle_{g}\sum_{i=1}^{l}\left(\frac{\partial}{\partial
\tau}\hat{\Psi}_{b_{i}}(t_{i};\tau)+\frac{1}{t_{i}\tau+1}\hat{\Psi}_{b_{i}+1}(t_{i};\tau)\right)\hat{\Psi}_{b_{L\setminus
\{i\}}}(t_{L\setminus \{i\}};\tau)\\\nonumber &=T_{1} +T_{2}+T_{3}
\end{align}
where
\begin{align*}
&T_{1}:= -\frac{(\tau^2+\tau)^{l-2}}{\tau+1}\sum_{1\leq i<j\leq
l}\sum_{\substack{a\geq 0\\b_{L\setminus\{i,j\} }\geq
0}}\langle\tau_{a}\tau_{b_{L\setminus\{i,j\}}}\Gamma_{g}(\tau)\rangle_{g}
\hat{\Psi}_{b_{L\setminus\{i,j\}
}}(t_{L\setminus\{i,j\}};\tau)\\\nonumber
&\cdot\frac{(t_{j}-1)(t_{i}^2\tau+t_{i})\hat{\Psi}_{a+1}(t_{i};\tau)-(t_{i}-1)(t_{j}^2\tau+t_{j})\hat{\Psi}_{a+1}(t_{j};\tau)}
{t_{i}-t_{j}}
\end{align*}
\begin{align*}
T_{2}:=\frac{(\tau^2+\tau)^{l}}{2}\sum_{i=1}^{l}\sum_{\substack{a_{1}\geq
0\\a_{2}\geq 0\\b_{L\setminus\{i\}}\geq
0}}\langle\tau_{a_{1}}\tau_{a_{2}}\tau_{b_{L\setminus\{i\}}}\Gamma_{g-1}(\tau)\rangle_{g-1}
\prod_{n=1}^{2}\hat{\Psi}_{a_{n}+1}(t_{i};\tau)\hat{\Psi}_{b_{L\setminus\{i\}}}(t_{L\setminus\{i\}};\tau)
\end{align*}
\begin{align*}
T_{3}:=-\frac{(\tau^2+\tau)^{l-1}}{2}\sum_{i=1}^{l}\sum_{\substack{a_{1}\geq
0\\a_{2}\geq 0\\b_{L\setminus\{i\}}}}
\sum_{\substack{g_{1}+g_{2}=g\\ \mathcal{I}\coprod\mathcal{J}=L\setminus\{i\}\\2g_{1}-1+|\mathcal{I}|>0\\
2g_{2}-1+|\mathcal{J}|>0}}\langle\tau_{a_{1}}\tau_{b_{\mathcal{I}}}\Gamma_{g_{1}}(\tau)\rangle_{g_{1}}\langle\tau_{a_{2}}
\tau_{b_{\mathcal{J}}}\Gamma_{g_{2}}(\tau)\rangle_{g_{2}}
\prod_{n=1}^{2}\hat{\Psi}_{a_{n}+1}(t_{i};\tau)\hat{\Psi}_{b_{L\setminus\{i\}}}(t_{L\setminus\{i\}};\tau)
\end{align*}
$$\hat{\Psi}_{n}(t;\tau)=\left(\frac{(t^2-t)(t\tau+1)}{\tau+1}\frac{d}{dt}\right)^n\left(\frac{t-1}{\tau+1}\right), n\geq 0$$
 and
$$\Gamma_{g}(\tau)=\Lambda_{g}^{\vee}(1)\Lambda_{g}^{\vee}(-\tau-1)\Lambda_{g}^{\vee}(\tau)$$
Let us denote the $\tau$ expansion of $\hat{\Psi}_{n}(t;\tau)$ and
$\Gamma_{g}(\tau)$ as follow,
\begin{align}
\hat{\Psi}_{b}(t;\tau)=\sum_{k=0}^{b}\frac{\tau^k}{(\tau+1)^{b+1}}\Psi_{b}^{k}(t)
\end{align}
and
\begin{align}
\Gamma_{g}(\tau)=\sum_{m=0}^{2g}\Lambda_{g}^{\vee}(1)a_{m}(\lambda)\tau^{m}
\end{align}
By the definition of
$\Lambda_{g}^{\vee}(t)=\sum_{j=0}^{g}(-1)^{g-j}\lambda_{g-j}t^{j}$
and Mumford's relation
$\Lambda_{g}^{\vee}(t)\Lambda_{g}^{\vee}(-t)=(-1)^{g}t^{2g}$, we can
compute out the coefficients $a_{m}(\lambda)$ in $(47)$. For
example,
\begin{align}
&a_{2g}(\lambda)=(-1)^{g}, a_{2g-1}(\lambda)=(-1)^{g}g,
\cdots\\\nonumber
&a_{1}(\lambda)=\sum_{m=1}^{g}m\lambda_{g-m}\lambda_{g}-(-1)^g\lambda_{g-1}\Lambda_{g}^{\vee}(-1)\\\nonumber
&a_{0}(\lambda)=(-1)^{g}\lambda_{g}\Lambda_{g}^{\vee}(-1)\\\nonumber
\end{align}

We also need to show the expansion form of $(46)$ and how to
calculate the coefficients $\Psi_{b}^{k}(t)$. By definition
$\hat{\Psi}_{b+1}(t,\tau)=(t^2-t)(\frac{y\tau+1}{\tau+1})\frac{d}{dy}\Psi_{b}(y,\tau)$,
$\hat{\Psi}_{0}(t,\tau)=\frac{t-1}{\tau+1}$, we have the expansion
form $(46)$. Moreover, we have
\begin{align}
\Psi_{b+1}^{k}(t)=(t^{3}-t^2)\frac{d}{dt}\Psi_{b}^{k-1}(t)+(t^2-t)\frac{d}{dt}\Psi_{b}^{k}(t),
k=0,..,b+1
\end{align}
Therefore, through the recursion formula $(49)$, all the
$\Psi_{b}^{k}(t)$ can be calculated.

As an illustration, we calculate some cases which will be used in
the following discussion.

\begin{align}
\Psi_{b+1}^{b+1}(t)=(t^3-t^2)\frac{d}{dt}\Psi_{b}^{b}(t)
\end{align}
\begin{align}
\Psi_{b+1}^{b}(t)=(t^3-t^2)\frac{d}{dt}\Psi_{b}^{b-1}(t)+(t^2-t)\frac{d}{dt}\Psi_{b}^{b}(t)
\end{align}
\begin{align}
\Psi_{b+1}^{1}(t)=(t^3-t^2)\frac{d}{dt}\Psi_{b}^{0}(t)+(t^2-t)\frac{d}{dt}\Psi_{b}^{1}(t)
\end{align}
\begin{align}
\Psi_{b+1}^{0}(t)=(t^2-t)\frac{d}{dt}\Psi_{b}^{0}(t)
\end{align}

It is clear that, the recursion relation for $\Psi_{b}^{b}(t)$ in
$(50)$ is just same for the definition of $\hat{\xi}_{b}(t)$ in
\cite{EMS}, and they share the same initial. thus
\begin{align}
\Psi_{b}^{b}(t)=\hat{\xi}_{b}(t)
\end{align}
By $(50)$, we write $\Psi_{b}^{b}(t)$ as
\begin{align}
\Psi_{b}^{b}(t)=\sum_{i=b+1}^{2b+1}f^{b}(b,i)t^i
\end{align}
then all the $f^{b}(b,i)$ could be calculated from the recursion
$(50)$. But we can explicitly write down $f^{b}(b,2b+1)=(2b-1)!!,
f^{b}(b,b+1)=(-1)^{b}b!$ which are the coefficients used in the
proof of $DVV$ equation and $\lambda_{g}$ integral
respectively\cite{MZ,CLL,GJV}.

We can also let
\begin{align}
\Psi_{b}^{b-1}(t)=\sum_{i=b}^{2b}f^{b-1}(b,i)t^i
\end{align}
then by $(51)$ we have
\begin{align}
f^{b}(b+1,k)=(k-2)f^{b-1}(b,k-2)-(k-1)f^{b-1}(b,k-1)+(k-1)f^{b}(b,k-1)-kf^{b}(b,k)
\end{align}

Thus, all $f^{b}(b+1,k)$ can be calculated from $(57)$. For example,
\begin{align*}
f^{b}(b+1,2b+2)=2bf^{b-1}(b,2b)+(2b+1)f^{b}(b,2b+1)=2bf^{b-1}(b,2b)+(2b+1)!!
\end{align*}
Hence
\begin{align}
f^{b-1}(b,2b)=\sum_{k=0}^{b-1}\frac{(2b-2)!!(2k+1)!!}{(2k)!!}
\end{align}
Similarly,
\begin{align*}
f^{b}(b+1,b+1)=-bf^{b+1}(b,b)-(b+1)f^{b}(b,b+1)=-bf^{b}(b,b)+(-1)^{b+1}(b+1)!
\end{align*}
Thus
\begin{align}
f^{b-1}(b,b)=(-1)^b\frac{(b+1)!}{2}
\end{align}

On the other side, from $(53)$, we can let
\begin{align*}
\Psi_{b}^{0}(t)=\sum_{i=1}^{b+1}f^{0}(b,i)y^i
\end{align*}
Then
\begin{align*}
f^{0}(b+1,i)=(i-1)f^{0}(b,i-1)-if^{0}(b,i)
\end{align*}
Thus,
\begin{align}
f^{0}(b,b+1)=b!
\end{align}

Then, from $(52)$, we can let
\begin{align}
\Psi_{b}^{1}(t)=\sum_{j=2}^{b+2}f^{1}(b,j)t^j
\end{align}
and
\begin{align*}
f^{1}(b+1,k)=(k-2)f^{0}(b,k-2)-(k-1)f^{0}(b,k-1)+(k-1)f^{1}(b,k-1)-kf^{1}(b,k)
\end{align*}
Thus
\begin{align}
f^{1}(b+1,b+3)=(b+1)f^{0}(b,b+1)+(b+2)f^{1}(b,b+2)
\end{align}
by initial value, $f^{1}(1,3)=1$ and $f^{0}(b,b+1)=b!$, we get
\begin{align}
f^{1}(b,b+2)=(b+1)!\sum_{k=2}^{b+1}\frac{1}{k}.
\end{align}

Now, let us substitute $(46)$ and $(47)$ to $(45)$, we have

\begin{align}
&LHS=-\sum_{b_{L}\geq 0}\sum_{0\leq k_{L}\leq
b_{L}}\sum_{m=0}^{2g}\langle\tau_{b_{L}}\Lambda_{g}^{\vee}(1)a_{m}(\lambda)\rangle_{g}\Psi_{b_{L}}^{k_{L}}(t_{L})\\\nonumber
&\times\frac{(l-2+m+|k_{L}|-|b_{L}|)\tau^{|k_{L}|+l+m-1}+(l-1+m+|k_{L}|)\tau^{|k_{L}|+l+m-2}}{(\tau+1)^{|b_{L}|+2}}\\\nonumber
&-\sum_{b_{L}\geq 0}\sum_{0\leq k_{L}\leq
b_{L}}\sum_{m=0}^{2g}\langle\tau_{b_{L}}\Lambda_{g}^{\vee}a_{m}(\lambda)
\rangle_{g}\sum_{i=1}^{l}(t_{i}^2-t_{i})\frac{\partial}{\partial
t_{i}}\Psi_{b_{i}}^{k_{i}}(t_{i})\Psi_{b_{L\setminus
\{i\}}}^{b_{L\setminus \{i\}}}(t_{L\setminus
\{i\}})\frac{\tau^{|b_{L}|+m+l-2}}{(\tau+1)^{|b_{L}|+2}}
\end{align}

\begin{align}
&T_{1}=-\sum_{1\leq i<j\leq l}\sum_{\substack{a\geq 0\\
b_{L\setminus \{i,j\}}\geq 0}}\sum_{m=0}^{2g}\sum_{\substack{0\leq
k\leq a+1\\0\leq k_{L\setminus \{i,j\}} \leq b_{L\setminus
\{i,j\}}}}\langle\tau_{a}\tau_{b_{L\setminus
\{i,j\}}}\Lambda_{g}^{\vee}(1)a_{m}(\lambda)\rangle_{g}\Psi_{b_{L\setminus
\{i,j\}}}^{b_{L\setminus \{i,j\}}}(t_{L\setminus
\{i,j\}})\\\nonumber
&\times\frac{(t_{j}-1)t_{i}^{2}\Psi_{a+1}^{k}(t_{i})-(t_{i}-1)t_{j}^{2}\Psi_{a+1}^{k}(t_{j})}{t_{i}-t_{j}}
\frac{\tau^{|k_{L\setminus
\{i,j\}}|+k+l+m-1}}{(\tau+1)^{|b_{L\setminus
\{i,j\}}|+a+3}}\\\nonumber  &-\sum_{1\leq i<j\leq l}\sum_{\substack{a\geq 0\\
b_{L\setminus \{i,j\}}\geq 0}}\sum_{m=0}^{2g}\sum_{\substack{0\leq
k\leq a+1\\0\leq k_{L\setminus \{i,j\}} \leq b_{L\setminus
\{i,j\}}}}\langle\tau_{a}\tau_{b_{L\setminus
\{i,j\}}}\Lambda_{g}^{\vee}(1)a_{m}(\lambda)\rangle_{g}\Psi_{b_{L\setminus
\{i,j\}}}^{b_{L\setminus \{i,j\}}}(t_{L\setminus
\{i,j\}})\\\nonumber
&\times\frac{(t_{j}-1)t_{i}\Psi_{a+1}^{k}(t_{i})-(t_{i}-1)t_{j}\Psi_{a+1}^{k}(t_{j})}{t_{i}-t_{j}}
\frac{\tau^{|k_{L\setminus
\{i,j\}}|+k+l+m-2}}{(\tau+1)^{|b_{L\setminus \{i,j\}}|+a+3}}
\end{align}
\begin{align}
&T_{2}=\frac{1}{2}\sum_{i=1}^{l}\sum_{\substack{a_{1}\geq
0\\a_{2}\geq 0\\b_{L\setminus \{i\}}\geq
0}}\sum_{m=0}^{2(g-1)}\sum_{\substack{0\leq n_{1}\leq a_{1}+1\\0\leq
n_{2} \leq a_{2}+1\\ 0\leq k_{L\setminus \{i\}} \leq b_{L\setminus
\{i\}}}}\langle\tau_{a_{1}}\tau_{a_{2}}\tau_{b_{L\setminus
\{i,j\}}}\Lambda_{g-1}^{\vee}(1)a_{m}(\lambda)\rangle_{g-1}\\\nonumber
&\times\Psi_{a_{1}+1}^{n_{1}}(t_{i})\Psi_{a_{2}+1}^{n_{2}}(t_{i})\Psi_{b_{L\setminus
\{i\}}}^{b_{L\setminus \{i\}}}(t_{L\setminus
\{i\}})\frac{\tau^{|k_{L\setminus
\{i\}}|+n_{1}+n_{2}+l+m}}{(\tau+1)^{a_{1}+a_{2}+|k_{L\setminus
\{i\}}|+3}}
\end{align}
\begin{align}
&T_{3}=-\frac{1}{2}\sum_{i=1}^{l}\sum_{\substack{a_{1}\geq
0\\a_{2}\geq 0\\b_{L\setminus \{i\}}\geq
0}}\sum_{\substack{g_{1}+g_{2}=g\\\mathcal{I}\cup
\mathcal{J}=L\setminus\{i\}\\2g_{1}-1+|\mathcal{I}|>0
\\2g_{2}-1+|J|>0}}\sum_{\substack{0\leq m_{1}\leq 2g_{1}\\0\leq m_{2}\leq
2g_{2}}}\sum_{\substack{0\leq n_{1} \leq a_{1}+1\\0\leq n_{2}\leq
a_{2}+1\\0\leq k_{L\{i\}} \leq b_{L\setminus
\{i\}}}}\langle\tau_{a_{1}}\tau_{b_{\mathcal{I}}}\Lambda_{g_{1}}^{\vee}(1)a_{m_{1}}(\lambda)
\rangle_{g_{1}}\\\nonumber
&\times\langle\tau_{a_{2}}\tau_{b_{\mathcal{J}}}\Lambda_{g_{2}}^{\vee}(1)a_{m_{2}}(\lambda)
\rangle_{g_{2}}\Psi_{a_{1}+1}^{n_{1}}(t_{i})\Psi_{a_{2}+1}^{n_{2}}(t_{i})\Psi_{b_{L\setminus\{i\}}}^{k_{L\setminus\{i\}}}
(t_{L\setminus
\{i\}})\frac{\tau^{|k_{L\setminus\{i\}}|+n_{1}+n_{2}+l+m_{1}+m_{2}-1}}{(\tau+1)^{a_{1}+a_{2}+|k_{L\setminus
\{i\}}|+4}}
\end{align}

Where we have used the notation $|b_{L}|=\sum_{i=1}^{l}b_{i}$ in
above formulas. We note that $(64)$, $(65)$, $(66)$ and $(67)$ are
all the functions of $\tau$. If we expand them as the series of
$\tau$ at different points, we will get some identities of Hodge
integrals after recollecting the coefficients of $\tau$ at both
sides of $(64)=(65)+(66)+(67)$.
\subsection{Expansion of $\tau$ at $\infty$}
For $b\geq 0$, expanding at $\tau=\infty$,
\begin{align*}
\frac{\tau^{a}}{(\tau+1)^{b}}=\sum_{h\geq
0}\binom{-b}{h}\tau^{a-b-i}
\end{align*}
for example, the term containing $\tau$ in (35),
\begin{align*}
\frac{\tau^{|k_{L}|+l+m-1}}{(\tau+1)^{|b_{L}|+2}}=\sum_{h\geq
0}\binom{-|b_{L}|-2}{h}\tau^{l+m-3+|k_{L}|-|b_{L}|-h}
\end{align*}
After this expansion, it is easy to see that the largest degree of
$\tau$ is $2g+l-3$ in the formulas $(64)$, $(65)$, $(66)$, $(67)$.
If $F(t,\tau)\in Q[t][[\tau]]$, we will use the notation
$[\tau^{k}]F(t,\tau)$ to mean the coefficient of $\tau^k$ in
$F(t,\tau)$. By $a_{2g}=(-1)^{g}$ in $(48)$ , after some
calculations,

\begin{align}
[\tau^{2g+l-3}]LHS=(-1)^{g+1}\sum_{b_{L}\geq
0}\langle\tau_{b_{L}}\Lambda_{g}^{\vee}(1)\rangle\left((2g-2+l)\Psi_{b_{L}}^{b_{L}}(t_{L})+\sum_{i=1}^{l}(t_{i}^{2}-t_{i})
\frac{\partial}{\partial
t}\Psi_{b_{i}}^{b_{i}}(t_{i})\Psi_{b_{L\setminus\{i\}}}^{b_{L\setminus\{i\}}}(t_{L\setminus\{i\}})\right)
\end{align}
\begin{align}
&[\tau^{2g+l-3}]T_{1}=(-1)^{g+1}\sum_{1\leq i<j\leq
l}\sum_{\substack{a\geq 0\\b_{L\setminus \{i,j\}}\geq
0}}\langle\tau_{a}\tau_{b_{L\setminus \{i,j\}}}\Lambda_{g}^{\vee}(1)
\rangle_{g}\Psi_{b_{L\setminus \{i,j\}}}^{b_{L\setminus
\{i,j\}}}(t_{L\setminus \{i,j\}})\\\nonumber
&\frac{(t_{j}-1)t_{i}^{2}\Psi_{a+1}^{a+1}(t_{i})-(t_{i}-1)t_{j}^{2}\Psi_{a+1}^{a+1}(t_{j})}{t_{i}-t_{j}}
\end{align}
\begin{align}
[\tau^{2g+l-3}]T_{2}=(-1)^{g+1}\frac{1}{2}\sum_{i=1}^{l}\sum_{\substack{a_{1}\geq
0\\a_{2}\geq 0\\b_{L \setminus \{i\}}\geq
0}}\langle\tau_{a_{1}}\tau_{a_{2}}\tau_{b_{L\setminus
\{i\}}}\Lambda_{g-1}^{\vee}(1) \rangle_{g-1}
\Psi_{a_{1}+1}^{a_{1}+1}(t_{i})\Psi_{a_{2}+1}^{a_{2}+1}(t_{i})\Psi_{b_{L\setminus\{i\}
}}^{b_{L\setminus\{i\} }}(t_{L\setminus\{i\}})
\end{align}
\begin{align}
[\tau^{2g+l-3}]T_{3}=(-1)^{g+1}\frac{1}{2}\sum_{i=1}^{l}\sum_{\substack{a_{1}\geq
0\\a_{2}\geq 0\\b_{L \setminus
\{i\}}\geq 0}}\sum_{\substack{g_{1}+g_{2}=g\\
\mathcal{I}\cup\mathcal{J}=L\setminus\{i\}}}\langle\tau_{a_{1}}\tau_{b_{\mathcal{I}}}\Lambda_{g_{1}}^{\vee}(1)
\rangle_{g_{1}}\langle\tau_{a_{2}}\tau_{b_{\mathcal{J}}}\Lambda_{g_{2}}^{\vee}(1)\rangle_{g_{2}}\prod_{n=1}^{2}
\Psi_{a_{n}+1}^{a_{n}+1}(t_{i})\Psi_{b_{L\setminus\{i\}
}}^{b_{L\setminus\{i\} }}(t_{L\setminus\{i\}})
\end{align}

As showed in formula $(54)$, $\Psi_{b}^{b}(t)=\hat{\xi}_{b}(t)$, we
have got the following corollary from $(68)=(69)+(70)+(71)$,
\begin{corollary}
\begin{align}
&\sum_{b_{L}\geq
0}\langle\tau_{b_{L}}\Lambda_{g}^{\vee}(1)\rangle\left((2g-2+l)\hat{\xi}_{b_{L}}(t_{L})+\sum_{i=1}^{l}(t_{i}^{2}-t_{i})
\frac{\partial}{\partial
t}\hat{\xi}_{b_{i}}(t_{i})\hat{\xi}_{b_{L\setminus\{i\}}}(t_{L\setminus\{i\}})\right)\\\nonumber
&=\sum_{1\leq i<j\leq l}\sum_{\substack{a\geq 0\\b_{L\setminus
\{i,j\}}\geq 0}}\langle\tau_{a}\tau_{b_{L\setminus
\{i,j\}}}\Lambda_{g}^{\vee}(1) \rangle_{g}\hat{\xi}_{b_{L\setminus
\{i,j\}}}\frac{(t_{j}-1)t_{i}^{2}\hat{\xi}_{a+1}(t_{i})-(t_{i}-1)t_{j}^{2}\hat{\xi}_{a+1}(t_{j})}{t_{i}-t_{j}}\\\nonumber
&+\frac{1}{2}\sum_{i=1}^{l}\sum_{\substack{a_{1}\geq 0\\a_{2}\geq
0\\b_{L \setminus \{i\}}\geq
0}}\left(\langle\tau_{a_{1}}\tau_{a_{2}}\tau_{b_{L\setminus
\{i\}}}\Lambda_{g-1}^{\vee}(1)
\rangle_{g-1}+\sum_{\substack{g_{1}+g_{2}=g\\
\mathcal{I}\cup\mathcal{J}=L\setminus\{i\}}}^{stable}\langle\tau_{a_{1}}\tau_{b_{\mathcal{I}}}\Lambda_{g_{1}}^{\vee}(1)
\rangle_{g_{1}}\langle\tau_{a_{2}}\tau_{b_{\mathcal{J}}}\Lambda_{g_{2}}^{\vee}(1)
\rangle_{g_{2}} \right)\\\nonumber
&\times
\hat{\xi}_{a_{1}+1}(t_{i})\hat{\xi}_{a_{2}+1}(t_{i})\hat{\xi}_{b_{L\setminus\{i\}
}}(t_{L\setminus\{i\}})
\end{align}
\end{corollary}
which is just the main theorem 1.1 showed in paper \cite{EMS}.
Formula $(72)$ has been applied by many people to derive DVV
equation \cite{CLL,MZ}, $\lambda_{g}$-conjecture \cite{GJV,MZ} and
$\lambda_{g-1}$ Hodge integral recursion \cite{Zhu}.

Obviously, in this procedure, we can obtain a series of such
corollaries. For example, calculating out the coefficients of
$\tau^{2g+l-4}$ by  $a_{2g}=(-1)^{g}$, $a_{2g-1}(\lambda)=(-1)^{g}g$
in $(48)$ we have
\begin{align}
&[\tau^{2g+l-4}]LHS\\\nonumber
&=(-1)^{g+1}(2g-3+l)\sum_{b_{L}\geq
0}\langle\tau_{b}\Lambda_{g}^{\vee}(1)\rangle_{g}[(g-1-|b_{L}|)\Psi_{b_{L}}^{b_{L}}(t_{L})+\sum_{i=1}^{l}\Psi_{b_{i}}^{b_{i}-1}(t_{i})
\Psi_{b_{L\setminus \{i\}}}^{b_{L\setminus \{i\}}}(t_{L\setminus
\{i\}})]\\\nonumber &+(-1)^{g+1}\sum_{b_{L}\geq
0}\langle\tau_{b_{L}}\Lambda_{g}^{\vee}(1) \rangle_{g}[
\sum_{i=1}^{l}(t_{i}^2-t_{i})\frac{\partial}{\partial
t_{i}}\left((g-|b_{L}|-2)\Psi_{b_{i}}^{b_{i}}(t_{i})+\Psi_{b_{i}}^{b_{i}-1}(t_{i})\right)(\Psi_{b_{L\setminus
\{i\}}}^{b_{L\setminus \{i\}}}(t_{L\setminus \{i\}})\\\nonumber
&+\sum_{j\neq
i}\Psi_{b_{j}}^{b_{j}-1}(t_{j})\Psi_{b_{L\setminus\{i,j\}}}^{b_{L\setminus\{i,j\}}}(t_{L\setminus\{i,j\}})
]
\end{align}

\begin{align}
&[\tau^{2g+l-4}]T_{1}=(-1)^{g+1}\sum_{1\leq i<j\leq
l}\sum_{\substack{a\geq 0\\b_{L\setminus \{i,j\}}\geq
0}}\langle\tau_{a}\tau_{b_{L\setminus \{i,j\}}}\Lambda_{g}^{\vee}(1)
\rangle_{g}\Psi_{b_{L\setminus \{i,j\}}}^{b_{L\setminus
\{i,j\}}}(t_{L\setminus\{i,j\}})\\\nonumber
&(\frac{(t_{j}-1)t_{i}^{2}[\Psi_{a+1}^{a}(t_{i})+(g-a-|b_{L\setminus
\{i,j\}}|-3+\frac{1}{t_{i}})\Psi_{a+1}^{a+1}(t_{i})]}{t_{i}-t_{j}}\\\nonumber
&-\frac{(t_{i}-1)t_{j}^{2}[\Psi_{a+1}^{a}(t_{j})+(g-a-|b_{L\setminus
\{i,j\}}|-3+\frac{1}{t_{j}})\Psi_{a+1}^{a+1}(t_{j})]}{t_{i}-t_{j}})
\end{align}
\begin{align}
&[\tau^{2g+l-4}]T_{2}=(-1)^{g+1}\frac{1}{2}\sum_{i=1}^{l}\sum_{\substack{a_{1}\geq
0\\a_{2}\geq 0\\b_{L \setminus \{i\}}\geq
0}}\langle\tau_{a_{1}}\tau_{a_{2}}\tau_{b_{L\setminus
\{i\}}}\Lambda_{g-1}^{\vee}(1)
\rangle_{g-1}[(g-a_{1}-a_{2}-|b_{L\setminus\{i,j\}}|-4)\\\nonumber
&\cdot\Psi_{a_{1}+1}^{a_{1}+1}(t_{i})\Psi_{a_{2}+1}^{a_{2}+1}(t_{i})
+\Psi_{a_{1}}^{a_{1}+1}(t_{i})\Psi_{a_{2}+1}^{a_{2}+1}(t_{i})+\Psi_{a_{1}+1}^{a_{1}+1}(t_{i})
\Psi_{a_{2}+1}^{a_{2}}(t_{i})]\Psi_{b_{L\setminus\{i\}
}}^{b_{L\setminus\{i\} }}(t_{L\setminus\{i\}})\\\nonumber
&+(-1)^{g+1}\frac{1}{2}\sum_{i=1}^{l}\sum_{\substack{a_{1}\geq
0\\a_{2}\geq 0\\b_{L \setminus \{i\}}\geq
0}}\langle\tau_{a_{1}}\tau_{a_{2}}\tau_{b_{L\setminus
\{i\}}}\Lambda_{g-1}^{\vee}(1)
\rangle_{g-1}\Psi_{a_{1}}^{a_{1}+1}(t_{i})\Psi_{a_{2}+1}^{a_{2}+1}(t_{i})\sum_{j\neq
i}\Psi_{b_{j}}^{b_{j}-1}(t_{j})\Psi_{b_{L\setminus\{i,j\}}}^{b_{L\setminus\{i,j\}}}(t_{L\setminus\{i,j\}})
\end{align}
\begin{align}
&[\tau^{2g+l-4}]T_{3}\\\nonumber
&=(-1)^{g+1}\frac{1}{2}\sum_{i=1}^{l}\sum_{\substack{a_{1}\geq
0\\a_{2}\geq 0\\b_{L \setminus
\{i\}}\geq 0}}\sum_{\substack{g_{1}+g_{2}=g\\
\mathcal{I}\cup\mathcal{J}=L\setminus\{i\}}}^{stable}\langle\tau_{a_{1}}\tau_{b_{\mathcal{I}}}\Lambda_{g_{1}}^{\vee}(1)
\rangle_{g_{1}}\langle\tau_{a_{2}}\tau_{b_{\mathcal{J}}}\Lambda_{g_{2}}^{\vee}(1)\rangle_{g_{2}}
[(g-a_{1}-a_{2}-|b_{L\setminus\{i,j\}}|-4)\\\nonumber
&\cdot\Psi_{a_{1}+1}^{a_{1}+1}(t_{i})\Psi_{a_{2}+1}^{a_{2}+1}(t_{i})
+\Psi_{a_{1}}^{a_{1}+1}(t_{i})\Psi_{a_{2}+1}^{a_{2}+1}(t_{i})+\Psi_{a_{1}+1}^{a_{1}+1}(t_{i})
\Psi_{a_{2}+1}^{a_{2}}(t_{i})]\Psi_{b_{L\setminus\{i\}
}}^{b_{L\setminus\{i\} }}(t_{L\setminus\{i\}})\\\nonumber
&+(-1)^{g+1}\frac{1}{2}\sum_{i=1}^{l}\sum_{\substack{a_{1}\geq
0\\a_{2}\geq 0\\b_{L \setminus
\{i\}}\geq 0}}\sum_{\substack{g_{1}+g_{2}=g\\
\mathcal{I}\cup\mathcal{J}=L\setminus\{i\}}}^{stable}\langle\tau_{a_{1}}\tau_{b_{\mathcal{I}}}\Lambda_{g_{1}}^{\vee}(1)
\rangle_{g_{1}}\\\nonumber
&\langle\tau_{a_{2}}\tau_{b_{\mathcal{J}}}\Lambda_{g_{2}}^{\vee}(1)\rangle_{g_{2}}
\Psi_{a_{1}}^{a_{1}+1}(t_{i})\Psi_{a_{2}+1}^{a_{2}+1}(t_{i})\sum_{j\neq
i}\Psi_{b_{j}}^{b_{j}-1}(t_{j})\Psi_{b_{L\setminus\{i,j\}}}^{b_{L\setminus\{i,j\}}}(t_{L\setminus\{i,j\}})
\end{align}

Thus, by $(73)=(74)+(75)+(76)$, we get a Hodge integral identity.
However, unfortunately, the only Hodge integral contained in this
identity is $\langle\tau_{b_{L}}\Lambda_{g}^{\vee}(1)\rangle_{g}$
because only the terms $a_{2g}(\lambda)=(-1)^g$ and
$a_{2g-1}(\lambda)=(-1)^{g}g$ was involved in the above calculation
which contains no more new information.

In order to get the Hodge integral identity with more than one
$\lambda$-class, we need to compare the coefficients of
$\tau^{2g+l-5}$ with the same procedure, we will arrive at a Hodge
integral identity which involves
$\langle\tau_{b_{L}}\Lambda_{g}^{\vee}(1)\rangle_{g}$ and
$\langle\tau_{b_{L}}\Lambda_{g}^{\vee}(1)\lambda_{1}\rangle_{g}$,
because
$a_{2g-2}(\lambda)=(-1)^{g}\left(\frac{g(g-1)}{2}-\lambda_{1}\right)$.

From such Hodge integral identity, we can calculate all the Hodge
integral of type $\langle\tau_{b_{L}} \lambda_{k}\lambda_{1}
\rangle_{g}$, $k=1,..,g$. However, more terms will appear in the
calculation of coefficients of $\tau^{2g+l-5}$ which make the
computation very complicated. so far, we have not written down the
explicit formula to calculate the Hodge integral of type
$\langle\tau_{b_{L}} \lambda_{k}\lambda_{1} \rangle_{g}$. But when
$k=g$, in the following subsection 5.4, we will give an explicit
formula for $\langle\tau_{b_{L}}\lambda_{g}\lambda_{1}\rangle_{g}$.

\subsection{Expansion of $\tau$ at 0}
Now, let us consider the expansion of
$\frac{\tau^{a}}{(\tau+1)^{b}}$ at $\tau=0$, we have
\begin{align*}
\frac{\tau^{a}}{(\tau+1)^{b}}=\sum_{i\geq 0}\binom{-b}{i}\tau^{a+i}
\end{align*}

After this expansion, it is easy to show that the lowest degree of
$\tau$ in $(64)$ is $l-2$. In fact, by
$a_{0}(\lambda)=(-1)^{g}\lambda_{g}\Lambda_{g}^{\vee}(-1)$ show in
$(48)$ and $\Lambda_{g}^{\vee}(1)\Lambda_{g}^{\vee}(-1)=(-1)^{g}$,
we have
\begin{align}
[\tau^{l-2}]LHS=-(l-1)\sum_{b_{L}\geq
0}\langle\tau_{b_{L}}\lambda_{g}\rangle_{g}\Psi_{b_{L}}^{0}(t_{L})
\end{align}
\begin{align}
[\tau^{l-2}]T_{1}=-\sum_{1\leq i<j \leq l}\sum_{\substack{a\geq 0\\
b_{L\setminus\{i,j\}}\geq 0}}\langle\tau_{a}\tau_{b_{L\setminus
\{i,j\}}}\lambda_{g}\rangle_{g}\Psi_{b_{L\setminus
\{i,j\}}}^{0}(t_{L\setminus\{i,j\}})\frac{(t_{j}-1)t_{i}^{2}\Psi_{a+1}^{0}(t_{i})-(t_{i}-1)t_{j}^{2}\Psi_{a+1}^{0}(t_{j})}
{t_{i}-t_{j}}
\end{align}
\begin{align}
[\tau^{l-2}]T_{2}=[\tau^{l-2}]T_{3}=0
\end{align}
Thus, we get
\begin{corollary}
\begin{align}
&\sum_{b_{L}\geq
0}\langle\tau_{b_{L}}\lambda_{g}\rangle_{g}\Psi_{b_{L}}^{0}(t_{L})=\frac{1}{l-1}\sum_{1\leq i<j \leq l}\sum_{\substack{a\geq 0\\
b_{L\setminus\{i,j\}}\geq 0}}\langle\tau_{a}\tau_{b_{L\setminus
\{i,j\}}}\lambda_{g}\rangle_{g}\Psi_{b_{L\setminus
\{i,j\}}}^{0}(t_{L\setminus\{i,j\}})\\\nonumber
&\cdot\frac{(t_{j}-1)t_{i}^{2}\Psi_{a+1}^{0}(t_{i})-(t_{i}-1)t_{j}^{2}\Psi_{a+1}^{0}(t_{j})}
{t_{i}-t_{j}}
\end{align}
\end{corollary}

Here, we introduce two notations. Let
$g(x_{1},..,x_{l})=\sum_{i_{k}\geq 0}a_{i_{1}i_{2}\cdots
i_{l}}x_{1}^{i_{1}}\cdots x_{l}^{i_{l}} \in Q[x_{1},..,x_{l}]$,
\begin{align*}
F_{d}(g(x_{1},\cdots,x_{l}))=\sum_{\sum_{k=1}^{l}i_{k}=d}a_{i_{1}i_{2}\cdots
i_{l}}x_{1}^{i_{1}}\cdots x_{l}^{i_{l}}
\end{align*}
and
\begin{align*}
[x_{1}^{j_{1}}\cdots
x_{l}^{j_{l}}]g(x_{1},\cdots,x_{l})=a_{j_{1}j_{2}\cdots j_{l}}.
\end{align*}

Now, let us use these two operator to corollary 5.2. By $(60)$, we
have
\begin{align*}
F_{2g-3+2l}(\text {$LHS$ of $(80)$})=(l-1)\sum_{b_{L}\geq 0}\langle
\tau_{b_{L}}\lambda\rangle_{g}b_{L}!y_{L}^{b_{L}+1}
\end{align*}
\begin{align*}
F_{2g-3+2l}(\text{$RHS$ of $(80)$})=\sum_{1\leq i<j \leq
l}\sum_{\substack{a\geq 0\\ b_{L}\geq
0}}\langle\tau_{b_{L\setminus\{i,j\}}}\lambda_{g}\rangle_{g}(a+1)!b_{L\setminus\{i,j\}}!t_{L\setminus
\{i,j\}}^{b_{L\setminus
\{i,j\}}+1}\frac{t_{j}t_{i}^{a+4}-t_{i}t_{j}^{a+4}}{t_{i}-t_{j}}
\end{align*}
Then,

\begin{align*}
[t_{L}^{b_{L}+1}]F_{2g-3+2l}(\text {$LHS$ of $(80)$})=(l-1)\langle
\tau_{b_{L}}\lambda\rangle_{g}b_{L}!
\end{align*}

\begin{align*}
[t_{L}^{b_{L}+1}]F_{2g-3+2l}(\text{$RHS$ of $(80)$})=\sum_{1\leq i<j
\leq
l}\langle\tau_{b_{i}+b_{j}-1}\tau_{b_{L\setminus\{i,j\}}}\lambda_{g}\rangle_{g}(b_{i}+b_{j})!b_{L\setminus
\{i,j\}}!
\end{align*}

Hence, we have
\begin{align*}
\langle \tau_{b_{L}}\lambda\rangle_{g}=\frac{1}{l-1}\sum_{1\leq i<j
\leq
l}\langle\tau_{b_{i}+b_{j}-1}\tau_{b_{L\setminus\{i,j\}}}\lambda_{g}\rangle_{g}\frac{(b_{i}+b_{j})!}{b_{i}!b_{j}!}
\end{align*}
Induction on $l$, we get the $\lambda_{g}$ integral,
\begin{align*}
\langle \tau_{b_{L}}\lambda\rangle_{g}=\binom{2g-3+l}{b_{L}}c_{g}
\end{align*}

Now, let us calculate the next degree of $\tau$. By
$a_{0}(\lambda)=(-1)^{g}\lambda_{g}\Lambda_{g}^{\vee}(-1)$ showed in
$(48)$ and $\Lambda_{g}^{\vee}(1)\Lambda_{g}^{\vee}(-1)=(-1)^{g}$,
we have
\begin{align}
[\tau^{l-1}]LHS&=-\sum_{b_{L}\geq
0}\langle\tau_{b_{L}}\lambda_{g}\rangle_{g}l\left(
-(|b_{L}|+1)\Psi_{b_{L}}^{0}(t_{L})+\sum_{j=1}\Psi_{b_{j}}^{1}(t_{j})\Psi_{b_{L\setminus\{j\}}}^{0}(t_{L\setminus\{j\}})\right)\\\nonumber
&-\sum_{b_{L}\geq
0}\langle\tau_{b_{L}}\lambda_{g}\rangle_{g}\sum_{i=1}^{l}(t_{i}^{2}-t_{i})\frac{\partial}{\partial
t_{i}}\Psi_{b_{i}}^{0}(t_{i})\Psi_{b_{L\setminus\{i\}}}^{0}(t_{L\setminus\{i\}})-\sum_{b_{L}\geq
0}\langle\tau_{b_{L}}\Lambda_{g}^{\vee}(1)a_{1}(\lambda)\rangle_{g}l\Psi_{b_{L}}^{0}(t_{L})
\end{align}
\begin{align}
&[\tau^{l-1}]T_{1}\\\nonumber
&=-\sum_{1\leq i<j \leq l}\sum_{\substack{a\geq 0\\
b_{L\setminus\{i,j\}}\geq 0}}\langle\tau_{a}\tau_{b_{L\setminus
\{i,j\}}}\lambda_{g}\rangle_{g}\Psi_{b_{L\setminus
\{i,j\}}}^{0}(t_{L\setminus\{i,j\}})[\frac{(t_{j}-1)t_{i}\left(\Psi_{a+1}^{1}(t_{i})+(t_{i}-|b_{L\setminus\{i,j\}}|-a-3)
\Psi_{a+1}^{0}(t_{i})\right)} {t_{i}-t_{j}}\\\nonumber &-\frac{
(t_{i}-1)t_{j}\left(\Psi_{a+1}^{1}(t_{j})+(t_{j}-|b_{L\setminus\{i,j\}}|-a-3)\Psi_{a+1}^{0}(t_{j})\right)}
{t_{i}-t_{j}}]\\\nonumber &-\sum_{1\leq i<j \leq l}\sum_{\substack{a\geq 0\\
b_{L\setminus\{i,j\}}\geq 0}}\langle\tau_{a}\tau_{b_{L\setminus
\{i,j\}}}\lambda_{g}\rangle_{g}\sum_{r\neq
i,j}\Psi_{b_{r}}^{1}(t_{r})\Psi_{b_{L\setminus
\{i,j\}}}^{0}(t_{L\setminus\{i,j\}})\frac{(t_{j}-1)t_{i}\Psi_{a+1}^{0}(t_{i})-(t_{i}-1)t_{j}\Psi_{a+1}^{0}(t_{j})}
{t_{i}-t_{j}}\\\nonumber &
-\sum_{1\leq i<j \leq l}\sum_{\substack{a\geq 0\\
b_{L\setminus\{i,j\}}\geq 0}}\langle\tau_{a}\tau_{b_{L\setminus
\{i,j\}}}\Lambda_{g}^{\vee}(1)a_{1}(\lambda)\rangle_{g}\Psi_{b_{L\setminus\{i,j\}}}^{0}(t_{L\setminus\{i,j\}})
\frac{(t_{j}-1)t_{i}\Psi_{a+1}^{0}(t_{i})-(t_{i}-1)t_{j}\Psi_{a+1}^{0}(t_{j})}{t_{i}-t_{j}}
\end{align}
\begin{align}
[\tau^{l-1}]T_{2}=0
\end{align}
\begin{align}
[\tau^{l-1}]T_{3}=-\frac{1}{2}\sum_{i=1}^{l}\sum_{\substack{a_{1}\geq
0\\a_{2}\geq 0\\b_{L\setminus\{i\}}\geq
0}}\sum_{\substack{g_{1}+g_{2}=g\\|\mathcal{I}|\cup|\mathcal{J}|=L\setminus\{i\}}}^{stable}\langle\tau_{a_{1}}\tau_{b_{\mathcal{I}}}
\lambda_{g_{1}}\rangle_{g_{1}}\langle\tau_{a_{2}}\tau_{b_{\mathcal{J}}}
\lambda_{g_{2}}\rangle_{g_{2}}\Psi_{a_{1}+1}^{0}(t_{i})\Psi_{a_{2}+1}^{0}(t_{i})\Psi_{b_{L\setminus\{i\}}}^{0}(t_{L\setminus\{i\}})
\end{align}

Recall formula $(48)$,
$a_{1}(\lambda)=\sum_{m=1}^{g}m\lambda_{g-m}\lambda_{g}-(-1)^{g}\Lambda_{g}^{\vee}(-1)\lambda_{g-1}$
and Mumford's relation
$\Lambda_{g}^{\vee}(1)\Lambda_{g}^{\vee}(-1)=(-1)^g$. We can write
\begin{align}
\Lambda_{g}^{\vee}(1)a_{1}(\lambda)=\sum_{d=g-1}^{3g-3}P_{d}(\lambda)
\end{align}
where $P_{d}(\lambda)$ is some combinatoric of $\lambda$-class with
degree $d$. For example
\begin{align}
P_{g-1}(\lambda)=\lambda_{g-1}, P_{g}(\lambda)=g\lambda_{g},
P_{g+1}(\lambda)=-\lambda_{g}\lambda_{1},\cdots,
P_{3g-3}(\lambda)=(-1)^{g+1}\lambda_{g}\lambda_{g-1}\lambda_{g-2}
\end{align}

By the above calculation, substituting $(85)$ to the identity
$(81)=(82)+(84)$, we have
\begin{corollary}
\begin{align}
&\sum_{b_{L}\geq 0}\langle\tau_{b_{L}}\lambda_{g}\rangle_{g}l\left(
-(|b_{L}|+1)\Psi_{b_{L}}^{0}(t_{L})+\sum_{j=1}\Psi_{b_{j}}^{1}(t_{j})\Psi_{b_{L\setminus\{j\}}}^{0}(t_{L\setminus\{j\}})\right)\\\nonumber
&+\sum_{b_{L}\geq
0}\langle\tau_{b_{L}}\lambda_{g}\rangle_{g}\sum_{i=1}^{l}(t_{i}^{2}-t_{i})\frac{\partial}{\partial
t_{i}}\Psi_{b_{i}}^{0}(t_{i})\Psi_{b_{L\setminus\{i\}}}^{0}(t_{L\setminus\{i\}})+\sum_{b_{L}\geq
0}\langle\tau_{b_{L}}\sum_{d=g-1}^{3g-3}P_{d}(\lambda)\rangle_{g}l\Psi_{b_{L}}^{0}(t_{L})\\\nonumber
&=\sum_{1\leq i<j \leq l}\sum_{\substack{a\geq 0\\
b_{L\setminus\{i,j\}}\geq 0}}\langle\tau_{a}\tau_{b_{L\setminus
\{i,j\}}}\lambda_{g}\rangle_{g}\Psi_{b_{L\setminus
\{i,j\}}}^{0}(t_{L\setminus\{i,j\}})[\frac{(t_{j}-1)t_{i}\left(\Psi_{a+1}^{1}(t_{i})+(t_{i}-|b_{L\setminus\{i,j\}}|-a-3)
\Psi_{a+1}^{0}(t_{i})\right)} {t_{i}-t_{j}}\\\nonumber & -\frac{
(t_{i}-1)t_{j}\left(\Psi_{a+1}^{1}(t_{j})+(t_{j}-|b_{L\setminus\{i,j\}}|-a-3)\Psi_{a+1}^{0}(t_{j})\right)}
{t_{i}-t_{j}}]\\\nonumber &+\sum_{1\leq i<j \leq l}\sum_{\substack{a\geq 0\\
b_{L\setminus\{i,j\}}\geq 0}}\langle\tau_{a}\tau_{b_{L\setminus
\{i,j\}}}\lambda_{g}\rangle_{g}\sum_{r\neq
i,j}\Psi_{b_{r}}^{1}(t_{r})\Psi_{b_{L\setminus
\{i,j\}}}^{0}(t_{L\setminus\{i,j\}})\frac{(t_{j}-1)t_{i}\Psi_{a+1}^{0}(t_{i})-(t_{i}-1)t_{j}\Psi_{a+1}^{0}(t_{j})}
{t_{i}-t_{j}}\\\nonumber &+
\sum_{1\leq i<j \leq l}\sum_{\substack{a\geq 0\\
b_{L\setminus\{i,j\}}\geq 0}}\langle\tau_{a}\tau_{b_{L\setminus
\{i,j\}}}\sum_{d=g-1}^{3g-3}P_{d}(\lambda)\rangle_{g}\Psi_{b_{L\setminus\{i,j\}}}^{0}(t_{L\setminus\{i,j\}})
\frac{(t_{j}-1)t_{i}\Psi_{a+1}^{0}(t_{i})-(t_{i}-1)t_{j}\Psi_{a+1}^{0}(t_{j})}{t_{i}-t_{j}}\\\nonumber
&+\frac{1}{2}\sum_{i=1}^{l}\sum_{\substack{a_{1}\geq 0\\a_{2}\geq
0\\b_{L\setminus\{i\}}\geq
0}}\sum_{\substack{g_{1}+g_{2}=g\\|\mathcal{I}|\cup|\mathcal{J}|=L\setminus\{i\}}}^{stable}\langle\tau_{a_{1}}\tau_{b_{\mathcal{I}}}
\lambda_{g_{1}}\rangle_{g_{1}}\langle\tau_{a_{2}}\tau_{b_{\mathcal{J}}}
\lambda_{g_{2}}\rangle_{g_{2}}\Psi_{a_{1}+1}^{0}(t_{i})\Psi_{a_{2}+1}^{0}(t_{i})\Psi_{b_{L\setminus\{i\}}}^{0}(t_{L\setminus\{i\}})
\end{align}
\end{corollary}

Since
$\langle\tau_{b_{L}}\lambda_{g}\rangle_{g}=\binom{2g-3+l}{b_{L}}c_{g}$,
thus all the Hodge integral of type
\begin{align}
\langle\tau_{b_{L}}\sum_{d=g-1}^{3g-3}P_{d}(\lambda)\rangle_{g}
\end{align}
can be calculated from formula $(87)$ in corollary 4.3.
\subsection{Explicit expression for Hodge integral $\langle\tau_{b_{L}}\lambda_{g}\lambda_{1}\rangle_{g}$}
Although corollary 4.3 says that all the Hodge integral $(88)$ can
be computed, it is not easy to write down their explicit formula. In
this subsection, we will show how to obtain an explicit formula for
$\langle\tau_{b_{L}}\lambda_{g}\lambda_{1}\rangle_{g}$.

Both sides of $(87)$ belong to $Q[t_{1},..,t_{l}]$. We have known
that
$\langle\tau_{b_L}\lambda_{g}\rangle_{g}=\binom{2g-3+l}{b_L}c_{g}$,
and $\langle\tau_{b_L}\lambda_{g-1}\rangle_{g}$ can also be
calculated easily from a recursion formula showed in \cite{Zhu}.
Thus, we have

\begin{theorem}
If $\sum_{i=1}^{l}b_{i}=2g-4+l$, there exists a constant
$C(g,l,b_1,..,b_{l})$ related to $g,l,b_{1},..,b_{l}$, such that
\begin{align}
\langle\tau_{b_{L}}\lambda_{g}\lambda_{1}\rangle_{g}=\frac{1}{l}\sum_{1\leq
i<j \leq l}\langle\tau_{b_{i}+b_{j}-1}\tau_{b_{L\setminus
\{i,j\}}}\lambda_{g}\lambda_{1}\rangle_{g}\frac{(b_{i}+b_{j})!}{b_{i}!b_{j}!}+C(g,l,b_1,..b_{l})
\end{align}
$C(g,l,b_1,..,b_{l})$ is a very complicated combinatoric constant
which is given at Appendix B.
\end{theorem}
\begin{proof}
Taking all the terms with degree $2g+2l-4$ in formula $(87)$. we
have
\begin{align}
F_{2g+2l-4}(\text{$LHS$ of $(87)$})=-l\sum_{b_{L}\geq
0}\langle\tau_{b_{L}}\lambda_{g}\lambda_{1}\rangle_{g}b_{L}!t_{L}^{b_{L}+1}+G_{1}(t_1,..,t_{l})
\end{align}
\begin{align}
&F_{2g+2l-4}(\text{$RHS$ of $(87)$})=-\sum_{1\leq i<j \leq
l}\sum_{\substack{a\geq 0\\b_{L\setminus\{i,j\}}\geq
0}}\langle\tau_{a}\tau_{b_{L\setminus\{i,j\}}}\lambda_{g}\lambda_{1}\rangle_{g}\\\nonumber
&\cdot b_{L\setminus\{i,j\}}!(a+1)!
\sum_{m=0}^{a+1}t_{L\setminus\{i,j\}}^{b_{L\setminus\{i,j\}}+1}t_{i}^{a+2-m}t_{j}^{m+1}+G_{2}(t_1,..,t_{l})
\end{align}
where
$G_{i}(t_1,..,t_{l})=\sum_{|d_{L}|=2g+2l-4}C_{i}(g,l,d_1,..,d_{l})t_{L}^{d_{L}}\in
\mathbb{Q}[t_{1},..t_{l}]$, $C_{i}(g,l,d_1,..,d_{l})$ is a
combinatoric constant related to $g,l,d_1,..,d_{l}$. Thus, we have
\begin{align}
&\sum_{b_{L}\geq
0}\langle\tau_{b_{L}}\lambda_{g}\lambda_{1}\rangle_{g}b_{L}!t_{L}^{b_{L}+1}
=\frac{1}{l}\sum_{1\leq i<j \leq l}\sum_{\substack{a\geq
0\\b_{L\setminus\{i,j\}}\geq
0}}\langle\tau_{a}\tau_{b_{L\setminus\{i,j\}}}\lambda_{g}\lambda_{1}\rangle_{g}\\\nonumber&
\cdot b_{L\setminus\{i,j\}}!(a+1)!
\sum_{m=0}^{a+1}t_{L\setminus\{i,j\}}^{b_{L\setminus\{i,j\}}+1}t_{i}^{a+2-m}t_{j}^{m+1}+G(t_1,..,t_{l})
\end{align}
where
$G(t_1,..,t_{l})=\sum_{|d_{L}|=2g+2l-4}\hat{C}(g,l,b_1,..,b_{l})t_{L}^{d_{L}}=\frac{1}{l}(G_{1}(t_1,..t_{l})-G_{2}(t_1,..,t_{l}))$

After taking the coefficients of $(92)$,
\begin{align}
&\langle\tau_{b_{L}}\lambda_{g}\lambda_{1}\rangle_{g}
=\frac{1}{l}\sum_{1\leq i<j\leq
l}\langle\tau_{b_{i}+b_{j}-1}\tau_{b_{L\setminus\{i,j\}}}\lambda_{g}\lambda_{1}\rangle_{g}
\frac{(b_{i}+b_{j})!}{b_{i}!b_{j}!}+\frac{\hat{C}(g,l,b_1+1,..,b_{l}+1)}{b_{L}!}
\end{align}
Let
$C(g,l,b_1,..,b_{l})=\frac{\hat{C}(g,l,b_1+1,..,b_{l}+1)}{b_{L}!}$,
thus, the theorem is proved.
\end{proof}
We know that the initial value
$\langle\tau_{2g-3}\lambda_{g}\lambda_{1}\rangle_{g}=\frac{1}{12}[g(2g-3)b_{g}+b_{1}b_{g-1}]$
has been computed by Y. Li \cite{Li}. Therefore, from the recursion
formula (62), we can compute out all the Hodge integral of type $
\langle\tau_{b_{L}}\lambda_{g}\lambda_{1}\rangle_{g} $.

\section{Conclusion}
The original cut-and-join equation of Mari\~no-Vafa formula is a
complicated formula about the partition which can be changed to a
polynomial identity with the help of the Laplace transform in this
paper or the symmetrization method in \cite{Ch}. This polynomial
identity has some applications because it is manageable. We derived
some corollaries about the Hodge integral identities. In fact, we
have found an algorithm to calculate the Hodge integral appearing in
Mari\~no-Vafa formula.

Step1, expanding the $\tau$ at some special point, and collect the
corresponding level of $\tau$ in theorem 1.1.

Step2, taking certain degrees of $t_{L}$, we will get the
corresponding Hodge integral appearing in Mari\~no-Vafa formula.

In this paper, we only calculate four cases by using step1 and
calculate out a new Hodge integral via Step2 from the result of last
case after doing step1. We have not considered other cases here,
because the computation will be more complicated. We hope that this
algorithm can be compiled to a computer program.

We note that, the recursion formulas to calculate
$\langle\tau_{b_L}\lambda_{g-1}\rangle_{g}$ \cite{Zhu} and
$\langle\tau_{b_{L}}\lambda_{g}\lambda_{1}\rangle_{g}$ in theorem
1.5 have the similar recursion structure. For a given $g$, Hodge
integral with $l$-point will reduce to 1-point after only $l$ times
recursions. Thus, it is a very effective recursion relation. In
fact, many Hodge integral recursions will have this type structure
if they are calculated out by above algorithm. Thus it is
interesting to consider the following combinatoric problem.

Let $Q_{l}(b_{L})=Q_{l}(b_{1},b_{2},..,b_{l})$ be a symmetric
function on $(b_{1},b_{2},..,b_{l})\in (\mathbb{Z^{+}})^{l}$ which
is defined by the following recursion relation.

For given $g, k$, \makeatletter
\let\@@@alph\@alph
\def\@alph#1{\ifcase#1\or \or $'$\or $''$\fi}\makeatother
\begin{subnumcases}
{Q_{1}(b)=} constant, &$b=3g-2-k$, \label{eq:a1}\\\nonumber 0,
&$others$.\label{eq:a2}\nonumber
\end{subnumcases}
\makeatletter\let\@alph\@@@alph\makeatother

\makeatletter
\let\@@@alph\@alph
\def\@alph#1{\ifcase#1\or \or $'$\or $''$\fi}\makeatother
\begin{subnumcases}
{Q_{l}(b_{L})=} \sum_{1\leq i<j\leq
l}Q_{l-1}(b_{i}+b_{j}-1,b_{L\setminus\{i,j\}})A_{l}(b_{L})+B_{l}(b_{L})
&$\sum_{i=1}^{l}b_{i}=3g-3+l-k$, \label{eq:a1}\\\nonumber 0,
&$others$.\label{eq:a2}
\end{subnumcases}
\makeatletter\let\@alph\@@@alph\makeatother

Where $A_{l}(b_{L})=A_{l}(b_{1},..,b_{l})$ and
$B_{l}(b_{L})=B_{l}(b_{1},..,b_{l})$ are some fixed functions
defined on $(b_{1},b_{2},..,b_{l})\in (\mathbb{Z^{+}})^{l}$.

It is interesting to study the properties of $Q_{l}(b_{L})$ defined
above. We hope that this combinatoric structure of Hodge integral
will be studied further in future.

\section{Appendix}
\subsection{Appendix A}
In this appendix, we will show how to derive our main result in
theorem 3.4  from the symmetrized cut-and-join equation of the
Mari\~no-Vafa formula \cite{Ch}.

when $b\geq 0$, Let
$$\Psi_{b}(y;\tau)=((y^2-y)(\frac{y\tau+1}{\tau+1})\frac{d}{dy})^b(\frac{y-1}{\tau+1})
=\sum_{k=0}^{b}\frac{\tau^k}{(\tau+1)^{b+1}}\psi_{b}^{k}(y)$$
Then the symmetrized partition function of Mari\~no-Vafa formula can
be written as
\begin{align}
W_{g,n}(y_{1},..,y_{n};\tau)=-(\tau(\tau+1))^{n-1}\sum_{\substack{b_{i}\geq
0\\i=1,..,n}}\langle\prod_{i=1}^{n}\tau_{b_{i}}
\Gamma_{g}(\tau)\rangle_{g}\prod_{i=1}^{n}\Psi_{b_{i}}(y_{i};\tau)
\end{align}

The main result of \cite{Ch}, i.e Theorem 3 in that paper, says
\begin{theorem}
The symmetrized generating series $W_{g,n}(y_{1},..,y_{n};\tau)$ is
a polynomial in the $y_{i}$ variables of total degree $6g-6+3n$ and
satisfies the symmetrized cut--and-join equation.
\begin{align*}
(\frac{\partial}{\partial
\tau}+\sum_{l=1}^{n}\frac{y_{l}^{2}-y_{l}}{\tau+1}\frac{\partial}{\partial
y_{l}})W_{g,n}(y_{1},..,y_{n};\tau)=T_{1}+T_{2}+T_{3}
\end{align*}
where
$$T_{1}=sym_{1,1}^{y}\frac{y_{1}(y_{2}-1)}{y_{1}-y_{2}}(\frac{y_{1}\tau+1}{\tau+1})D_{y_{1}}^{\tau}
W_{g,n-1}(y_{1},y_{3},..,y_{n};\tau)$$
$$
T_{2}=-\frac{1}{2}\sum_{l=1}^{n}D_{y_{l}}^{\tau}D_{y_{n+1}}^{\tau}W_{g-1,n+1}(y_{1},..,y_{n},y_{n+1};\tau)|_{y_{n+1}=y_{l}}
$$
$$
T_{3}=-\frac{1}{2}\sum_{\substack{g_{1}+g_{2}=g\\n_{1}+n_{2}=n+1\\2g_{1}-2+n_{1}>0\\2g_{2}-2+n_{2}>0}}
sym_{1,n_{1}-1}^{y}D_{y_{1}}^{\tau}
W_{g_{1},n_{1}}(y_{1},..,y_{n_{1}};\tau)D_{y_{1}}^{\tau}W_{g_{2},n_{2}}(y_{1},y_{n_{1}+1},..,y_{n};\tau)
$$
We need to explain the notations appear in above theorem. Where
$D_{y}^{\tau}=(y^2-y)(\frac{y\tau+1}{\tau+1})\frac{\partial}{\partial
y}$. For $i,j\geq 0$, $i+j\leq n$, let $sym_{i,j}^{x}$ be the
mapping, applied to a series in $x_{1},..,x_{n}$, given by
$$
sym_{i,j}^{x}f(x_{1},..,x_{n})=\sum_{\mathcal{R},\mathcal{S},\mathcal{T}}f(\mathbf{x}_{\mathcal{R}},
\mathbf{x}_{\mathcal{S}},\mathbf{x}_{\mathcal{T}})
$$
where the sum is over all ordered partitions
$(\mathcal{R},\mathcal{S},\mathcal{T})$ of $\{1,2,..,n\}$,
$\mathcal{R}=\{x_{r_{1}},..,x_{r_{i}} \}$,
$\mathcal{S}=\{x_{s_{1}},..,x_{s_{j}}\}$,
$\mathcal{T}=\{x_{t_{1}},..,x_{t_{n-i-j}}\}$ and
$(\mathbf{x}_{\mathcal{R}},\mathbf{x}_{\mathcal{S}},\mathbf{x}_{\mathcal{T}})
=(x_{r_{1}},..,x_{r_{i}},x_{s_{1}},..,x_{s_{j}},x_{t_{1}},..,x_{t_{n-i-j}})$,
where $r_{1}<\cdots<r_{i}$, $s_{1}<\cdots<s_{j}$, and
$t_{1}<\cdots<t_{n-i-j}$.
\end{theorem}
Then, by equation $(96)$, we have
\begin{align*}
&LHS=-(\tau^2+\tau)^{n-2}\sum_{\substack{b_{i}\geq
0\\i=1,..,n}}\left((n-1)(2\tau+1)\langle\prod_{i=1}^{n}\tau_{b_{i}}\Gamma_{g}(\tau)\rangle_{g}+(\tau^{2}+\tau)
\langle\prod_{i=1}^{n}\tau_{b_{i}}\frac{d}{d\tau}
\Gamma_{g}(\tau)\rangle_{g}\right)\prod_{i=1}^{n}\Psi_{b_{i}}(y_{i};\tau)\\
&-(\tau^{2}+\tau)^{n-1}\sum_{\substack{b_{i}\geq
0\\i=1,2,..,n}}\langle\prod_{i=1}^{n}\tau_{b_{i}}
\Gamma_{g}(\tau)\rangle_{g}\sum_{l=1}^{n}(\frac{\partial}{\partial
\tau}\Psi_{b_{l}}(y_{l};\tau)+\frac{1}{y_{l}\tau+1}\Psi_{b_{l}+1}(y_{l};\tau))\prod_{\substack{k=1\\k\neq
l}}^{n}\Psi_{b_{k}}(y_{k};\tau)
\end{align*}
\begin{align*}
&T_{1}=-\frac{(\tau^2+\tau)^{n-2}}{\tau+1}\sum_{1\leq i<j\leq
n}\sum_{\substack{a\geq 0\\b_{k}\geq 0\\k\neq
i,j}}\langle\tau_{a}\prod_{\substack{k=1\\k\neq
i,j}}^{n}\tau_{b_{k}}\Gamma_{g}(\tau)\rangle_{g}\prod_{\substack{k=1\\k\neq
i,j}}^n\Psi_{b_{k}}(y_{k};\tau)\\
&\times\frac{(y_{j}-1)(y_{i}^2\tau+y_{i})\Psi_{a+1}(y_{i};\tau)-(y_{i}-1)(y_{j}^2\tau+y_{j})\Psi_{a+1}(y_{j};\tau)}
{y_{i}-y_{j}}
\end{align*}
\begin{align*}
&T_{2}=\frac{1}{2}(\tau^2+\tau)^n\sum_{i=1}^{n}\sum_{\substack{a_{1}\geq
0\\a_{2}\geq 0\\b_{k}\geq 0,k\neq
i}}\langle\tau_{a_{1}}\tau_{a_{2}}\prod_{\substack{k=1\\k\neq
i}}^{n}\tau_{b_{k}}\Gamma_{g-1}(\tau)\rangle_{g-1}\Psi_{a_{1}+1}(y_{i};\tau)\Psi_{a_{2}+1}(y_{i};\tau)\prod_{\substack{k=1\\k\neq
i}}^{n}\Psi_{b_{k}}(y_{k};\tau)
\end{align*}
\begin{align*}
&T_{3}=-\frac{(\tau^2+\tau)^{n-1}}{2}\sum_{i=1}^{n}\sum_{\substack{a_{1}\geq 0\\a_{2}\geq 0\\b_{k}\geq 0,k\neq i}}
\sum_{\substack{g_{1}+g_{2}=g\\ \mathcal{I}\coprod\mathcal{J}=\{1,.,\hat{i},.n\}\\2g_{1}-1+|\mathcal{I}|>0\\
2g_{2}-1+|\mathcal{J}|>0}}\langle\tau_{a_{1}}\prod_{k\in
\mathcal{I}}\tau_{b_{k}}\Gamma_{g_{1}}(\tau)\rangle_{g_{1}}\\
&\times\langle\tau_{a_{2}}\prod_{k\in \mathcal{J}}\tau_{b_{k}}
\Gamma_{g_{2}}(\tau)\rangle_{g_{2}}
\Psi_{a_{1}}(y_{i};\tau)\Psi_{a_{2}}(y_{i};\tau)\prod_{\substack{k=1\\k\neq
i}}^{n}\Psi_{b_{i}}(y_{i};\tau)
\end{align*}

Then, by theorem 5.1, we arrive our main results in section 3, i.e.
theorem 3.4.
\subsection{Appendix B}
In this appendix, we will calculate the constant $C(g,l,b_1,..,b_l)$
which appears in theorem 1.5. The computation is verbose and boring.

Let us write the formula $(77)$ of Corollary 5.3 as
\begin{align}
&L_1+L_2+L_3+L_4+L_5+L_6+L_7\\\nonumber
&=R_1+R_2+R_3+R_4+R_5+R_6+R_7+R_8+R_9+R_{10}+R_{11}+R_{12}+R_{13}+R_{14}+R_{15}+R_{16}
\end{align}
where
\begin{align*}
&L_{1}=\sum_{b_{L}\geq
0}\langle\tau_{b_L}\lambda_{g}\rangle_{g}l(g-|b_L|-1)\Psi_{b_L}^{0}(t_{L})\\
&L_{2}=\sum_{b_{L}\geq
0}\langle\tau_{b_L}\lambda_{g}\rangle_{g}l\sum_{j=1}^{l}\Psi_{b_j}^{1}(t_j)\Psi_{b_{L\setminus\{j\}}}^{0}(t_{L\setminus\{j\}})\\
&L_{3}=\sum_{b_{L}\geq
0}\langle\tau_{b_L}\lambda_{g}\rangle_{g}\sum_{j=1}^{l}t_{j}^{2}\frac{\partial}{\partial
t_{j}}\Psi_{b_j}^{0}(t_j)\Psi_{b_{L\setminus
\{j\}}}^{0}(t_{L\setminus\{j\}})\\
&L_{4}=-\sum_{b_{L}\geq
0}\langle\tau_{b_L}\lambda_{g}\rangle_{g}\sum_{j=1}^{l}t_{j}\frac{\partial}{\partial
t_{j}}\Psi_{b_j}^{0}(t_j)\Psi_{b_{L\setminus
\{j\}}}^{0}(t_{L\setminus\{j\}})\\
&L_{5}=-\sum_{b_{L}\geq
0}\langle\tau_{b_L}\lambda_{g-1}\rangle_{g}l\Psi_{b_L}^{0}(t_L)\\
&L_{6}=-\sum_{b_{L}\geq
0}\langle\tau_{b_L}\lambda_{g}\lambda_{1}\rangle_{g}l\Psi_{b_L}^{0}(t_{L})\\
&L_{7}=\sum_{b_{L}\geq
0}\langle\tau_{b_L}\sum_{d=g+2}^{3g-3}P_{d}(\lambda)\rangle_{g}l\Psi_{b_L}^{0}(t_L)
\end{align*}
\begin{align*}
&R_{1}=\sum_{1\leq i<j \leq l}\sum_{\substack{a\geq
0\\b_{L\setminus\{i,j\}}\geq
0}}\langle\tau_{a}\tau_{b_{L\setminus\{i,j\}}}\lambda_{g}
\rangle_{g}\Psi_{b_{L\setminus\{i,j\}}}^{0}(t_{L\setminus\{i,j\}})\frac{t_jt_i\Psi_{a+1}^{1}(t_i)-t_it_j\Psi_{a+1}^{1}(t_j)}{t_{i}-t_j}\\
&R_{2}=\sum_{1\leq i<j \leq l}\sum_{\substack{a\geq
0\\b_{L\setminus\{i,j\}}\geq
0}}\langle\tau_{a}\tau_{b_{L\setminus\{i,j\}}}\lambda_{g}
\rangle_{g}\Psi_{b_{L\setminus\{i,j\}}}^{0}(t_{L\setminus\{i,j\}})\frac{-t_i\Psi_{a+1}^{1}(t_i)+t_j\Psi_{a+1}^{1}(t_j)}{t_{i}-t_j}\\
&R_{3}=\sum_{1\leq i<j \leq l}\sum_{\substack{a\geq
0\\b_{L\setminus\{i,j\}}\geq
0}}\langle\tau_{a}\tau_{b_{L\setminus\{i,j\}}}\lambda_{g}
\rangle_{g}\Psi_{b_{L\setminus\{i,j\}}}^{0}(t_{L\setminus\{i,j\}})\frac{t_jt_{i}^{2}\Psi_{a+1}^{0}(t_i)-t_it_{j}^{2}
\Psi_{a+1}^{0}(t_j)}{t_{i}-t_j}\\
&R_{4}=\sum_{1\leq i<j \leq l}\sum_{\substack{a\geq
0\\b_{L\setminus\{i,j\}}\geq
0}}\langle\tau_{a}\tau_{b_{L\setminus\{i,j\}}}\lambda_{g}
\rangle_{g}\Psi_{b_{L\setminus\{i,j\}}}^{0}(t_{L\setminus\{i,j\}})\frac{-t_{i}^{2}\Psi_{a+1}^{0}(t_i)+t_{j}^{2}\Psi_{a+1}^{0}(t_j)}{t_{i}-t_j}\\
\end{align*}
\begin{align*}
&R_{5}=\sum_{1\leq i<j \leq l}\sum_{\substack{a\geq
0\\b_{L\setminus\{i,j\}}\geq
0}}\langle\tau_{a}\tau_{b_{L\setminus\{i,j\}}}\lambda_{g}
\rangle_{g}\Psi_{b_{L\setminus\{i,j\}}}^{0}(t_{L\setminus\{i,j\}})(|b_{L\setminus \{i,j\}}|+a+3)\frac{-t_jt_i\Psi_{a+1}^{0}(t_i)
+t_it_j\Psi_{a+1}^{0}(t_j)}{t_{i}-t_j}\\
&R_{6}=\sum_{1\leq i<j \leq l}\sum_{\substack{a\geq
0\\b_{L\setminus\{i,j\}}\geq
0}}\langle\tau_{a}\tau_{b_{L\setminus\{i,j\}}}\lambda_{g}
\rangle_{g}\Psi_{b_{L\setminus\{i,j\}}}^{0}(t_{L\setminus\{i,j\}})(|b_{L\setminus
\{i,j\}}|+a+3)\frac{t_i\Psi_{a+1}^{0}(t_i)-t_j\Psi_{a+1}^{0}(t_j)}{t_{i}-t_j}
\end{align*}
\begin{align*}
&R_{7}=\sum_{1\leq i<j \leq l}\sum_{\substack{a\geq
0\\b_{L\setminus\{i,j\}}\geq
0}}\langle\tau_{a}\tau_{b_{L\setminus\{i,j\}}}\lambda_{g}
\rangle_{g}\sum_{r\in
L\setminus\{i,j\}}\Psi_{b_r}^{1}(t_r)\Psi_{b_{L\setminus\{i,j,r\}}}^{0}(t_{L\setminus\{i,j,r\}})\frac{t_jt_i\Psi_{a+1}^{0}(t_i)
-t_it_j\Psi_{a+1}^{0}(t_j)}{t_{i}-t_j}\\
&R_{8}=\sum_{1\leq i<j \leq l}\sum_{\substack{a\geq
0\\b_{L\setminus\{i,j\}}\geq
0}}\langle\tau_{a}\tau_{b_{L\setminus\{i,j\}}}\lambda_{g}
\rangle_{g}\sum_{r\in
L\setminus\{i,j\}}\Psi_{b_r}^{1}(t_r)\Psi_{b_{L\setminus\{i,j,r\}}}^{0}(t_{L\setminus\{i,j,r\}})\frac{-t_i\Psi_{a+1}^{0}(t_i)
+t_j\Psi_{a+1}^{0}(t_j)}{t_{i}-t_j}
\end{align*}
\begin{align*}
&R_{9}=-\sum_{1\leq i<j \leq l}\sum_{\substack{a\geq
0\\b_{L\setminus\{i,j\}}\geq
0}}\langle\tau_{a}\tau_{b_{L\setminus\{i,j\}}}\lambda_{g-1}
\rangle_{g}\Psi_{b_{L\setminus\{i,j\}}}^{0}(t_{L\setminus\{i,j\}})\frac{t_jt_i\Psi_{a+1}^{0}(t_i)
-t_it_j\Psi_{a+1}^{0}(t_j)}{t_{i}-t_j}\\
&R_{10}=-\sum_{1\leq i<j \leq l}\sum_{\substack{a\geq
0\\b_{L\setminus\{i,j\}}\geq
0}}\langle\tau_{a}\tau_{b_{L\setminus\{i,j\}}}\lambda_{g-1}
\rangle_{g}\Psi_{b_{L\setminus\{i,j\}}}^{0}(t_{L\setminus\{i,j\}})\frac{-t_i\Psi_{a+1}^{0}(t_i)
+t_j\Psi_{a+1}^{0}(t_j)}{t_{i}-t_j}
\end{align*}
\begin{align*}
&R_{11}=\sum_{1\leq i<j \leq l}\sum_{\substack{a\geq
0\\b_{L\setminus\{i,j\}}\geq
0}}\langle\tau_{a}\tau_{b_{L\setminus\{i,j\}}}\lambda_{g}
\rangle_{g}g\Psi_{b_{L\setminus\{i,j\}}}^{0}(t_{L\setminus\{i,j\}})\frac{t_jt_i\Psi_{a+1}^{0}(t_i)
-t_it_j\Psi_{a+1}^{0}(t_j)}{t_{i}-t_j}\\
&R_{12}=\sum_{1\leq i<j \leq l}\sum_{\substack{a\geq
0\\b_{L\setminus\{i,j\}}\geq
0}}\langle\tau_{a}\tau_{b_{L\setminus\{i,j\}}}\lambda_{g}
\rangle_{g}g\Psi_{b_{L\setminus\{i,j\}}}^{0}(t_{L\setminus\{i,j\}})\frac{-t_i\Psi_{a+1}^{0}(t_i)
+t_j\Psi_{a+1}^{0}(t_j)}{t_{i}-t_j}
\end{align*}
\begin{align*}
&R_{13}=-\sum_{1\leq i<j \leq l}\sum_{\substack{a\geq
0\\b_{L\setminus\{i,j\}}\geq
0}}\langle\tau_{a}\tau_{b_{L\setminus\{i,j\}}}\lambda_{g}\lambda_{1}
\rangle_{g}\Psi_{b_{L\setminus\{i,j\}}}^{0}(t_{L\setminus\{i,j\}})\frac{t_jt_i\Psi_{a+1}^{0}(t_i)
-t_it_j\Psi_{a+1}^{0}(t_j)}{t_{i}-t_j}\\
&R_{14}=-\sum_{1\leq i<j \leq l}\sum_{\substack{a\geq
0\\b_{L\setminus\{i,j\}}\geq
0}}\langle\tau_{a}\tau_{b_{L\setminus\{i,j\}}}\lambda_{g}\lambda_{1}
\rangle_{g}\Psi_{b_{L\setminus\{i,j\}}}^{0}(t_{L\setminus\{i,j\}})\frac{-t_i\Psi_{a+1}^{0}(t_i)
+t_j\Psi_{a+1}^{0}(t_j)}{t_{i}-t_j}
\end{align*}
\begin{align*}
&R_{15}=\sum_{1\leq i<j \leq l}\sum_{\substack{a\geq
0\\b_{L\setminus\{i,j\}}\geq
0}}\langle\tau_{a}\tau_{b_{L\setminus\{i,j\}}}\sum_{d=g+2}^{3g-3}P_{d}(\lambda)
\rangle_{g}\Psi_{b_{L\setminus\{i,j\}}}^{0}(t_{L\setminus\{i,j\}})\frac{(t_j-1)t_i\Psi_{a+1}^{0}(t_i)
-(t_i-1)t_j\Psi_{a+1}^{0}(t_j)}{t_{i}-t_j}
\end{align*}
\begin{align*}
R_{16}=\frac{1}{2}\sum_{j=1}^{l}\sum_{\substack{a_1\geq 0\\a_2\geq
0\\b_{L\setminus \{j\}}\geq 0}}\sum_{\substack{g_1+g_2=g\\
\mathcal{I}\cup\mathcal{J}=L\setminus\{j\}}}^{stable}\langle\tau_{a_1}\tau_{b_{\mathcal{I}}}\lambda_{g_1}
\rangle_{g_1}\langle\tau_{a_2}\tau_{b_{\mathcal{J}}}\lambda_{g_2}\rangle_{g_2}\Psi_{a_1+1}^{0}(t_j)\Psi_{a_2+1}^{0}(t_j)
\Psi_{b_{L\setminus\{j\}}}^{0}(t_{L\setminus\{j\}})
\end{align*}
Then,
\begin{align}
F_{2g+2l-4}(L_1)=\sum_{|b_{L}|=2g-3+l}\langle\tau_{b_L}\lambda_{g}\rangle_{g}
l(2-g-l)\sum_{j=1}^{l}f^{0}(b_j,b_j)t_{j}^{b_j}f^{0}(b_{L\setminus\{j\}},b_{L\setminus\{j\}}+1)
t_{L\setminus\{j\}}^{b_{L\setminus\{j\}}+1}
\end{align}
\begin{align}
&F_{2g+2l-4}(L_2)=\sum_{|b_{L}|=2g-3+l}\langle\tau_{b_L}\lambda_{g}\rangle_{g}
l\left(\sum_{j=1}^{l}f^{1}(b_j,b_j+1)t_j^{b_j+1}\sum_{k\in
L\setminus\{j\}}f^{0}(b_k,b_k)t_k^{b_k}\right.\\\nonumber
&\left.f^{0}(b_{L\setminus\{j,k\}},b_{L\setminus\{j,k\}}+1)t_{L\setminus\{j,k\}}^{b_{L\setminus\{j,k\}}+1}+\sum_{j=1}^{l}
f^{1}(b_j,b_j)t_j^{b_j}f^{0}(b_{L\setminus\{j\}},b_{L\setminus\{j\}}+1)t_{L\setminus\{j\}}^{b_{L\setminus\{j\}}+1}\right.\\\nonumber
&\left.+\sum_{j=1}^{l}f^{1}(b_j,b_j+2)t_j^{b_j+2}\sum_{k\in
L\setminus\{j\}}f^{0}(b_k,b_k-1)t_{k}^{b_k-1}f^{0}(b_{L\setminus\{j,k\}},b_{L\setminus\{j,k\}}+1)
t_{L\setminus\{j,k\}}^{b_{L\setminus\{j,k\}}+1}\right.\\\nonumber
&\left.+\sum_{j=1}^{l}f^{1}(b_j,b_j+2)t_j^{b_j+2}\sum_{k_1\in
L\setminus\{j\}}f^{0}(b_{k_1},b_{k_1})t_{k_1}^{b_{k_1}}\sum_{k_2\in
L\setminus\{j,k_1\}}f^{0}(b_{k_2},b_{k_2})t_{k_2}^{b_{k_2}}\right.\\\nonumber
&\left.f^{0}(b_{L\setminus
\{j,k_1,k_2\}},b_{L\setminus\{j,k_1,k_2\}}+1)t_{L\setminus\{j,k_1,k_2\}}^{b_{L\setminus\{j,k_1,k_2\}}+1}
\right)
\end{align}
\begin{align}
&F_{2g+2l-4}(L_3)=\sum_{|b_{L}|=2g-3+l}\langle\tau_{b_L}\lambda_{g}\rangle_{g}
\left(\sum_{j=1}^{l}b_jf^{0}(b_j,b_j)t_j^{b_j+1}\sum_{k\in
L\setminus\{j\}}f^{0}(b_k,b_k)t_k^{b_k}\right.\\\nonumber
&\left.f^{0}(b_{L\setminus\{j,k\}},b_{L\setminus\{j,k\}}+1)t_{L\setminus\{j,k\}}^{b_{L\setminus\{j,k\}}+1}+\sum_{j=1}^{l}
(b_j-1)f^{0}(b_j,b_j-1)t_j^{b_j}f^{0}(b_{L\setminus\{j\}},b_{L\setminus\{j\}}+1)t_{L\setminus\{j\}}^{b_{L\setminus\{j\}}+1}\right.\\\nonumber
&\left.+\sum_{j=1}^{l}(b_j+1)f^{0}(b_j,b_j+1)t_j^{b_j+2}\sum_{k\in
L\setminus\{j\}}f^{0}(b_k,b_k-1)t_{k}^{b_k-1}f^{0}(b_{L\setminus\{j,k\}},b_{L\setminus\{j,k\}}+1)
t_{L\setminus\{j,k\}}^{b_{L\setminus\{j,k\}}+1}\right.\\\nonumber
&\left.+\sum_{j=1}^{l}(b_j+1)f^{0}(b_j,b_j+1)t_j^{b_j+2}\sum_{k_1\in
L\setminus\{j\}}f^{0}(b_{k_1},b_{k_1})t_{k_1}^{b_{k_1}}\sum_{k_2\in
L\setminus\{j,k_1\}}f^{0}(b_{k_2},b_{k_2})t_{k_2}^{b_{k_2}}\right.\\\nonumber
&\left.f^{0}(b_{L\setminus
\{j,k_1,k_2\}},b_{L\setminus\{j,k_1,k_2\}}+1)t_{L\setminus\{j,k_1,k_2\}}^{b_{L\setminus\{j,k_1,k_2\}}+1}
\right)
\end{align}
\begin{align}
&F_{2g+2l-4}(L_4)=-\sum_{|b_{L}|=2g-3+l}\langle\tau_{b_L}\lambda_{g}\rangle_{g}
\left(\sum_{j=1}^{l}b_jf^{0}(b_j,b_j)t_j^{b_j}f^{0}(b_{L\setminus\{j\}},b_{L\setminus\{j\}}+1)t_{L\setminus\{j\}}^{b_{L\setminus
\{j\}}+1}\right.\\\nonumber
&\left.+\sum_{j=1}^{l}(b_j+1)f^{0}(b_j,b_j+1)t_j^{b_j+1}\sum_{k\in
L\setminus\{j\}}f^{0}(b_k,b_k)t_k^{b_k}f^{0}(b_{L\setminus\{j,k\}},b_{L\setminus\{j,k\}}+1)t_{L\setminus\{j,k\}}^{b_{L\setminus\{j,k\}}+1}
\right)
\end{align}
\begin{align}
&F_{2g+2l-4}(L_5)=-\sum_{|b_{L}|=2g-2+l}\langle\tau_{b_L}\lambda_{g-1}\rangle_{g}
l\left(\sum_{j=1}^{l}f^{0}(b_j,b_j-1)t_j^{b_j-1}f^{0}(b_{L\setminus\{j\}},b_{L\setminus\{j\}}+1)t_{L\setminus\{j\}}^{b_{L\setminus
\{j\}}+1}\right.\\\nonumber
&\left.+\sum_{j=1}^{l}f^{0}(b_j,b_j)t_j^{b_j}\sum_{k\in
L\setminus\{j\}}f^{0}(b_k,b_k)t_k^{b_k}f^{0}(b_{L\setminus\{j,k\}},b_{L\setminus\{j,k\}}+1)t_{L\setminus\{j,k\}}^{b_{L\setminus\{j,k\}}+1}
\right)
\end{align}
\begin{align}
F_{2g+2l-4}(L_6)=-\sum_{|b_{L}|=2g-4+l}\langle\tau_{b_L}\lambda_{g}\lambda_{1}\rangle_{g}
l f^{0}(b_L,b_L+1)t_{L}^{b_{L}+1}
\end{align}
\begin{align}
F_{2g+2l-4}(L_7)=0
\end{align}

\begin{align}
&F_{2g+2l-4}(R_1)=\sum_{1\leq i<j\leq
l}\sum_{a+|b_{L\setminus\{i,j\}}|=2g+l-4}\langle\tau_{a}\tau_{b_{L\setminus\{i,j\}}}\lambda_{g}
\rangle_{g}\left(f^{0}(b_{L\setminus
\{i,j\}},b_{L\setminus\{i,j\}}+1)t^{b_{L\setminus\{i,j\}}+1}_{L\setminus\{i,j\}}\right.\\\nonumber
&\left.f^{1}(a+1,a+1)\sum_{k\geq
0}^{a}t_{i}^{k+1}t_{j}^{a+1-k}+\sum_{r\in
L\setminus\{i,j\}}f^{0}(b_r,b_r)t_r^{b_r}f^{0}(b_{L\setminus\{i,j,r\}},b_{L\setminus\{i,j,r\}}+1)
t_{L\setminus\{i,j,r\}}^{b_{L\setminus\{i,j,r\}}+1}\right.\\\nonumber
&\left.f^{1}(a+1,a+2)\sum_{k\geq
0}^{a+1}t_i^{k+1}t_{j}^{a+2-k}+\sum_{r\in
L\setminus\{i,j\}}f^{0}(b_r,b_r-1)t_r^{b_r-1}f^{0}(b_{L\setminus\{i,j,r\}},b_{L\setminus\{i,j,r\}}+1)
t_{L\setminus\{i,j,r\}}^{b_{L\setminus\{i,j,r\}}+1}\right.\\\nonumber
&\left.f^{1}(a+1,a+3)\sum_{k\geq
0}^{a+2}t_i^{k+1}t_{j}^{a+3-k}+\sum_{r_1\in L\in
L\setminus\{i,j\}}f^{0}(b_{r_1},b_{r_1})t_{r_1}^{b_{r_1}}\sum_{r_2\in
L\setminus\{i,j,r_1\}}f^{0}(b_{r_2},b_{r_2})t_{r_2}^{b_{r_2}}\right.\\\nonumber
&\left.f^{0}(b_{L\setminus\{i,j,r_1,r_2\}},b_{L\setminus\{i,j,r_1,r_2\}}+1)t_{L\setminus\{i,j,r_1,r_2\}}^{b_{L\setminus\{i,j,r_1,r_2\}}+1}
f^{1}(a+1,a+3)\sum_{k\geq 0}^{a+2}t_i^{k+1}t_j^{a+3-k} \right)
\end{align}

\begin{align}
&F_{2g+2l-4}(R_2)=-\sum_{1\leq i<j\leq
l}\sum_{a+|b_{L\setminus\{i,j\}}|=2g+l-4}\langle\tau_{a}\tau_{b_{L\setminus\{i,j\}}}\lambda_{g}
\rangle_{g}\left(f^{0}(b_{L\setminus
\{i,j\}},b_{L\setminus\{i,j\}}+1)t^{b_{L\setminus\{i,j\}}+1}_{L\setminus\{i,j\}}\right.\\\nonumber
&\left.f^{1}(a+1,a+2)\sum_{k\geq
0}^{a+2}t_{i}^{k}t_{j}^{a+2-k}+\sum_{r\in
L\setminus\{i,j\}}f^{0}(b_r,b_r)t_r^{b_r}f^{0}(b_{L\setminus\{i,j,r\}},b_{L\setminus\{i,j,r\}}+1)
t_{L\setminus\{i,j,r\}}^{b_{L\setminus\{i,j,r\}}+1}\right.\\\nonumber
&\left.f^{1}(a+1,a+3)\sum_{k\geq 0}^{a+3}t_i^{k}t_{j}^{a+3-k}\right)
\end{align}

\begin{align}
&F_{2g+2l-4}(R_3)=\sum_{1\leq i<j\leq
l}\sum_{a+|b_{L\setminus\{i,j\}}|=2g+l-4}\langle\tau_{a}\tau_{b_{L\setminus\{i,j\}}}\lambda_{g}
\rangle_{g}\left(f^{0}(b_{L\setminus
\{i,j\}},b_{L\setminus\{i,j\}}+1)t^{b_{L\setminus\{i,j\}}+1}_{L\setminus\{i,j\}}\right.\\\nonumber
&\left.f^{0}(a+1,a)\sum_{k\geq
0}^{a}t_{i}^{k+1}t_{j}^{a+1-k}+\sum_{r\in
L\setminus\{i,j\}}f^{0}(b_r,b_r)t_r^{b_r}f^{0}(b_{L\setminus\{i,j,r\}},b_{L\setminus\{i,j,r\}}+1)
t_{L\setminus\{i,j,r\}}^{b_{L\setminus\{i,j,r\}}+1}\right.\\\nonumber
&\left.f^{0}(a+1,a+1)\sum_{k\geq
0}^{a+1}t_i^{k+1}t_{j}^{a+2-k}+\sum_{r\in
L\setminus\{i,j\}}f^{0}(b_r,b_r-1)t_r^{b_r-1}f^{0}(b_{L\setminus\{i,j,r\}},b_{L\setminus\{i,j,r\}}+1)
t_{L\setminus\{i,j,r\}}^{b_{L\setminus\{i,j,r\}}+1}\right.\\\nonumber
&\left.f^{0}(a+1,a+2)\sum_{k\geq
0}^{a+2}t_i^{k+1}t_{j}^{a+3-k}+\sum_{r_1\in L\in
L\setminus\{i,j\}}f^{0}(b_{r_1},b_{r_1})t_{r_1}^{b_{r_1}}\sum_{r_2\in
L\setminus\{i,j,r_1\}}f^{0}(b_{r_2},b_{r_2})t_{r_2}^{b_{r_2}}\right.\\\nonumber
&\left.f^{0}(b_{L\setminus\{i,j,r_1,r_2\}},b_{L\setminus\{i,j,r_1,r_2\}}+1)t_{L\setminus\{i,j,r_1,r_2\}}^{b_{L\setminus\{i,j,r_1,r_2\}}+1}
f^{0}(a+1,a+2)\sum_{k\geq 0}^{a+2}t_i^{k+1}t_j^{a+3-k} \right)
\end{align}

\begin{align}
&F_{2g+2l-4}(R_4)=-\sum_{1\leq i<j\leq
l}\sum_{a+|b_{L\setminus\{i,j\}}|=2g+l-4}\langle\tau_{a}\tau_{b_{L\setminus\{i,j\}}}\lambda_{g}
\rangle_{g}\left(f^{0}(b_{L\setminus
\{i,j\}},b_{L\setminus\{i,j\}}+1)t^{b_{L\setminus\{i,j\}}+1}_{L\setminus\{i,j\}}\right.\\\nonumber
&\left.f^{0}(a+1,a+1)\sum_{k\geq
0}^{a+2}t_{i}^{k}t_{j}^{a+2-k}+\sum_{r\in
L\setminus\{i,j\}}f^{0}(b_r,b_r)t_r^{b_r}f^{0}(b_{L\setminus\{i,j,r\}},b_{L\setminus\{i,j,r\}}+1)
t_{L\setminus\{i,j,r\}}^{b_{L\setminus\{i,j,r\}}+1}\right.\\\nonumber
&\left.f^{0}(a+1,a+2)\sum_{k\geq 0}^{a+3}t_i^{k}t_{j}^{a+3-k}\right)
\end{align}

\begin{align}
&F_{2g+2l-4}(R_5)=-\sum_{1\leq i<j\leq
l}\sum_{a+|b_{L\setminus\{i,j\}}|=2g+l-4}\langle\tau_{a}\tau_{b_{L\setminus\{i,j\}}}\lambda_{g}
\rangle_{g}\left((2g+l-1)f^{0}(b_{L\setminus
\{i,j\}},b_{L\setminus\{i,j\}}+1)t^{b_{L\setminus\{i,j\}}+1}_{L\setminus\{i,j\}}\right.\\\nonumber
&\left.f^{0}(a+1,a+1)\sum_{k\geq
0}^{a}t_{i}^{k+1}t_{j}^{a+1-k}+(2g+l-1)\sum_{r\in
L\setminus\{i,j\}}f^{0}(b_r,b_r)t_r^{b_r}f^{0}(b_{L\setminus\{i,j,r\}},b_{L\setminus\{i,j,r\}}+1)
t_{L\setminus\{i,j,r\}}^{b_{L\setminus\{i,j,r\}}+1}\right.\\\nonumber
&\left.f^{0}(a+1,a+2)\sum_{k\geq
0}^{a+1}t_i^{k+1}t_{j}^{a+2-k}\right)
\end{align}

\begin{align}
&F_{2g+2l-4}(R_6)=\sum_{1\leq i<j\leq
l}\sum_{a+|b_{L\setminus\{i,j\}}|=2g+l-4}\langle\tau_{a}\tau_{b_{L\setminus\{i,j\}}}\lambda_{g}
\rangle_{g}\left((2g+l-1)f^{0}(b_{L\setminus
\{i,j\}},b_{L\setminus\{i,j\}}+1)t^{b_{L\setminus\{i,j\}}+1}_{L\setminus\{i,j\}}\right.\\\nonumber
&\left.f^{0}(a+1,a+2)\sum_{k\geq
0}^{a+2}t_{i}^{k}t_{j}^{a+2-k}\right)
\end{align}

\begin{align}
&F_{2g+2l-4}(R_7)=\sum_{1\leq i<j\leq
l}\sum_{a+|b_{L\setminus\{i,j\}}|=2g+l-4}\langle\tau_{a}\tau_{b_{L\setminus\{i,j\}}}\lambda_{g}
\rangle_{g}\left(\sum_{r\in
L\setminus\{i,j\}}f^{1}(b_r,b_r+2)t_r^{b_r+2}\right.\\\nonumber
&\left.f^{0}(b_{L\setminus
\{i,j,r\}},b_{L\setminus\{i,j,r\}}+1)t^{b_{L\setminus\{i,j,r\}}+1}_{L\setminus\{i,j,r\}}f^{0}(a+1,a)\sum_{k\geq
0}^{a-1}t_{i}^{k+1}t_{j}^{a-k}+\sum_{r\in
L\setminus\{i,j\}}f^{1}(b_r,b_r+2)t_r^{b_r+2}\right.\\\nonumber
&\left.\sum_{s\in
L\setminus\{i,j,r\}}f^{0}(b_s,b_s-1)t_s^{b_s-1}f^{0}(b_{L\setminus\{i,j,r,s\}},b_{L\setminus\{i,j,r,s\}}+1)
t_{L\setminus\{i,j,r,s\}}^{b_{L\setminus\{i,j,r,s\}}+1}f^{0}(a+1,a+2)\sum_{k\geq
0}^{a+1}t_i^{k+1}t_{j}^{a+2-k}\right.\\\nonumber &\left.+\sum_{r\in
L\setminus\{i,j\}}f^{1}(b_r,b_r+2)t_r^{b_r+2}\sum_{s_1\in
L\setminus\{i,j,r\}}f^{0}(b_{s_1},b_{s_1})t_{s_1}^{b_{s_1}}\sum_{s_2\in
L\setminus\{i,j,r_1,s_1\}}f^{0}(b_{s_2},b_{s_2})t_{s_2}^{b_{s_2}}\right.\\\nonumber
&\left.f^{0}(b_{L\setminus\{i,j,r,s_1,s_2\}},b_{L\setminus\{i,j,r,s_1,s_2\}}+1)
t_{L\setminus\{i,j,r,s_1,s_2\}}^{b_{L\setminus\{i,j,r,s_1,s_2\}}+1}f^{0}(a+1,a+2)\sum_{k\geq
0}^{a+1}t_i^{k+1}t_{j}^{a+2-k}+\sum_{r\in
L\setminus\{i,j\}}f^{1}(b_r,b_r)t_r^{b_r}\right.\\\nonumber
&\left.f^{0}(b_{L\setminus\{i,j,r\}},b_{L\setminus\{i,j,r\}}+1)
t_{L\setminus\{i,j,r\}}^{b_{L\setminus\{i,j,r\}}+1}f^{0}(a+1,a+2)\sum_{k\geq
0}^{a+1}t_i^{k+1}t_{j}^{a+2-k}+\sum_{r_1\in
L\setminus\{i,j\}}f^{1}(b_{r_1},b_{r_1}+1)t_{r_1}^{b_{r_1}+1}\right.\\\nonumber
&\left.\sum_{r_2\in
L\setminus\{i,j,r_1\}}f^{0}(b_{r_2},b_{r_2})t_{r_2}^{b_{r_2}}f^{0}(b_{L\setminus\{i,j,r_1,r_2\}},b_{L\setminus\{i,j,r_1,r_2\}}+1)
t_{L\setminus\{i,j,r_1,r_2\}}^{b_{L\setminus\{i,j,r_1,r_2\}}+1}f^{0}(a+1,a+2)\right.\\\nonumber
&\left.\sum_{k\geq 0}^{a+1}t_{i}^{k+1}t_{j}^{a+2-k}+\sum_{r\in
\{i,j\}}f^{1}(b_r,b_r+1)t_r^{b_r+1}f^{0}(b_{L\setminus\{i,j,r\}},b_{L\setminus\{i,j,r\}}+1)
t_{L\setminus\{i,j,r\}}^{b_{L\setminus\{i,j,r\}}+1}f^{0}(a+1,a+1)\right.\\\nonumber
&\left.\sum_{k\geq 0}^{a}t_{i}^{k+1}t_{j}^{a+1-k}+\sum_{r_1\in
L\setminus\{i,j\}}f^{1}(b_{r_1},b_{r_1}+2)t_{r_1}^{b_{r_1}+2}\sum_{r_2\in
L\setminus\{i,j,r_1\}}f^{0}(b_{r_2},b_{r_2})t_{r_2}^{b_{r_2}}\right.\\\nonumber
&\left.f^{0}(b_{L\setminus\{i,j,r_1,r_2\}},b_{L\setminus\{i,j,r_1,r_2\}}+1)t_{L\setminus\{i,j,r_1,r_2\}}^{b_{L\setminus\{i,j,r_1,r_2\}}+1}
f^{0}(a+1,a+1)\sum_{k\geq 0}^{a}t_{i}^{k+1}t_{j}^{a+1-k}\right)
\end{align}

\begin{align}
&F_{2g+2l-4}(R_8)=-\sum_{1\leq i<j\leq
l}\sum_{a+|b_{L\setminus\{i,j\}}|=2g+l-4}\langle\tau_{a}\tau_{b_{L\setminus\{i,j\}}}\lambda_{g}
\rangle_{g}\left(\sum_{r\in
L\setminus\{i,j\}}f^{1}(b_r,b_r+2)t_{r}^{b_r+2}\right.\\\nonumber
&\left.f^{0}(b_{L\setminus
\{i,j,r\}},b_{L\setminus\{i,j,r\}}+1)t^{b_{L\setminus\{i,j,r\}}+1}_{L\setminus\{i,j,r\}}f^{0}(a+1,a+1)\sum_{k\geq
0}^{a+1}t_{i}^{k}t_{j}^{a+1-k}+\sum_{r\in
L\setminus\{i,j\}}f^{1}(b_r,b_r+1)t_r^{b_r+1}\right.\\\nonumber
&\left.f^{0}(b_{L\setminus\{i,j,r\}},b_{L\setminus\{i,j,r\}}+1)
t_{L\setminus\{i,j,r\}}^{b_{L\setminus\{i,j,r\}}+1}f^{0}(a+1,a+2)\sum_{k\geq
0}^{a+2}t_i^{k}t_{j}^{a+2-k}+\sum_{r\in
L\setminus\{i,j\}}f^{1}(b_r,b_r+2)t_r^{b_r+2}\right.\\\nonumber
&\left.\sum_{s\in
L\setminus\{i,j,r\}}f^{0}(b_{s},b_s)t_s^{b_s}f^{0}(b_{L\setminus\{i,j,r,s\}},b_{L\setminus\{i,j,r,s\}}+1)
t_{L\setminus\{i,j,r,s\}}^{b_{L\setminus\{i,j,r,s\}}+1}f^{0}(a+1,a+2)\sum_{k\geq
0}^{a+2}t_{i}^{k}t_{j}^{a+2-k}\right)
\end{align}

\begin{align}
&F_{2g+2l-4}(R_9)=-\sum_{1\leq i<j\leq
l}\sum_{a+|b_{L\setminus\{i,j\}}|=2g+l-3}\langle\tau_{a}\tau_{b_{L\setminus\{i,j\}}}\lambda_{g-1}
\rangle_{g}\left(f^{0}(b_{L\setminus
\{i,j\}},b_{L\setminus\{i,j\}}+1)t^{b_{L\setminus\{i,j\}}+1}_{L\setminus\{i,j\}}\right.\\\nonumber
&\left.f^{0}(a+1,a)\sum_{k\geq
0}^{a-1}t_{i}^{k+1}t_{j}^{a-k}+\sum_{r\in
L\setminus\{i,j\}}f^{0}(b_r,b_r)t_r^{b_r}f^{0}(b_{L\setminus\{i,j,r\}},b_{L\setminus\{i,j,r\}}+1)
t_{L\setminus\{i,j,r\}}^{b_{L\setminus\{i,j,r\}}+1}\right.\\\nonumber
&\left.f^{0}(a+1,a+1)\sum_{k\geq
0}^{a}t_i^{k+1}t_{j}^{a+1-k}+\sum_{r\in
L\setminus\{i,j\}}f^{0}(b_r,b_r-1)t_r^{b_r-1}f^{0}(b_{L\setminus\{i,j,r\}},b_{L\setminus\{i,j,r\}}+1)
t_{L\setminus\{i,j,r\}}^{b_{L\setminus\{i,j,r\}}+1}\right.\\\nonumber
&\left.f^{0}(a+1,a+2)\sum_{k\geq
0}^{a+1}t_i^{k+1}t_{j}^{a+2-k}+\sum_{r_1\in L\in
L\setminus\{i,j\}}f^{0}(b_{r_1},b_{r_1})t_{r_1}^{b_{r_1}}\sum_{r_2\in
L\setminus\{i,j,r_1\}}f^{0}(b_{r_2},b_{r_2})t_{r_2}^{b_{r_2}}\right.\\\nonumber
&\left.f^{0}(b_{L\setminus\{i,j,r_1,r_2\}},b_{L\setminus\{i,j,r_1,r_2\}}+1)t_{L\setminus\{i,j,r_1,r_2\}}^{b_{L\setminus\{i,j,r_1,r_2\}}+1}
f^{0}(a+1,a+2)\sum_{k\geq 0}^{a+1}t_i^{k+1}t_j^{a+2-k} \right)
\end{align}

\begin{align}
&F_{2g+2l-4}(R_{10})=\sum_{1\leq i<j\leq
l}\sum_{a+|b_{L\setminus\{i,j\}}|=2g+l-3}\langle\tau_{a}\tau_{b_{L\setminus\{i,j\}}}\lambda_{g-1}
\rangle_{g}\left(f^{0}(b_{L\setminus
\{i,j\}},b_{L\setminus\{i,j\}}+1)t^{b_{L\setminus\{i,j\}}+1}_{L\setminus\{i,j\}}\right.\\\nonumber
&\left.f^{0}(a+1,a+1)\sum_{k\geq
0}^{a+1}t_{i}^{k}t_{j}^{a+1-k}+\sum_{r\in
L\setminus\{i,j\}}f^{0}(b_r,b_r)t_r^{b_r}f^{0}(b_{L\setminus\{i,j,r\}},b_{L\setminus\{i,j,r\}}+1)
t_{L\setminus\{i,j,r\}}^{b_{L\setminus\{i,j,r\}}+1}\right.\\\nonumber
&\left.f^{0}(a+1,a+2)\sum_{k\geq 0}^{a+2}t_i^{k}t_{j}^{a+2-k}\right)
\end{align}

\begin{align}
&F_{2g+2l-4}(R_{11})=\sum_{1\leq i<j\leq
l}\sum_{a+|b_{L\setminus\{i,j\}}|=2g+l-4}\langle\tau_{a}\tau_{b_{L\setminus\{i,j\}}}\lambda_{g}
\rangle_{g}g\left(f^{0}(b_{L\setminus
\{i,j\}},b_{L\setminus\{i,j\}}+1)t^{b_{L\setminus\{i,j\}}+1}_{L\setminus\{i,j\}}\right.\\\nonumber
&\left.f^{0}(a+1,a+1)\sum_{k\geq
0}^{a}t_{i}^{k+1}t_{j}^{a+1-k}+\sum_{r\in
L\setminus\{i,j\}}f^{0}(b_r,b_r)t_r^{b_r}f^{0}(b_{L\setminus\{i,j,r\}},b_{L\setminus\{i,j,r\}}+1)
t_{L\setminus\{i,j,r\}}^{b_{L\setminus\{i,j,r\}}+1}\right.\\\nonumber
&\left.f^{0}(a+1,a+2)\sum_{k\geq 0}^{a+1}t_i^{k}t_{j}^{a+2-k}\right)
\end{align}

\begin{align}
&F_{2g+2l-4}(R_{12})=-\sum_{1\leq i<j\leq
l}\sum_{a+|b_{L\setminus\{i,j\}}|=2g+l-4}\langle\tau_{a}\tau_{b_{L\setminus\{i,j\}}}\lambda_{g}
\rangle_{g}g\left(f^{0}(b_{L\setminus
\{i,j\}},b_{L\setminus\{i,j\}}+1)t^{b_{L\setminus\{i,j\}}+1}_{L\setminus\{i,j\}}\right.\\\nonumber
&\left.f^{0}(a+1,a+2)\sum_{k\geq
0}^{a+2}t_{i}^{k}t_{j}^{a+2-k}\right)
\end{align}

\begin{align}
&F_{2g+2l-4}(R_{13})=-\sum_{1\leq i<j\leq
l}\sum_{a+|b_{L\setminus\{i,j\}}|=2g+l-5}\langle\tau_{a}\tau_{b_{L\setminus\{i,j\}}}\lambda_{g}\lambda_{1}
\rangle_{g}\left(f^{0}(b_{L\setminus
\{i,j\}},b_{L\setminus\{i,j\}}+1)t^{b_{L\setminus\{i,j\}}+1}_{L\setminus\{i,j\}}\right.\\\nonumber
&\left.f^{0}(a+1,a+2)\sum_{k\geq
0}^{a+1}t_{i}^{k+1}t_{j}^{a+2-k}\right)
\end{align}

\begin{align}
F_{2g+2l-4}(R_{14})=F_{2g+2l-4}(R_{15})=0
\end{align}

\begin{align}
&F_{2g+2l-4}(R_{16})=\frac{1}{2}\sum_{j=1}^{l}\sum_{\substack{g_1+g_2=g\\
\mathcal{I}\cup\mathcal{J}=L\setminus\{j\}}}^{stable}\sum_{\substack{a_1+|b_{\mathcal{I}}|=2g_1-2+|\mathcal{I}|\\
a_2+|b_{\mathcal{J}}|=2g_2-2+|\mathcal{J}|}}\langle\tau_{a_1}\tau_{b_{\mathcal{I}}}\lambda_{g_1}
\rangle_{g_1}\langle\tau_{a_2}\tau_{b_{\mathcal{J}}}\lambda_{g_2}\rangle_{g_2}\\\nonumber
&f^{0}(a_1+1,a_1+2)f^{0}(a_2+1,a_2+2)
t_{j}^{a_1+a_2+4}f^{0}(b_{L\setminus\{j\}},b_{L\setminus\{j\}}+1)t_{L\setminus\{j\}}^{b_{L\setminus\{j\}}+1}
\end{align}

\begin{align}
[t_{L}^{b_{L}+1}]F_{2g+2l-4}(L_{1})=\sum_{j=1}^{l}\langle\tau_{b_{j}+1}\tau_{b_{L\setminus\{j\}}}\lambda_{g}
\rangle_{g}l(2-g-l)f^{0}(b_{j}+1,b_{j}+1)f^{0}(b_{L\setminus\{j\}},b_{L\setminus\{j\}}+1)
\end{align}

\begin{align}
&[t_{L}^{b_{L}+1}]F_{2g+2l-4}(L_2)=\sum_{j=1}^{l}\sum_{k\in
L\setminus\{j\}}\langle\tau_{b_j}\tau_{b_{k}+1}\tau_{b_{L\setminus\{j,k\}}}\lambda_{g}
\rangle_{g}lf^{1}(b_{j},b_{j}+1)f^{0}(b_{k}+1,b_{k}+1)\\\nonumber
&f^{0}(b_{L\setminus\{j,k\}},b_{L\setminus\{j,k\}}+1)+\sum_{j=1}^{l}\langle\tau_{b_j+1}\tau_{b_{L\setminus\{j\}}}\lambda_{g}
\rangle_{g}lf^{1}(b_{j}-1,b_{j}+1)f^{0}(b_{L\setminus\{j\}},b_{L\setminus\{j\}}+1)\\\nonumber
&+\sum_{j=1}^{l}\sum_{k\in
L\setminus\{j\}}\langle\tau_{b_j-1}\tau_{b_k+2}\tau_{b_{L\setminus\{j,k\}}}\lambda_{g}
\rangle_{g}lf^{1}(b_{j}-1,b_{j}+1)f^{0}(b_k+2,b_k+1)f^{0}(b_{L\setminus\{j,k\}},b_{L\setminus\{j,k\}}+1)\\\nonumber
&+\sum_{j=1}^{l}\sum_{k_1\in L\setminus\{j\}}\sum_{k_2\in
L\setminus\{j,k_1\}}\langle\tau_{b_j-1}\tau_{b_{k_1}+1}\tau_{b_{k_2}+1}\tau_{b_{L\setminus\{j,k_1,k_2\}}}\lambda_{g}
\rangle_{g}lf^{1}(b_{j}-1,b_{j}+1)\\\nonumber
&f^{0}(b_{k_1}+1,b_{k_{1}}+1)f^{0}(b_{k_2}+1,b_{k_2}+1)
f^{0}(b_{L\setminus\{j,k_1,k_2\}},b_{L\setminus\{j,k_1,k_2\}}+1)
\end{align}

\begin{align}
&[t_{L}^{b_{L}+1}]F_{2g+2l-4}(L_3)=\sum_{j=1}^{l}\sum_{k\in
L\setminus\{j\}}\langle\tau_{b_j}\tau_{b_{k}+1}\tau_{b_{L\setminus\{j,k\}}}\lambda_{g}
\rangle_{g}b_jf^{0}(b_{j},b_{j})f^{0}(b_{k}+1,b_{k}+1)\\\nonumber
&f^{0}(b_{L\setminus\{j,k\}},b_{L\setminus\{j,k\}}+1)+\sum_{j=1}^{l}\langle\tau_{b_j+1}\tau_{b_{L\setminus\{j\}}}\lambda_{g}
\rangle_{g}b_jf^{0}(b_{j}+1,b_{j})f^{0}(b_{L\setminus\{j\}},b_{L\setminus\{j\}}+1)\\\nonumber
&+\sum_{j=1}^{l}\sum_{k\in
L\setminus\{j\}}\langle\tau_{b_j-1}\tau_{b_k+2}\tau_{b_{L\setminus\{j,k\}}}\lambda_{g}
\rangle_{g}b_jf^{0}(b_{j}-1,b_{j})f^{0}(b_k+2,b_k+1)f^{0}(b_{L\setminus\{j,k\}},b_{L\setminus\{j,k\}}+1)\\\nonumber
&+\sum_{j=1}^{l}\sum_{k_1\in L\setminus\{j\}}\sum_{k_2\in
L\setminus\{j,k_1\}}\langle\tau_{b_j-1}\tau_{b_{k_1}+1}\tau_{b_{k_2}+1}\tau_{b_{L\setminus\{j,k_1,k_2\}}}\lambda_{g}
\rangle_{g}b_jf^{0}(b_{j}-1,b_{j})\\\nonumber
&f^{0}(b_{k_1}+1,b_{k_{1}}+1)f^{0}(b_{k_2}+1,b_{k_2}+1)
f^{0}(b_{L\setminus\{j,k_1,k_2\}},b_{L\setminus\{j,k_1,k_2\}}+1)
\end{align}

\begin{align}
&[t_{L}^{b_{L}+1}]F_{2g+2l-4}(L_4)=-\sum_{j=1}^{l}\langle\tau_{b_j+1}\tau_{b_{L\setminus\{j\}}}\lambda_{g}
\rangle_{g}(b_j+1)f^{0}(b_j+1,b_j+1)f^{0}(b_{L\setminus\{j\}},b_{L\setminus\{j\}}+1)\\\nonumber
&-\sum_{j=1}^{l}\sum_{k\in
L\setminus\{j\}}\langle\tau_{b_j}\tau_{b_k+1}\tau_{b_{L\setminus\{j,k\}}}\lambda_{g}
\rangle_{g}(b_j+1)f^{0}(b_j,b_j+1)f^{0}(b_k+1,b_k+1)f^{0}(b_{L\setminus\{j,k\}},b_{L\setminus\{j,k\}}+1)
\end{align}

\begin{align}
&[t_{L}^{b_{L}+1}]F_{2g+2l-4}(L_5)=-\sum_{j=1}^{l}\langle\tau_{b_j+2}\tau_{b_{L\setminus\{j\}}}\lambda_{g-1}
\rangle_{g}lf^{0}(b_j+2,b_j+1)f^{0}(b_{L\setminus\{j\}},b_{L\setminus\{j\}}+1)\\\nonumber
&-\sum_{j=1}^{l}\sum_{k\in
L\setminus\{j\}}\langle\tau_{b_j+1}\tau_{b_k+1}\tau_{b_{L\setminus\{j,k\}}}\lambda_{g-1}
\rangle_{g}lf^{0}(b_j+1,b_j+1)f^{0}(b_k+1,b_k+1)f^{0}(b_{L\setminus\{j,k\}},b_{L\setminus\{j,k\}}+1)
\end{align}

\begin{align}
[t_{L}^{b_{L}+1}]F_{2g+2l-4}(L_6)=-\langle\tau_{b_L}\lambda_{g}\lambda_{1}
\rangle_{g}lf^{0}(b_L,b_{L}+1)
\end{align}

on the right hand side,

\begin{align}
&[t_{L}^{b_{L}+1}]F_{2g+2l-4}(R_1)=\sum_{1\leq i<j\leq
l}\langle\tau_{b_i+b_j}\tau_{b_{L\setminus\{i,j\}}}\lambda_{g}
\rangle_{g}f^{0}(b_{L\setminus\{i,j\}},b_{L\setminus\{i,j\}}+1)f^{1}(b_i+b_j+1,b_i+b_j+1)\\\nonumber
&+\sum_{1\leq i<j\leq l}\sum_{r\in
L\setminus\{i,j\}}\langle\tau_{b_i+b_j-1}\tau_{b_r+1}\tau_{b_{L\setminus\{i,j,r\}}}\lambda_{g}
\rangle_{g}f^{1}(b_i+b_j,b_i+b_j+1)f^{0}(b_{L\setminus\{i,j,r\}},b_{L\setminus\{i,j,r\}}+1)\\\nonumber
&f^{0}(b_r+1,b_r+1)+\sum_{1\leq i<j\leq l}\sum_{r\in
L\setminus\{i,j\}}\langle\tau_{b_i+b_j-2}\tau_{b_r+2}\tau_{b_{L\setminus\{i,j,r\}}}\lambda_{g}
\rangle_{g}f^{1}(b_i+b_j-1,b_i+b_j+1)f^{0}(b_r+2,b_r+1)\\\nonumber
&f^{0}(b_{L\setminus\{i,j,r\}},b_{L\setminus\{i,j,r\}}+1)+\sum_{1\leq
i<j\leq l}\sum_{r_1\in L\setminus\{i,j\}}\sum_{r_2\in
L\setminus\{i,j,r_1\}}\langle\tau_{b_i+b_j-2}\tau_{b_{r_1}+1}\tau_{b_{r_2}+1}\tau_{b_{L\setminus\{i,j,r_1,r_2\}}}
\lambda_{g}\rangle_{g}\\\nonumber
&f^{0}(b_{r_1}+1,b_{r_1}+1)f^{0}(b_{r_2}+1,b_{r_2}+1)
f^{0}(b_{L\setminus\{i,j,r_1,r_2\}},b_{L\setminus\{i,j,r_1,r_2\}}+1)f^{1}(b_i+b_j-1,b_i+b_j+1)
\end{align}

\begin{align}
&[t_{L}^{b_{L}+1}]F_{2g+2l-4}(R_2)=-\sum_{1\leq i<j\leq
l}\langle\tau_{b_i+b_j}\tau_{b_{L\setminus\{i,j\}}}\lambda_{g}
\rangle_{g}f^{0}(b_{L\setminus\{i,j\}},b_{L\setminus\{i,j\}}+1)f^{1}(b_i+b_j+1,b_i+b_j+2)\\\nonumber
&-\sum_{1\leq i<j\leq l}\sum_{r\in
L\setminus\{i,j\}}\langle\tau_{b_i+b_j-1}\tau_{b_r+1}\tau_{b_{L\setminus\{i,j,r\}}}\lambda_{g}
\rangle_{g}f^{0}(b_r+1,b_r+1)f^{0}(b_{L\setminus\{i,j,r\}},b_{L\setminus\{i,j,r\}}+1)\\\nonumber
&f^{1}(b_i+b_j+1,b_i+b_j+3)
\end{align}

\begin{align}
&[t_{L}^{b_{L}+1}]F_{2g+2l-4}(R_3)=\sum_{1\leq i<j\leq
l}\langle\tau_{b_i+b_j}\tau_{b_{L\setminus\{i,j\}}}\lambda_{g}
\rangle_{g}f^{0}(b_{L\setminus\{i,j\}},b_{L\setminus\{i,j\}}+1)f^{0}(b_i+b_j+1,b_i+b_j)\\\nonumber
&+\sum_{1\leq i<j\leq l}\sum_{r\in
L\setminus\{i,j\}}\langle\tau_{b_i+b_j-1}\tau_{b_r+1}\tau_{b_{L\setminus\{i,j,r\}}}\lambda_{g}
\rangle_{g}f^{0}(b_i+b_j,b_i+b_j)f^{0}(b_{L\setminus\{i,j,r\}},b_{L\setminus\{i,j,r\}}+1)\\\nonumber
&f^{0}(b_r+1,b_r+1)+\sum_{1\leq i<j\leq l}\sum_{r\in
L\setminus\{i,j\}}\langle\tau_{b_i+b_j-2}\tau_{b_r+2}\tau_{b_{L\setminus\{i,j,r\}}}\lambda_{g}
\rangle_{g}f^{0}(b_i+b_j-1,b_i+b_j)f^{0}(b_r+2,b_r+1)\\\nonumber
&f^{0}(b_{L\setminus\{i,j,r\}},b_{L\setminus\{i,j,r\}}+1)+\sum_{1\leq
i<j\leq l}\sum_{r_1\in L\setminus\{i,j\}}\sum_{r_2\in
L\setminus\{i,j,r_1\}}\langle\tau_{b_i+b_j-2}\tau_{b_{r_1}+1}\tau_{b_{r_2}+1}\tau_{b_{L\setminus\{i,j,r_1,r_2\}}}
\lambda_{g}\rangle_{g}\\\nonumber
&f^{0}(b_{r_1}+1,b_{r_1}+1)f^{0}(b_{r_2}+1,b_{r_2}+1)
f^{0}(b_{L\setminus\{i,j,r_1,r_2\}},b_{L\setminus\{i,j,r_1,r_2\}}+1)f^{0}(b_i+b_j-1,b_i+b_j)
\end{align}

\begin{align}
&[t_{L}^{b_{L}+1}](F_{2g+2l-4}(R_4)+F_{2g+2l-4}(R_5))=-\sum_{1\leq
i<j\leq
l}\langle\tau_{b_i+b_j}\tau_{b_{L\setminus\{i,j\}}}\lambda_{g}
\rangle_{g}(2g+l)\\\nonumber
&f^{0}(b_{L\setminus\{i,j\}},b_{L\setminus\{i,j\}}+1)f^{0}(b_i+b_j+1,b_i+b_j+1)-\sum_{1\leq
i<j\leq l}\sum_{r\in
L\setminus\{i,j\}}\langle\tau_{b_i+b_j-1}\tau_{b_r+1}\tau_{b_{L\setminus\{i,j,r\}}}\lambda_{g}
\rangle_{g}\\\nonumber
&(2g+l)f^{0}(b_r+1,b_r+1)f^{0}(b_{L\setminus\{i,j,r\}},b_{L\setminus\{i,j,r\}}+1)f^{0}(b_i+b_j,b_i+b_j+1)
\end{align}

\begin{align}
&[t_{L}^{b_{L}+1}]F_{2g+2l-4}(R_6)=\sum_{1\leq i<j\leq
l}\langle\tau_{b_i+b_j}\tau_{b_{L\setminus\{i,j\}}}\lambda_{g}
\rangle_{g}(2g+l-1)f^{0}(b_{L\setminus\{i,j\}},b_{L\setminus\{i,j\}}+1)\\\nonumber
&f^{0}(b_i+b_j+1,b_i+b_j+2)
\end{align}

\begin{align}
&[t_{L}^{b_{L}+1}]F_{2g+2l-4}(R_7)=\sum_{1\leq i<j\leq l}\sum_{r\in
L\setminus\{i,j\}}\langle\tau_{b_i+b_j+1}\tau_{b_r-1}\tau_{b_{L\setminus\{i,j,r\}}}\lambda_{g}
\rangle_{g}f^{1}(b_r-1,b_r+1)f^{0}(b_{L\setminus\{i,j,r\}},b_{L\setminus\{i,j,r\}}+1)\\\nonumber
&f^{0}(b_i+b_j+2,b_i+b_j+1)+\sum_{1\leq i<j\leq l}\sum_{r\in
L\setminus\{i,j\}}\sum_{s\in
L\setminus\{i,j,r\}}\langle\tau_{b_i+b_j-1}\tau_{b_r-1}\tau_{b_s+2}\tau_{b_{L\setminus\{i,j,r,s\}}}\lambda_{g}\rangle_{g}\\\nonumber
&f^{1}(b_r-1,b_r+1)f^{0}(b_s+2,b_s+1)f^{0}(b_{L\setminus\{i,j,r,s\}},b_{L\setminus\{i,j,r,s\}}+1)
f^{0}(b_i+b_j,b_i+b_j+1)\\\nonumber
&+\sum_{1\leq i<j\leq
l}\sum_{r\in L\setminus\{i,j\}}\sum_{s_1\in
L\setminus\{i,j,r\}}\sum_{s_2\in
L\setminus\{i,j,r,s_1\}}\langle\tau_{b_i+b_j-1}\tau_{b_r-1}\tau_{b_{s_1}+1}\tau_{b_{s_2}+1}
\tau_{b_{L\setminus\{i,j,r,s_1,s_2\}}}\lambda_{g}\rangle_{g}f^{1}(b_r-1,b_r+1)\\\nonumber
&f^{0}(b_{s_1}+1,b_{s_1}+1)f^{0}(b_{s_2}+1,b_{s_2}+1)f^{0}(b_{L\setminus\{i,j,r,s_1,s_2\}},b_{L\setminus\{i,j,r,s_1,s_2\}}
+1)f^{0}(b_i+b_j,b_i+b_j+1)\\\nonumber
&+\sum_{1\leq i<j\leq
l}\sum_{r\in
L\setminus\{i,j\}}\langle\tau_{b_i+b_j-1}\tau_{b_r+1}\tau_{b_{L\setminus\{i,j,r\}}}\lambda_{g}
\rangle_{g}f^{1}(b_r+1,b_r+1)f^{0}(b_{L\setminus\{i,j,r\}},b_{L\setminus\{i,j,r\}}+1)\\\nonumber
&f^{0}(b_i+b_j,b_i+b_j+1)+\sum_{1\leq i<j\leq l}\sum_{r\in
L\setminus\{i,j\}}\sum_{s\in
L\setminus\{i,j,r\}}\langle\tau_{b_i+b_j-1}\tau_{b_r}\tau_{b_s+1}\tau_{b_{L\setminus\{i,j,r,s\}}}\lambda_{g}\rangle_{g}f^{1}(b_r,b_r+1)\\\nonumber
&f^{0}(b_s+1,b_s+1)f^{0}(b_{L\setminus\{i,j,r,s\}},b_{L\setminus\{i,j,r,s\}}+1)
f^{0}(b_i+b_j,b_i+b_j+1)+\sum_{1\leq i<j\leq l}\sum_{r\in
L\setminus\{i,j\}}\langle\tau_{b_i+b_j}\tau_{b_r}\tau_{b_{L\setminus\{i,j,r\}}}\lambda_{g}
\rangle_{g}\\\nonumber
&f^{1}(b_r,b_r+1)f^{0}(b_{L\setminus\{i,j,r\}},b_{L\setminus\{i,j,r\}}+1)f^{0}(b_i+b_j+1,b_i+b_j+1)\\\nonumber
&+\sum_{1\leq i<j\leq l}\sum_{r\in L\setminus\{i,j\}}\sum_{s\in
L\setminus\{i,j,r\}}\langle\tau_{b_i+b_j}\tau_{b_r-1}\tau_{b_s+1}\tau_{b_{L\setminus\{i,j,r,s\}}}
\lambda_{g}\rangle_{g}f^{1}(b_r-1,b_r+1)f^{0}(b_s+1,b_s+1)\\\nonumber
&f^{0}(b_{L\setminus\{i,j,r,s\}},b_{L\setminus\{i,j,r,s\}}+1)
f^{0}(b_i+b_j+1,b_i+b_j+1)
\end{align}

\begin{align}
&[t_{L}^{b_{L}+1}]F_{2g+2l-4}(R_8)=-\sum_{1\leq i<j\leq l}\sum_{r\in
L\setminus\{i,j\}}\langle\tau_{b_i+b_j+1}\tau_{b_r-1}\tau_{b_{L\setminus\{i,j,r\}}}\lambda_{g}
\rangle_{g}f^{1}(b_r-1,b_r+1)\\\nonumber
&f^{0}(b_{L\setminus\{i,j,r\}},b_{L\setminus\{i,j,r\}}+1)f^{0}(b_i+b_j+2,b_i+b_j+2)-\sum_{1\leq
i<j\leq l}\sum_{r\in
L\setminus\{i,j\}}\langle\tau_{b_i+b_j}\tau_{b_r}\tau_{b_{L\setminus\{i,j,r\}}}\lambda_{g}
\rangle_{g}\\\nonumber
&f^{1}(b_r,b_r+1)f^{0}(b_{L\setminus\{i,j,r\}},b_{L\setminus\{i,j,r\}}+1)f^{0}(b_i+b_j+1,b_i+b_j+2)\\\nonumber
&-\sum_{1\leq i<j\leq l}\sum_{r\in L\setminus\{i,j\}}\sum_{s\in
L\setminus\{i,j,r\}}\langle\tau_{b_i+b_j}\tau_{b_r-1}\tau_{b_s+1}\tau_{b_{L\setminus\{i,j,r,s\}}}\lambda_{g}
\rangle_{g}f^{1}(b_r-1,b_r+1)f^{0}(b_s+1,b_s+1)\\\nonumber
&f^{0}(b_{L\setminus\{i,j,r,s\}},b_{L\setminus\{i,j,r,s\}}+1)f^{0}(b_i+b_j+1,b_i+b_j+2)
\end{align}

\begin{align}
&[t_{L}^{b_{L}+1}]F_{2g+2l-4}(R_9)=-\sum_{1\leq i<j\leq
l}\langle\tau_{b_i+b_j+1}\tau_{b_{L\setminus\{i,j\}}}\lambda_{g-1}
\rangle_{g}f^{0}(b_{L\setminus\{i,j\}},b_{L\setminus\{i,j\}}+1)
f^{0}(b_i+b_j+2,b_i+b_j+1)\\\nonumber &-\sum_{1\leq i<j\leq
l}\sum_{r\in
L\setminus\{i,j\}}\langle\tau_{b_i+b_j}\tau_{b_r+1}\tau_{b_{L\setminus\{i,j,r\}}}\lambda_{g-1}
\rangle_{g}f^{0}(b_r+1,b_r+1)f^{0}(b_{L\setminus\{i,j,r\}},b_{L\setminus\{i,j,r\}}+1)\\\nonumber
&f^{0}(b_i+b_j+1,b_i+b_j+1)-\sum_{1\leq i<j\leq l}\sum_{r\in
L\setminus\{i,j\}}\sum_{s\in
L\setminus\{i,j,r\}}\langle\tau_{b_i+b_j-1}\tau_{b_r+1}\tau_{b_s+1}\tau_{b_{L\setminus\{i,j,r,s\}}}\lambda_{g-1}
\rangle_{g}\\\nonumber
&f^{0}(b_r+1,b_r+1)f^{0}(b_s+1,b_s+1)f^{0}(b_{L\setminus\{i,j,r,s\}},b_{L\setminus\{i,j,r,s\}}+1)f^{0}(b_i+b_j,b_i+b_j+1)\\\nonumber
&-\sum_{1\leq i<j\leq l}\sum_{r\in
L\setminus\{i,j\}}\langle\tau_{b_i+b_j-1}\tau_{b_r+2}\tau_{b_{L\setminus\{i,j,r\}}}\lambda_{g-1}
\rangle_{g}f^{0}(b_r+2,b_r+1)f^{0}(b_{L\setminus\{i,j,r\}},b_{L\setminus\{i,j,r\}}+1)\\\nonumber
&f^{0}(b_i+b_j,b_i+b_j+1)
\end{align}

\begin{align}
&[t_{L}^{b_{L}+1}]F_{2g+2l-4}(R_{10})=\sum_{1\leq i<j\leq
l}\langle\tau_{b_i+b_j+1}\tau_{b_{L\setminus\{i,j\}}}\lambda_{g-1}
\rangle_{g}f^{0}(b_{L\setminus\{i,j\}},b_{L\setminus\{i,j\}}+1)f^{0}(b_i+b_j+2,b_i+b_j+2)\\\nonumber
&+\sum_{1\leq i<j\leq l}\sum_{r\in
L\setminus\{i,j\}}\langle\tau_{b_i+b_j}\tau_{b_r+1}\tau_{b_{L\setminus\{i,j,r\}}}\lambda_{g-1}
\rangle_{g}f^{0}(b_r+1,b_r+1)f^{0}(b_{L\setminus\{i,j,r\}},b_{L\setminus\{i,j,r\}}+1)\\\nonumber
&f^{0}(b_i+b_j+1,b_i+b_j+2)
\end{align}

\begin{align}
&[t_{L}^{b_{L}+1}]F_{2g+2l-4}(R_{11})=\sum_{1\leq i<j\leq
l}\langle\tau_{b_i+b_j}\tau_{b_{L\setminus\{i,j\}}}\lambda_{g}
\rangle_{g}gf^{0}(b_{L\setminus\{i,j\}},b_{L\setminus\{i,j\}}+1)f^{0}(b_i+b_j+1,b_i+b_j+1)\\\nonumber
&+\sum_{1\leq i<j\leq l}\sum_{r\in
L\setminus\{i,j\}}\langle\tau_{b_i+b_j-1}\tau_{b_r+1}\tau_{b_{L\setminus\{i,j,r\}}}\lambda_{g}
\rangle_{g}gf^{0}(b_r+1,b_r+1)f^{0}(b_{L\setminus\{i,j,r\}},b_{L\setminus\{i,j,r\}}+1)\\\nonumber
&f^{0}(b_i+b_j,b_i+b_j+1)
\end{align}

\begin{align}
&[t_{L}^{b_{L}+1}]F_{2g+2l-4}(R_{12})=-\sum_{1\leq i<j\leq
l}\langle\tau_{b_i+b_j}\tau_{b_{L\setminus\{i,j\}}}\lambda_{g}
\rangle_{g}gf^{0}(b_{L\setminus\{i,j\}},b_{L\setminus\{i,j\}}+1)f^{0}(b_i+b_j+1,b_i+b_j+2)
\end{align}

\begin{align}
&[t_{L}^{b_{L}+1}]F_{2g+2l-4}(R_{13})=-\sum_{1\leq i<j\leq
l}\langle\tau_{b_i+b_j-1}\tau_{b_{L\setminus\{i,j\}}}\lambda_{g}\lambda_{1}
\rangle_{g}f^{0}(b_{L\setminus\{i,j\}},b_{L\setminus\{i,j\}}+1)f^{0}(b_i+b_j,b_i+b_j+1)
\end{align}

\begin{align}
&[t_L^{b_L+1}]F_{2g+2l-4}(R_{16})=\frac{1}{2}\sum_{j=1}^{l}\sum_{a_1+a_2=b_j-3}\sum_{\substack{g_1+g_2=g\\
\mathcal{I}\cup\mathcal{J}=L\setminus\{j\}}}^{stable}\langle\tau_{a_1}\tau_{b_{\mathcal{I}}}\lambda_{g_1}
\rangle_{g_1}\langle\tau_{a_2}\tau_{b_{\mathcal{J}}}\lambda_{g_2}\rangle_{g_2}\\\nonumber
&f^{0}(a_1+1,a_1+2)f^{0}(a_2+1,a_2+2)f^{0}(b_{L\setminus\{j\}},b_{L\setminus\{j\}}+1)
\end{align}

Therefore, from the equation $(97)$, we get the identity
\begin{align}
&(120)+(121)+(122)+(123)+(124)+(125)=(126)+(127)\\\nonumber
&+(128)+(129)+(130)+(131)+(132)+(133)+(134)+(135)+136)+(137)+(138)
\end{align}

We note that, three types of the Hodge integrals appear in identity
$(139)$, $\langle\tau_{b_L}\lambda_{g}\rangle_{g}$,
$\langle\tau_{b_L}\lambda_{g-1}\rangle_{g}$ and
$\langle\tau_{b_L}\lambda_{g}\lambda_{1}\rangle_{g}$.

By now, we have known the closed formula for $\lambda_{g}$-integral
\cite{FabP} and a recursion formula for $\lambda_{g-1}$-integral
\cite{Zhu}.
\begin{equation}
\int_{\overline{\sM}_{g,l}}\psi_1^{d_1}\dots\psi_l^{d_l}\lambda_g=\binom{2g-3+l}{b_{1},..,b_{l}}c_{g},
\end{equation}
where $\sum_{i=1}^{l}b_{i}=2g-3+l, b_{1},..,b_{l}\geq 0$ and $c_{g}$
is a constant that depends only on $g$.

\begin{align}
\langle\tau_{b_{1}}\cdots\tau_{b_{l}}\lambda_{g-1}\rangle_{g}=\frac{1}{l}\sum_{1\leq
i<j\leq l}\langle\tau_{b_{i}+b_{j}-1}\prod_{k\neq
i,j}\tau_{b_{k}}\lambda_{g-1}\rangle_{g}\frac{(b_{i}+b_{j})!}{b_{i}!b_{j}!}+C_{g,l}(b_{1},..,b_{l})
\end{align}
where $C_{g,n}(b_{1},..,b_{l})$ is a constant related to
$b_{1},..,b_{l}$ and $g,l$.

On the other hand, the constants $f^{0}(b,k)$ and $f^{1}(b,k)$ are
available by their definition and the recursions $(52)$, $(53)$. For
example, $f^{0}(b,b+1)=b!$,
$f^{1}(b,b+2)=(b+1)!\sum_{k=2}^{b+1}\frac{1}{k}$. Thus, solving the
equation $(139)$, we will get a formula for the computation of Hodge
integral $\langle\tau_{b_L}\lambda_{g}\lambda_{1}\rangle_{g}$.

We observe that only the terms $(125)$ and $(137)$ contain the Hodge
integral of type
$\langle\tau_{b_L}\lambda_{g}\lambda_{1}\rangle_{g}$. Moving and
combining the corresponding terms, we have
\begin{align}
&\langle\tau_{b_{L}}\lambda_{g}\lambda_{1}\rangle_{g}
=\frac{1}{l}\sum_{1\leq i<j\leq
l}\langle\tau_{b_{i}+b_{j}-1}\tau_{b_{L\setminus\{i,j\}}}\lambda_{g}\lambda_{1}\rangle_{g}
\frac{(b_{i}+b_{j})!}{b_{i}!b_{j}!}+C(g,l,b_1,..,b_l)
\end{align}
Where
\begin{align}
&C(g,l,b_1,..,b_l)=-\frac{1}{l\prod_{i=1}^{l}b_{i}!}\left((126)+(127)+(128)+(129)+(130)+(131)+(132)\right.\\\nonumber
&\left.+(133)+(134)+(135)+(136)+(138)-(120)-(121)-(122)-(123)-(124)
\right)
\end{align}

Formula $(143)$ contains many terms, we hope to make a computer
program to calculate $C(g,l,b_1,..,b_l)$.

$$ \ \ \ \ $$

\end{document}